\documentclass[12pt]{amsart}
\usepackage{a4wide,amssymb,amsthm,amscd,amsmath,mathtools,stmaryrd,url,fullpage,enumerate,mathabx}
\usepackage[initials,alphabetic]{amsrefs}
\usepackage[shortlabels]{enumitem}
\usepackage{bm}
\usepackage{tikz-cd}
\usepackage{stackengine}
\usepackage{xcolor}
\usepackage{hyperref}
\definecolor{couleur_cite}{rgb}{0.05,.4,0.05}
\definecolor{couleur_link}{rgb}{0.05,0.05,0.4}
\hypersetup{bookmarksopen,bookmarksnumbered,colorlinks,linkcolor=couleur_link,citecolor=couleur_cite}

\stackMath
\usepackage[active]{srcltx} 
\usepackage{mathrsfs}
\usepackage{soul}

\usepackage[all,cmtip]{xy}

\numberwithin{equation}{section}

\newcommand{\bT}{\mathbf{T}}
\newcommand{\bPB}{\mathbf{PB}}
\newcommand{\bG}{\mathbf{G}}
\newcommand{\bB}{\mathbf{B}}

\newcommand{\SO}{\mathrm{SO}}

\newcommand{\C}{\mathbb{C}}
\newcommand{\G}{\mathbf{G}}
\newcommand{\GL}{\mathrm{GL}}

\newcommand{\SL}{\mathrm{SL}}
\newcommand{\N}{\mathbb{N}}
\newcommand{\R}{\mathbb{R}}
\newcommand{\Z}{\mathbb{Z}}
\newcommand{\Q}{\mathbb{Q}}
\newcommand{\A}{\mathbb{A}}

\newcommand{\ClD}{\mathrm{Cl}_D}

\newtheorem{theorem}{Theorem}[section]
\newtheorem{lemma}{Lemma}[section]
\newtheorem{proposition}{Proposition}[section]

\newtheorem{definition}[theorem]{Definition}
\newtheorem{corollary}[theorem]{Corollary}

\theoremstyle{remark}
\newtheorem{remark}[theorem]{Remark}
\newtheorem{remarks}[theorem]{Remarks}

\begin{document}

\author{Valentin Blomer}
\author{Farrell Brumley}
\author{Maksym Radziwi\l\l}
 
\address{Mathematisches Institut, Endenicher Allee 60, 53115 Bonn, Germany}
\email{blomer@math.uni-bonn.de}

\address{Sorbonne Université, Université Paris Cité, CNRS, IMJ-PRG, F-75005 Paris, France}
\email{brumley@imj-prg.fr}

\address{Courant Institute of Mathematical Sciences, 251 Mercer Street, NY 10012, USA}
\email{maksym.radziwill@gmail.com}
  
\title{Joint Linnik problems}

\thanks{The first author is supported by DFG through SFB-TRR 358 and EXC-2047/1 - 390685813 and by ERC
Advanced Grant 101054336. The second author is supported by the Institut Universitaire de France and ANR-FNS Grant ANR-24-CE93-0016. The third author is supported by NSF grant NSF DMS-2401106}

\begin{abstract} 
We prove a conjecture of Michel--Venkatesh on joinings
of distinct Linnik problems, in the setting of simultaneous quaternionic
embeddings of imaginary quadratic fields having sufficiently many small
split primes. This splitting condition is known to hold for all
but $O((\log\log X)^{1 + o(1)})$ discriminants up to $X$. We also
treat a non-equivariant form of this conjecture proposed by
Aka--Einsiedler--Shapira, which in particular applies to the
classical Gau{\ss} construction joining Linnik points on the sphere with
CM points on the modular surface.
\end{abstract}

\subjclass[2010]{Primary: 11F67, 11F70, 11M41}
\keywords{Toric periods, simultaneous equidistribution, $L$-functions, Siegel zeros}

\setcounter{tocdepth}{1} 
\maketitle

\section{Introduction}

Problems of Linnik type, such as the equidistribution of points in $\Z^3$ of large norm projected to the unit sphere and the equidistribution of CM points of large discriminant on the modular curve, have long been a fertile ground for the development of tools in ergodic theory and analytic number theory, initiated by Linnik \cite{Li} and culminating in their resolution by Duke \cite{Du}.

A new generation of equidistribution problems has been put forward by Michel and Venkatesh \cite{MV} and concerns simultaneous versions, or \textit{joinings}, of Duke's theorem. In \cite{BB}, the first two authors proved partial results on joinings of two distinct Linnik problems, using techniques emanating from automorphic forms and analytic number theory, in particular, the high moment method of Soundararajan \cite{So} for central values of $L$-functions. These results were conditional on two major assumptions: the generalized Riemann hypothesis (GRH), applied to automorphic $L$-functions up to degree 8, and the cuspidality of testing functions.

In this paper, we reconsider this problem from a new viewpoint, which favors mollification, period formulae, and spectral theory. These techniques allow us to substantially relax the main hypotheses of \cite{BB}  on simultaneous equidistribution, to the point of removing any automorphic assumptions other than the quantitative splitting behavior of small primes. Most notably, we replace GRH with a much weaker condition roughly equivalent to the absence of Siegel zeros for quadratic Dirichlet characters. Moreover, we succeed in treating the continuous spectrum, thereby eliminating the cuspidality restriction in \cite{BB} when one of the two factors is non-compact. This applies, for example, to the orthogonal complement procedure of Gau{\ss}, in which one joins integer points on the sphere with CM points on the modular surface, as investigated by Aka--Einsiedler--Shapira \cite{AES}. 

Our main result is Theorem \ref{thm11}, which we summarize informally as follows; complete details and an effective version will be given in Section \ref{second-intro}.

\begin{theorem}\label{thm:vague}
The simultaneous equidistribution conjecture of Michel--Venkatesh, as well as a quadratic variant due to Aka--Einsiedler--Shapira, hold for quaternionic varieties over $\Q$ at almost maximal level, for discriminants $-D$ such that the quadratic Dirichlet $L$-function $L(s, \chi_{-D})$ has no zeros in $|s-1| \leq \psi(D)/\log D$ for some $\psi(D) \rightarrow \infty$.  
\end{theorem}

By a density theorem \cite{Ga}, applied to intervals of the form $[Y, Y^2)$, Theorem \ref{thm:vague} holds for all but at most  
$O((\log\log X)^{1 + o(1)})$ 
discriminants up to $X$.

We think of $\psi$ as growing arbitrarily slowly, so that the hypothesis on $L(s, \chi_{-D})$ can be seen as an $o(1)$-strengthening of a no Siegel zero condition. We view this condition as an analytic formulation of the abundance of small split primes in quadratic number fields. As such, it captures, in a quantitative way, the types of splitting conditions present in ergodic approaches to this problem, which we discuss below. It is interesting to note that ours is the same type of zero-free region under which Linnik claimed (without proof) to remove his congruence condition for the equidistribution of primitive integer points on spheres \cite[Theorem 2]{Li}.

Several works have been devoted to proving versions of the simultaneous equidistribution conjecture of Michel--Venkatesh and its variants, subject to a pair of splitting conditions, using methods from homogeneous dynamics \cites{AES,Kh,ALMW}. The key ingredient is a powerful joinings theorem of Einsiedler and Lindenstrauss \cite{EL}, which classifies measures invariant under the action of a higher rank diagonalisable subgroup. When applied to tori associated with imaginary quadratic field extensions $\Q(\sqrt{-d})$, this requires fixing two distinct primes $p_1, p_2$ and asking that they split in $\Q(\sqrt{-d})$. More generally, it is possible to fix a finite set of primes $\{p_1, \ldots , p_k\}$ and restrict to squarefree $d$ for which there exists $1 \leq i < j \leq k$ such that $p_i$ and $p_j$ split in $\Q(\sqrt{-d})$. Despite this flexibility, regardless of how one chooses the set of primes $\{p_1, \ldots , p_k\}$, the above construction can never produce an exceptional set of cardinality smaller than $o(X)$. 

Finally, we remark that Theorem \ref{thm:vague}, when applied to two distinct Shimura varieties, implies a strong form of the Andr\'e--Oort conjecture, first proved under generalized Riemann hypothesis for $L(s,\chi_{-D})$ by Yafaev \cite{Ya}, but with Zariski density replaced by equidistribution.
\section{Examples, historical context, and the main theorem}\label{second-intro}

In this section we give a precise formulation of our main result, stated as Theorem \ref{thm11}, and 
explain the new ideas and methods. 
Before doing so, we present two concrete instances (Theorems \ref{thm-AES-intro} and \ref{thm-sphere-ellipsoid}) which highlight many of the new features of our work. 

\subsection{Orthogonal complement of Gau{\ss}}

For a positive squarefree integer $d$ let ${\rm R}(d)=\{x\in\Z^3: \|x\|^2=d\}$ denote the set of integral representations of $d$ as the sum of three squares. The classical theorem of Legendre and Gau{\ss} affirms that ${\rm R}(d)$ is non-empty precisely when $d$ is locally admissible, meaning that $d$ avoids the obvious congruence obstruction of $7\bmod 8$. On the other hand, for a negative fundamental discriminant $-D$ consider the set $\mathcal{Q}_{-D}$ of $\SL_2(\Z)$-equivalence classes of primitive binary quadratic forms of discriminant $-D$. 

In \cite[Art.\ 291]{Gau}, Gau{\ss} discovered a beautiful link between these two settings. For $x \in {\rm R}(d)$ he considers the rank-2 orthogonal lattice $\Lambda_x=x^{\perp} \cap \Z^3$. Fix the standard orientation of $\Z^3$ and give $\Lambda_x$ the induced orientation under wedge product with $x$. The quadratic form $\phi_x$ obtained by restriction of $(\|x\|^2,\Z^3)$ to $\Lambda_x$ (and rescaled by $1/2$ if the form is even) is of discriminant $-D={\rm disc}(\Q(\sqrt{-d}))$. Concretely, if $(B, C)$ is a positively oriented $\Z$-basis of $\Lambda_x$, one has $\phi_x(t,u)=\|B t+C u\|^2$ or $\frac12\|B t+C u\|^2$ according to whether $d\equiv 1,2\bmod 4$ or $d\equiv 3\bmod 8$. Let $[\phi_x]\in\mathcal{Q}_{-D}$ denote the $\SL_2(\Z)$-equivalence class of $\phi_x$. Gau{\ss}'s construction descends to equivalence classes $\mathcal{R}(d):=\SO_3(\Z)\backslash {\rm R}(d)$, yielding a map
\[
{\rm Orth}(d): \mathcal{R}(d)\rightarrow \mathcal{Q}_{-D}, \quad [x]\mapsto [\phi_x],
\] 
whose image lies in a single genus (depending on $d$).
\subsubsection*{Numerical example}

In \cite[Art.\ 292]{Gau} Gau{\ss} worked out this correspondence numerically in the case of $d=770=2\cdot 5\cdot 7\cdot 11$. On one hand we have $|\mathcal{R}(770)|=16$, and a complete list of elements up to $\SO_3(\Z)$-equivalence can be given by
\[
(15,16,17),\; (1,12, 25),\;  (8,9,25),\; (4,15, 23),\; (3,19,20),\;  (4,5,27),\;  (9,17, 20),\;  (5,13,24)
\]
together with their negatives. On the other hand, ${\rm disc}(\Q(\sqrt{-770}))=-3080$ and Gau{\ss} calculates that $|\mathcal{Q}_{-3080}|=32$. There are $8=2^{4-1}$ genera, each of size $4$. Gau{\ss} shows that the image of ${\rm Orth}(770)$ lies in the genus consisting of the proper equivalence classes of
\[
\phi_\pm = 6t^2\pm 4tu+129u^2,\quad \phi'_\pm =19t^2\pm 6tu+41u^2.
\]
For example, if $x=\pm (15,16,17)$ then $\Lambda_x=\{y\in\Z^3: 15y_1+16y_2+17y_3=0\}$ admits $(B,\pm C)=( (1,-2,1),\pm (8,1,-8))$ as an oriented $\Z$-basis and $\|Bt\pm Cu\|^2=\phi_\mp (t,u)$.

\bigskip

We now turn our attention to the asymptotic distribution of $\mathcal{R}(d)$ and $\mathcal{Q}_{-D}$, both separately and jointly.

For $d$ locally admissible, the combined work of Gau{\ss}, Dirichlet, and Siegel shows that $|\mathcal{R}(d)|=d^{1/2+o(1)}$. In the mid 20th century, Linnik \cite{Li} investigated the angular distribution of such solutions using his ergodic method. Let $S^2$ be the unit sphere in $\R^3$,  equipped with the uniform probability measure $\mu_{S^2}$. Let $\mathcal{S}^2=\SO_3(\Z)\backslash S^2$ and consider
\[
{\rm Arg}:\mathcal{R}(d)\rightarrow\mathcal{S}^2,\quad [x]\mapsto {\rm Arg}_{[x]}=\SO_3(\Z).(d^{-1/2}x).
\]
Linnik proved that, for any fixed auxiliary prime $p>2$, the set $\{{\rm Arg}_{[x]} :\, [x] \in \mathcal{R}(d)\}$ equidistributes on $\mathcal{S}^2$ relative to the quotient measure $\mu_{\mathcal{S}^2}$ as $d \rightarrow \infty$ along locally admissible squarefree integers for which $p>2$ splits in $\Q(\sqrt{-d})$. Duke \cite{Du} and Golubeva--Fomenko \cite{GF} subsequently removed Linnik's splitting condition, and provided a power-savings convergence rate, using automorphic methods and a breakthrough of Iwaniec \cite{Iw} on bounds for Fourier coefficients of half-integral holomorphic modular forms.

Similarly, for a negative fundamental discriminant $-D$, one can consider the asymptotic distribution of the set $\mathscr{H}_D=\{\tau_{[\phi]}: [\phi]\in\mathcal{Q}_{-D}\}$ of CM points of discriminant $-D$, where $\tau_{[\phi]}\in Y_0(1)$ is the $\SL_2(\Z)$-orbit of the unique root of $\phi(z,1)$ lying in the upper half-plane. Linnik's student Skubenko \cite{Sku} showed that $\mathscr{H}_D$ equidistributes on $Y_0(1)$ relative to the normalized hyperbolic measure $\mu_{Y_0(1)}$, provided one restricts to those $D$ for which a fixed auxiliary prime $p$ splits in $\Q(\sqrt{-D})$. By extending Iwaniec's bounds to non-holomorphic half-integral weight forms, Duke \cite{Du} removed Skubenko's auxiliary congruence condition, and, as above, provided a power-savings rate of convergence. Duke's result in fact applies to the slightly smaller set of CM points restricted to a fixed genus.

\medskip

The \textit{graph} of the Gau{\ss} orthogonal complement ${\rm Orth}(d)$ can be realized as
\begin{equation}\label{Gamma-d}
\Gamma_d=\{({\rm Arg}_{[x]}, \tau_{[\phi_x]}) : [x] \in \mathcal{R}(d)\}\subset \mathcal{S}^2\times Y_0(1).
\end{equation}
Since the projection of $\Gamma_d$ onto each factor equidistributes as $d\rightarrow\infty$, it is natural to ask whether $\Gamma_d$ becomes equidistributed in the product space. In \cite[Conj.\ 1.1]{AES} this was conjectured to be the case, under no conditions on $d\rightarrow\infty$ other than local representability. Our first main theorem yields progress towards this conjecture.
 
\begin{theorem}\label{thm-AES-intro}
The set $\Gamma_d$ equidistributes in $\mathcal{S}^2\times Y_0(1)$ according to the product measure $\mu_{\mathcal{S}^2}\otimes\mu_{Y_0(1)}$, as $d\rightarrow\infty$ along squarefree locally admissible integers $d$ for which $L(s, \chi_{-D})$ 
 has no zeros in $|s-1|\leq \psi(D)/\log D$ for some $\psi(D)\rightarrow\infty$.
\end{theorem}


\begin{remark}\label{rem22} The orthogonal complement construction applies, in fact, to any integral ternary quadratic form \cite[Art.\ 280]{Gau}. Our methods allow us to prove similar theorems for quadratic forms arising from the reduced norm of a quaternion algebra over $\Q$, when restricted to the trace-zero elements of a maximal order. We discuss the definite and indefinite cases below.

In the definite case, we may take as an example the form
$f(x)=2x_1^2 + 5x_2^2 + 10x_3^2$ and consider, for a positive squarefree integer $d$, the set ${\rm R}_f(d) := \{x \in \Z^3 : f(x)= d\}$. The set ${\rm R}_f(d)$ is non-empty precisely when $d$ avoids the local congruence obstructions of $3\bmod 8$ and $1,4\bmod 5$. Let $\mathcal{R}_f(d):=\SO_f(\Z)\backslash {\rm R}_f(d)$. The same procedure as before defines a map
\[
{\rm Orth}_f(d) :\mathcal{R}_f(d)\rightarrow\mathcal{Q}_{-D},\; [x]\mapsto [\phi_x],
\]
with image a single genus (depending on $d$). Define the ellipsoid $E^2 := \{x \in\R^3 : f(x)= 1\}$ and its quotient $\mathcal{E}^2=\SO_f(\Z)\backslash E^2$. Let $\mu_{\mathcal{E}^2}$ be the measure on $\mathcal{E}^2$ obtained as the quotient of the unique probability normalized $\SO_f(\R)$-invariant Borel measure on $E^2$. Just as in the case of the sphere, we have a norm one projection map
\[
{\rm Arg}: \mathcal{R}_f(d)\rightarrow\mathcal{E}^2,\quad [x]\mapsto {\rm Arg}_{[x]}=\SO_f(\Z).(d^{-1/2}x).
\]
In \cite{DSP} it was shown that $\{{\rm Arg}_{[x]}: [x]\in\mathcal{R}_f(d)\}$ equidistributes on $\mathcal{E}^2$ relative to $\mu_{\mathcal{E}^2}$ as $d\rightarrow\infty$ along locally admissible squarefree integers. With $\Gamma_d$ defined similarly to \eqref{Gamma-d}, the analogue of Theorem \ref{thm-AES-intro} with $\mathcal{S}^2$ replaced by $\mathcal{E}^2$ holds and follows from our Theorem \ref{thm11}.

In the indefinite case, the corresponding map ${\rm Orth}_f(d): \mathcal{R}_f(d)\rightarrow\mathcal{Q}_{-D}$ is defined on equivalence classes of integer points on \textit{hyperboloids}. Depending on the sign of the integer $d$ being represented, the image $[\phi_x]\in\mathcal{Q}_{-D}$ will correspond to either a CM point or a closed geodesic in $Y_0(1)$. Our Theorem \ref{thm11} covers this setting, as long as $f$ is anisotropic. The isotropic determinant form $x_2^2-x_1x_3$ is outside the scope of that theorem, however, as it requires different techniques to sample from two incomplete Eisenstein series. We plan to address this question in a future paper, using ideas from multiplicative Gau{\ss}ian chaos and the work of Adam Harper.
\end{remark}
\subsection{Simultaneous equidistribution}
We may reframe the preceding discussion in terms of quaternion algebras, optimal embeddings, and class group actions. This more structural framework will allow us to state our main results giving concrete instances of the Michel--Venkatesh simultaneous equidistribution conjecture \cite[Section 6.4.1]{MV}.

We focus on the example of Theorem \ref{thm-AES-intro}. Let $\bB^{(2,\infty)}$ denote the unique (up to isomorphism) quaternion algebra over $\Q$ which is ramified at $2$ and $\infty$, and $\mathbb{O}_{\rm Hurw}$ the maximal order determined by the Hurwitz quaternions in $\bB^{(2,\infty)}(\Q)$. In this case, the projective unit group $\mathbb{O}_{\rm Hurw}^\times/\{\pm 1\}$ may be identified with the unique index $2$ subgroup $\SO_3^+(\Z)$ of $\SO_3(\Z)$. It is more convenient to work with the finer equivalence relation on ${\rm R}(d)$ modulo $\SO_3^+(\Z)$. Accordingly we write $\widetilde{\mathcal{R}}(d)=\SO_3^+(\Z)\backslash {\rm R}(d)$. Then every $x \in {\rm R}(d)$ defines an optimal embedding of the ring of integers $\mathscr{O}_{\Q(\sqrt{-d})}$ into $\mathbb{O}_{\rm Hurw}$ and the equivalence classes of these embeddings realize $\widetilde{\mathcal{R}}(d)$ as a disjoint union of either one or two principal homogeneous spaces under the action of the class group $\ClD$, according to whether $d\equiv 1,2\bmod 4$ or $d\equiv 3\bmod 8$. Similarly, if ${\rm Mat}_2$ denotes the split quaternion algebra over $\Q$, then any binary quadratic form $\phi$ of discriminant $-D$ determines an optimal embedding of $\mathscr{O}_{\Q(\sqrt{-d})}$ into the maximal order ${\rm Mat}_2(\Z)$. Equivalence classes of such embeddings realize $\mathcal{Q}_{-D}$ as a principal homogeneous space under the action of $\ClD$; see \cite{EMV} for more details. 

The map $\widetilde{\mathcal{R}}(d)\rightarrow\mathcal{Q}_{-D}$, obtained from ${\rm Orth}(d)$ by factorisation through the projection map $\widetilde{\mathcal{R}}(d)\rightarrow\mathcal{R}(d)$, is not equivariant with respect to the respective class group actions. We illustrate this in the case of $d\equiv 1,2\bmod 4$, where $\widetilde{\mathcal{R}}(d)$ is a single orbit under $\ClD$. Choose a basepoint $[x]\in \widetilde{\mathcal{R}}(d)$. Then, for any $[\mathfrak{a}]\in \ClD$ we have $[\phi_{[\mathfrak{a}].[x]}]=[\mathfrak{a}]^2.[\phi_x]$; see \cite[Section 4.2]{EMV}. Thus, for $d\equiv 1,2\bmod 4$ we have the parametrization
\begin{equation}\label{Gammad}
\Gamma_d=\left\{({\rm Arg}_{[\mathfrak{a}].[x]}, \tau_{[\mathfrak{a}]^2.[\phi_x]}) : [\mathfrak{a}] \in \ClD\right\}.
\end{equation}
Thus $\ClD$ acts ``with different speeds'' on the two components. The same is true for $d\equiv 3\bmod 8$, for each of the two class group orbits.

This interpretation in terms of class group actions allows us to formulate new, related problems, that lie outside of the scope of the Gau{\ss} orthogonal complement procedure.

\medskip

With this in mind, we consider a second concrete example involving two \textit{definite} quaternion algebras with an equivariant class group action. Besides $\textbf{B}_1 = \textbf{B}^{(2, \infty)}$ considered earlier, we now let $\textbf{B}_2 = \textbf{B}^{(5, \infty)}$ denote the unique quaternion algebra over $\Q$ ramified at $5$ and $\infty$. Both $\bB_1$ and $\bB_2$ are of class number one. Let $\mathbb{O}_1\subset\bB_1(\Q)$ be the Hurwitz quaternions and choose a maximal order $\mathbb{O}_2\subset\bB_2(\Q)$. When restricted to the trace zero elements of $\mathbb{O}_1, \mathbb{O}_2$ the two reduced norm forms are given by $f_1=(x_1^2 + x_2^2 + x_3^2,\Z^3)$ and $f_2=(2x_1^2 + 5x_2^2 + 10x_3^2,\Z^3)$, respectively. The class group $\ClD$ acts freely on $\widetilde{\mathcal{R}}_{f_i}(d)$ ($i=1,2)$. The orbit structure of $\widetilde{\mathcal{R}}_{f_1}(d)$ has already been discussed. For $\widetilde{\mathcal{R}}_{f_2}(d)$, there are 1, 2, or 4 orbits according to explicit congruence conditions on $d\bmod 40$ \cite[Theorem 4.2]{Sh}. 

\begin{theorem}\label{thm-sphere-ellipsoid}
Let $\mathbb{D}$ denote the set of positive squarefree numbers avoiding $3\bmod 4$ and $1,4\bmod 5$. For $d\in\mathbb{D}$ fix base points $[x_i] \in \widetilde{\mathcal{R}}_{f_i}(d)$ $(i=1,2)$.
The set
\[
\big\{({\rm Arg}_{[\mathfrak{a}].[x_1]}, {\rm Arg}_{[\mathfrak{a}].[x_2]}) : [\mathfrak{a}] \in \ClD \big\}
\]
equidistributes in $\mathcal{S}^2 \times \mathcal{E}^2$ according to the product measure $\mu_{\mathcal{S}^2}\otimes\mu_{\mathcal{E}^2}$, as $d\rightarrow\infty$ along $d \in \mathbb{D}$  for which $L(s,\chi_{-D})$ has no zeros in $|s-1|\leq \psi(D)/\log D$ for some $\psi(D)\rightarrow\infty$. 
\end{theorem}

\begin{remark}\label{rem:variants} The settings of Theorems \ref{thm-AES-intro} and \ref{thm-sphere-ellipsoid} can be mixed. For instance, just as in \eqref{Gammad}, we can allow in Theorem \ref{thm-sphere-ellipsoid} that the class group acts on one component with ``double speed'' and prove an analogous equidistribution statement for 
\[
\big\{({\rm Arg}_{[\mathfrak{a}].[x_1]}, {\rm Arg}_{[\mathfrak{a}]^2.[x_2]}) : [\mathfrak{a}] \in \ClD \big\}.
\]
In fact, in this case we can also allow self-products (where one takes the same ternary quadratic form) inside $\mathcal{S}^2\times\mathcal{S}^2$ or $\mathcal{E}^2\times\mathcal{E}^2$, since the square on the second factor guarantees escape from the diagonal. 

Conversely, we can also obtain an equidistribution statement in a set-up where the Gau{\ss} orthogonal complement construction \eqref{Gammad} is replaced with its equivariant sibling
\[
\big\{({\rm Arg}_{[\mathfrak{a}].[x]} , \tau_{[\mathfrak{a}].[\phi]} ) : [\mathfrak{a}] \in \ClD\big\} \subseteq \mathcal{S}^2 \times Y_0(1)
\]
for base points $[x]\in\widetilde{\mathcal{R}}(d)$ and $[\phi]\in \mathcal{Q}_{-D}$. 
\end{remark}

\subsection{General context}\label{secgeneral}
 
We now turn to the general setting of joinings of two distinct Linnik problems and state our main results in full generality. We will adopt the language of algebraic groups, which has the advantage of covering all cases in a unified fashion and can be specialized to reproduce classical situations of particular arithmetic interest. Several concrete examples in classical language are given in \cite[Section 3]{BB}. In addition to results of the type presented in Theorems \ref{thm-AES-intro} and \ref{thm-sphere-ellipsoid}, they also include the equidistribution of simultaneous supersingular reduction of CM elliptic curves, as in \cite{ALMW}.
 
For $i=1,2$, let $\bB_i$ be a quaternion algebra over $\Q$ and write $\bG_i=\bPB_i^\times$. Let ${\rm Ram}(\bB_i)$ denote the set of primes at which $\bB_i$ is ramified. Let $\mathbb{O}_i$ be a maximal order in $\bB_i(\Q)$. Write $\bG_i(\Z_p)$ for the image of $(\mathbb{O}_i\otimes \Z_p)^\times$ inside $\bG_i(\Q_p)$. Then $\bG_i(\Z_p)$ is a maximal compact subgroup of $\bG_i(\Q_p)$ whenever $p\notin {\rm Ram}(\bB_i)$ and an index $2$ subgroup of the compact group $\bG_i(\Q_p)$ whenever $p\in {\rm Ram}(\bB_i)$. We put $\bG_i(\widehat{\Z})=\prod_p \bG_i(\Z_p)$ and let $K_{\infty,i}\subset \bG_i(\R)$ be a maximal compact torus. We may then define the quaternionic variety $Y_i=\bG_i(\Q)\backslash\bG_i(\A)/K_i$, where $K_i=\bG_i(\widehat{\Z})K_{\infty,i}$. The subgroup $K_i$ is the ``almost maximal" level structure referred to in Theorem \ref{thm:vague}. Specializing this to $\bB_i=\bB^{(2,\infty)}$, $\bB^{(5,\infty)}$, or ${\rm Mat}_2$ yields $Y_i=\mathcal{S}^2$, $\mathcal{E}^2$, or $Y_0(1)$, respectively. We equip $Y_i$ with the pushforward measure $\mu_i$ of the unique probability $\bG_i(\A)$-invariant Borel measure on $\bG_i(\Q)\backslash \G_i(\A)$.

Let $E$ be an imaginary quadratic field extension of $\Q$. Let $E\hookrightarrow\bB_i$ be an embedding of $\Q$-algebras and write $\iota_i:{\rm Res}_{E/\Q}\mathbb{G}_m/\mathbb{G}_m\hookrightarrow\bG_i$ for the induced embedding and $\bT_i$ for its image. Let $g_{\infty,i}\in\bG_i(\R)$ such that $g_{\infty,i}\bT_i(\R)g_{\infty,i}^{-1}=K_{\infty,j}$. Then $\iota_i$ and $g_i=(g_{f,i},g_{\infty,i})\in\bG_i(\A)$ together define the Heegner datum $\mathscr{D}_i=(\iota_i,g_i)$. Following \cite[Section 3.3]{BBK}, we define the Heegner packet $H_{\mathscr{D}_i}$ as the image of $\bT_i(\Q)\backslash\bT_i(\A)g_i$ in $Y_i$. Define a quadratic order in $\iota(E)\simeq E$ by $\mathscr{O}_{\mathscr{D}_i}=\iota_i(E)\cap \mathbb{O}^{g_i}$, where $\mathbb{O}^{g_i}=\bB(\Q)\cap g_i\widehat{\mathbb{O}}_ig_i^{-1}$. Then $H_{\mathscr{D}_i}$ is a torsor for the Picard group ${\rm Pic}(\mathscr{O}_{\mathscr{D}_i})$ \cite[Section 3.4]{BBK}. We fix a base point $x_i\in H_{\mathscr{D}_i}$, so that $H_{\mathscr{D}_i}=\{[\mathfrak{a}].x_i: [\mathfrak{a}]\in {\rm Pic}(\mathscr{O}_{\mathscr{D}_i})\}$. Henceforth we shall assume that $\mathscr{O}_{\mathscr{D}_1}=\mathscr{O}_{\mathscr{D}_2}=\mathscr{O}_E$. Write $D=|{\rm disc}(\mathscr{O}_E)|$ and $\ClD={\rm Pic}(\mathscr{O}_E)$. 

For $\nu=1,2$, we call the projection to $Y_1\times Y_2$ of $\{(\iota_1(t)g_1, \iota_2(t)^\nu g_2): t\in \A_E^\times/\A^\times\}$ the \textit{$\nu$-diagonal Heegner packet} and denote it by $H_{\mathscr{D}_1,\mathscr{D}_2}^{\Delta_\nu}$. Then by \cite[(4.8)]{BBK} for any $x_i\in H_{\mathscr{D}_i}$ we have
\[
H_{\mathscr{D}_1,\mathscr{D}_2}^{\Delta_\nu}=\{ ( [\mathfrak{a}].x_1, [\mathfrak{a}]^\nu. x_2): [\mathfrak{a}]\in \ClD\}\subset Y_1\times Y_2.
\]
Correspondingly, for a function $\varphi\in C^\infty(Y_1\times Y_2)$ we define 
\[
P_{\mathscr{D}_1,\mathscr{D}_2}^{\Delta_\nu}(\varphi)=\frac{1}{|\ClD|}\sum_{[\mathfrak{a}]\in \ClD} \varphi([\mathfrak{a}].x_1,[\mathfrak{a}]^\nu.x_2).
\]
When the case $\nu=1$ is considered separately, we shall often drop the subscript from $\Delta_1$, writing $H_{\mathscr{D}_1,\mathscr{D}_2}^\Delta$ and $P_{\mathscr{D}_1,\mathscr{D}_2}^\Delta$. The dependency of the quantities $H_{\mathscr{D}_1,\mathscr{D}_2}^{\Delta_\nu}$ and $P_{\mathscr{D}_1,\mathscr{D}_2}^{\Delta_\nu}(\varphi)$ on the choice of base points $x_1,x_2$ has been suppressed from the notation, as all estimates we prove in this paper are uniform in them.

The following is our main theorem, from which Theorems \ref{thm-AES-intro} and \ref{thm-sphere-ellipsoid} and the variants sketched in Remarks \ref{rem22} and \ref{rem:variants}, can be obtained as special cases. 

\begin{theorem}\label{thm11} If $\nu = 1$, assume that $\bB_1$ and $\bB_2$ are non-isomorphic, and if $\nu=2$ assume that $\bB_1$ is non-split.
Let $1 \leq \psi(D) \leq o(\log D)$ be a function tending to infinity.

Then the $\nu$-diagonal Heegner packets $H_{\mathscr{D}_1,\mathscr{D}_2}^{\Delta_\nu}$
equidistribute on $Y_1\times Y_2$ relative to $\mu_1\times\mu_2$ as $D\rightarrow\infty$ along subsequences for which $L(s, \chi_{-D})$ has no zero in the region 
\begin{equation}\label{zero-freePsi}
\Re s \geq 1 - \frac{20 \psi(D) \log \psi(D)}{\log D}, \quad  |\Im s| \leq \frac{2\psi(D)^2}{\log D}.
\end{equation}
More precisely, there is $\rho > 0$ such that, for every 
$\varphi\in C^\infty_b(Y_1\times Y_2)$ and a negative fundamental discriminant $-D$ as above, we have
\[
P_{\mathscr{D}_1,\mathscr{D}_2}^{\Delta_\nu}(\varphi)=\int_{Y_1\times Y_2}\varphi \, {\rm d}(\mu_1\times\mu_2)+O_{\varphi,\bB_1,\bB_2}({\rm err}_\nu(\psi(D))),
\]
where
\begin{equation}\label{error}
{\rm err}_\nu(\psi(D))=\begin{cases} \psi(D)^{-\rho},& \nu=1;\\
 (\log \psi(D))^{-1/2},&\nu=2.
\end{cases}
\end{equation}
\end{theorem}

\begin{remark}
Our argument works for any $\nu\geq 2$ for which the $\nu$-torsion subgroup satisfies $|\ClD[\nu]|\ll D^{\varepsilon_0}$, for a sufficiently small (but fixed) $\varepsilon_0>0$. This is quite surprising as the final saving in \eqref{error} is at best a fractional power of $\log\log D$. The rate of convergence in the case $\nu = 2$ can be certainly improved, perhaps also to $\psi(D)^{- \rho}$, we do not pursue this in the current paper. 
\end{remark}

From the finiteness of the class number of $\bB_1$ and $\bB_2$, we can and will assume that $g_f=e$. In that case, since $D$ is fundamental, the embeddings $E\hookrightarrow \bB_i(\Q)$ are optimal. For these facts see \cite[Sections 3.3, 4.5.1]{BBK}.

\subsection{Strategy, outline, and auxiliary results}\label{sec:strategy}
The proof of Theorem \ref{thm11} passes through an appeal to the Weyl equidistribution criterion, according to which it is enough to sample equidistribution from two automorphic forms $f_1\in\sigma_1$ on $\bPB_1^\times$ and $f_2\in\sigma_2$ on $\bPB_2^\times$.

In view of the hypotheses of Theorem \ref{thm11}, we may assume, without loss of generality, that $\bB_1$ is not split. If, however, $\bB_2$ is the matrix algebra, as in the Gau{\ss} orthogonal complement procedure of Theorem \ref{thm-AES-intro}, then we must allow for $f_2$ to lie in the continuous spectrum. This is a delicate issue, since the results of \cite{BB} do not (and cannot) apply if $f_2$ is taken to be the spectral Eisenstein series. For this reason, we shall take $f_2$ to be an \textit{incomplete} Eisenstein series, which features an additional continuous average, but whose Fourier coefficients are no longer multiplicative.

We now divide the argument into three steps. Steps 1 and 2 prove Theorem \ref{thm:unconditional}, which provides an unconditional bound on $P_{\mathscr{D}_1,\mathscr{D}_2}^{\Delta_\nu}(f_1\otimes f_2)$ in terms of short arithmetic expressions involving the Hecke eigenvalues of $f_1, f_2$ and the splitting behavior of primes in the extension $\Q(\sqrt{-d})$. Step 3 then deduces Theorem \ref{thm11} from Theorem \ref{thm:unconditional}, using the assumption on Siegel zeros. We describe these steps in more detail below.

\medskip

-- {\it Step 1: Sections \ref{sec:mollification}-\ref{sec-off}.} The key to proving Theorem \ref{thm:unconditional} is to reduce it to instances of the \emph{mixing conjecture} of Michel--Venkatesh, and its twisted variants, for low-complexity parameter ranges. We do this via two different mollification schemes, tailored respectively to the equivariant setting of Theorem \ref{thm-sphere-ellipsoid} and the off-speed setting of Theorem \ref{thm-AES-intro}. The corresponding arguments are developed in Sections \ref{sec:mollification} and \ref{sec-off}. 

In the equivariant setting, where we may suppose $\sigma_1\neq\sigma_2$, we use a mollification procedure introduced in \cite{RS} to exploit the statistical independence of the $\chi$-twisted Heegner periods
\[
\chi\mapsto \frac{1}{|\ClD|}\sum_{[\mathfrak{a}]\in \ClD} f_1([\mathfrak{a}].x)\chi(\mathfrak{a})\qquad\textrm{and}\qquad \chi\mapsto \frac{1}{|\ClD|}\sum_{[\mathfrak{a}]\in \ClD} f_2([\mathfrak{a}].x)\chi(\mathfrak{a}),
\]
where $\chi$ ranges over the group of class group characters $\ClD^\wedge$. The mollifier is constructed as a truncated exponential of prime-ideal sums, with a truncation parameter depending on the ideal norm. It is symmetric in $f_1$ and $f_2$ when both are cuspidal. If $f_2$ is an incomplete Eisenstein series, we introduce a new asymmetric scheme based on the information from the cusp form $f_1$ alone, as a way of counteracting the lack of multiplicative structure of its Fourier coefficients. In either case, this strategy leads to bounds for $P_{\mathscr{D}_1,\mathscr{D}_2}^{\Delta_\nu}(f_1\otimes f_2)$ in terms of short Euler products.  
  
In the off-speed setting we introduce a completely new technique to exploit the statistical independence of the skew-frequency Heegner periods
\[
\chi\mapsto \frac{1}{|\ClD|}\sum_{[\mathfrak{a}]\in \ClD} f_1([\mathfrak{a}].x)\chi^2(\mathfrak{a})\qquad\textrm{and}\qquad \chi\mapsto \frac{1}{|\ClD|}\sum_{[\mathfrak{a}]\in \ClD} f_2([\mathfrak{a}].x)\chi(\mathfrak{a}),
\]
\textit{whether or not} $\sigma_1=\sigma_2$. We regard this as one of the most innovative parts of our paper. Inprired by the Tur\'an-Kubilius inequality, our approach applies the Cauchy--Schwarz inequality after conditioning on the typical behavior of the above two Heegner periods. Here our mollifier is a Dirichlet polynomial over prime ideals --- with no truncated exponential --- and yields bounds for $P_{\mathscr{D}_1,\mathscr{D}_2}^{\Delta_\nu}(f_1\otimes f_2)$ in terms of short Dirichlet polynomials (rather than short Euler products). We fortify the argument by an additional average over quadratic class group characters $\psi$ to overcome the lack of surjectivity of the map $\chi \mapsto \chi^2$ on $\ClD^\wedge$. 

\medskip

{\it Step 2: Section \ref{twistedmoments}.} By Step 1, the proof of Theorem \ref{thm:unconditional} is now reduced to estimating the correlation sum
\[
\frac{1}{|\ClD|}\sum_{[\mathfrak{a}]\in \ClD} f_i([\mathfrak{a}].x)\bar{f_i}([\mathfrak{n}][\mathfrak{a}].x)\qquad (i=1,2),
\]
for integral ideals $\mathfrak{n}\subset\mathscr{O}_{\Q(\sqrt{-d})}$ of small norm relative to $D$.  Theorem \ref{thm12}, which we prove in Section \ref{twistedmoments}, provides uniform asymptotics for such sums, with a power-savings error term. In fact, we treat slightly more general correlation sums, which include twists by $\psi([\mathfrak{a}])$, for $\psi\in \ClD^\wedge$ a fixed class group character. 

We may express Theorem \ref{thm12} as a bound on a twisted average of toric periods, as in \eqref{cusp} below. It is important to note, however, that the relation of this twisted average to a fractional moment of central $L$-values, via Waldspurger's theorem, is not relevant for this step in the argument. Indeed, the spectral expansion which drives the proof is most naturally expressed at the level of periods. Only after the spectral expansion do we estimate the resulting periods by an appeal to hybrid subconvexity bounds for $L$-functions proved in \cite{HM} and \cite{DFI3}. It is nevertheless of independent interest to reinterpret Theorem \ref{thm12} as a result on moments of $L$-functions. As such it has the shape of a new spectral reciprocity formula, namely a non-split version of the Motohashi formula. We refer to Remark \ref{reciprocity} for more details. 

\medskip

-- {\it Step 3: Sections \ref{sec2}-\ref{finalsec}.} The third and final step is to show that the short arithmetic expressions produced in Theorem  \ref{thm:unconditional} tend to zero as $D\rightarrow\infty$. We use known cases of functoriality and analytic number theory to show that, under the hypotheses of Theorem \ref{thm11}, such arithmetic expressions vanish asymptotically. We do so by an approximation to the Sato--Tate law for the distribution of Hecke eigenvalues based on known cases of functoriality (Theorem \ref{thm15}) and a precise dictionary between zeros of quadratic Dirichlet $L$-functions and sums over primes (Theorem \ref{thm14}). The ultimate success of this step relies on some extremely tight linear programming the outcome of which could not have been foreseen; see Remark \ref{rem:S-T}. The proof of the numerically critical Lemma \ref{lem:poly} has been formalized in Lean4.

\section{Reduction to number theoretic statements}\label{sec:aux-results}

In this section, we reduce the proof of Theorem \ref{thm11} to a collection of auxiliary number theoretic results that constitute the three steps described in Section \ref{sec:strategy}. The remainder of the paper is then dedicated to the proofs of these auxiliary results, which should be of independent interest.

\subsection{Weyl criterion}\label{seceisen}

Theorem \ref{thm11} states that the measure $\varphi\mapsto P_{\mathscr{D}_1,\mathscr{D}_2}^{\Delta_\nu}(\varphi)$ weak-* converges to the uniform probability measure $\mu_1\times \mu_2$ on $Y_1\times Y_2$, with an effective rate of convergence. By the Weyl criterion for equidistribution, it suffices to bound $P_{\mathscr{D}_1,\mathscr{D}_2}^{\Delta_\nu}(\varphi)$ for a spanning set of functions $\varphi=f_1\otimes f_2\in L^2(Y_1\times Y_2)$, provided one has an additional polynomial uniformity in the spectral parameters $\lambda_1,\lambda_2$ of $f_1,f_2$. See \cite[Section 6.2]{BBK} for the definition of spectral parameter in this setting.

By definition, the discrete spectrum $L^2_{\rm disc}(Y_1\times Y_2)$ decomposes as 
\[
L^2_{\rm disc}(Y_1\times Y_2)=\bigoplus_{\sigma_i\subset L^2(\bG_i(\Q)\backslash\bG_i(\A))}\sigma_1^{K_1}\otimes \sigma_2^{K_2}.
\]
Since $K_i$ is almost maximal, we have $\dim \sigma_1^{K_2}\dim \sigma_2^{K_2}\leq 1$. In the case of equality we choose an $L^2$-normalized $\varphi=f_1\otimes f_2\in \sigma_1^{K_2}\otimes\sigma_2^{K_2}$. We must show, under the assumptions of Theorem \ref{thm11}, that
\[
P_{\mathscr{D}_1,\mathscr{D}_2}^{\Delta_\nu}(f_1\otimes f_2)=\mu_1(f_1)\mu_2(f_2)+O_{\bB_1,\bB_2}\left((1+\max_{i=1,2} \{|\lambda_i|\})^A{\rm err}_\nu(\psi(D))\right),
\]
for some constant $A>0$. We shall generally suppress the polynomial dependence on the spectral parameters in our estimates.

If $\sigma_1=\sigma_2=\C_{\rm triv}$, the trivial representation, there is nothing to show. If only one of $\sigma_1$ or $\sigma_2$ is $\C_{\rm triv}$, then the desired convergence follows from Duke's equidistribution theorem, as recalled in the introduction. For the discrete part of the spectrum we are therefore reduced to showing that $P_{\mathscr{D}_1,\mathscr{D}_2}^{\Delta_\nu}(f_1\otimes f_2)=O({\rm err}_\nu(\psi(D)))$, for $L^2$-normalized $K_i$-invariant functions $f_i$ lying in infinite dimensional unitary representations $\sigma_i\subset L^2(\bG_i(\Q)\backslash\bG_i(\A))$.

It remains to treat the case when $\bB_2={\rm Mat}_2(\R)$, for which
\[
Y_2=\mathbf{PGL}_2(\Q)\backslash\mathbf{PGL}_2(\A)/\mathbf{PGL}_2(\widehat{\Z}){\rm PO}(2)=Y_0(1).
\]
Let $E(z,s)=\sum_{\gamma\in \Gamma_\infty\backslash \SL_2(\Z)} y(\gamma.z)^s$ be the  spectral Eisenstein series. Let $\Psi$ be a smooth function of compact support in $\R_{>0}$ with $\int_0^{\infty} \Psi(y) dy = 0$ and Mellin transform $\hat{\Psi}$. Consider the incomplete Eisenstein series 
\[
E_{\Psi}(z) = \int_{(1/2)} \hat{\Psi}(-s) E(z, s) \frac{ds}{2\pi i}.
\]
By Duke's theorem for the modular curve, the desired convergence holds for $\varphi=1\otimes E_\Psi$. For the continuous spectrum we are therefore reduced to showing that $P_{\mathscr{D}_1,\mathscr{D}_2}^{\Delta_\nu}(f_1\otimes E_\Psi)=O({\rm err}_\nu(\psi(D)))$, for $L^2$-normalized $K_1$-invariant functions $f_1$ lying in an infinite-dimensional $\sigma_1\subset L^2(\bG_1(\Q)\backslash\bG_1(\A))$ and $\Psi$ a mean-zero compactly supported smooth function on $\R_{>0}$.

\subsection{Reduction to short arithmetic expressions}\label{secoff} 
We now bound $P_{\mathscr{D}_1,\mathscr{D}_2}^{\Delta_\nu}(f_1\otimes f_2)$ by appropriate short arithmetic expressions. By this, we mean a short Euler product (when $\nu=1$) or the inverse of a short Dirichlet series (when $\nu=2$), constructed from the Hecke eigenvalues of $f_1$ and $f_2$ at primes which split in $E$. For heuristics on how and why these arithmetical expressions arise, see Sections \ref{sec:heuristic2} and \ref{sec:heuristic1}.

When $f_i$ is cuspidal, we denote by $\pi_i={\rm JL}(\sigma_i)$ the Jacquet--Langlands lift of $\sigma_i$ to ${\bf PGL}_2/\Q$. Since $K_i$ is almost maximal, $\pi_i$ has conductor the discriminant $d_{\bB_i}$ of $\bB_i$. For a prime $p\nmid d_{\bB_i}$ write $\lambda_{\pi_i}(p)$ for the Hecke eigenvalue of $\pi_i$ at $p$. 
\begin{enumerate}
\item In \textit{symmetric} situations, in which $\nu=1$ and both $f_1,f_2$ are cuspidal, we shall be interested in the following quantity: 
\[
S_D(\pi_1,\pi_2)=\sum_{\substack{ C_0 \leq  p\leq D^c\\ p \textrm{ {\rm split} in } E\\ |\lambda_{\pi_1}(p)|, |\lambda_{\pi_2}(p)| \leq B}}\frac{(\lambda_{\pi_1}(p)-\lambda_{\pi_2}(p))^2}{p},
\]
where $C_0>\max\{d_{\bB_1}, d_{\bB_2}\}$, $c, B>0$ are constants.

\item In \textit{asymmetric situations}, either $\nu=2$ or $f_1$ is cuspidal and $f_2=E_\Psi$ is an incomplete Eisenstein series, we put
\[
T_D(\pi_1)=\sum_{\substack{ C_0 \leq  p\leq D^c\\ p \textrm{ {\rm split} in } E\\  |\lambda_{\pi_1}(p)| \leq B}}\frac{\lambda_{\pi_1}(p)^2}{p}.
\]
Here, $C_0>d_{\bB_1}$ and $c>0$ are constants, but now we also allow $B\in\R_{>0}\cup\{\infty\}$ to be infinite (so that the condition on the size of $\lambda_{\pi_1}(p)$ disappears).
\end{enumerate}
Throughout, we denote by $\Psi$ a smooth and compactly supported function on $\R_{>0}$ such that $\int_0^{\infty} \Psi(y) \, {\rm d} y = 0$. We denote by $\hat{\Psi}$ its Mellin transform.

We prove the following theorem over the course of Sections \ref{twistedmoments}-\ref{sec-off}.

\begin{theorem}\label{thm:unconditional}
As in Theorem \ref{thm11}, we assume that $\bB_1$ and $\bB_2$ are non-isomorphic when $\nu = 1$, and that $\bB_1$ is non-split when $\nu=2$. With notations as above,
\begin{enumerate}
\item\label{boundmoment} there exist constants $c>0$, $C_0>\max\{d_{\bB_1}, d_{\bB_2}\}$ such that, for every (finite) $B>0$ and $\alpha < 1/4$, we have
\begin{align*}
P_{\mathscr{D}_1,\mathscr{D}_2}^{\Delta}(f_1\otimes f_2)&\ll 
\begin{cases}
\exp\Big(-\frac14 S_D(\pi_1,\pi_2)\Big),& f_2 \textrm{ cuspidal};\\[10pt]
\exp\Big(-(\alpha - \alpha^2) T_D(\pi_1)\Big),& f_2=E_\Psi.
\end{cases}
\end{align*}
\item\label{thm21a} there exist constants $c>0$, $C_0 >\max\{d_{\bB_1}, d_{\bB_2}\}$ such that, with $B=\infty$,
\[
P_{\mathscr{D}_1,\mathscr{D}_2}^{\Delta_2}(f_1\otimes f_2) \ll (1+T_D(\pi_1))^{-1/2}.
\]
\end{enumerate}
\end{theorem}

\begin{remarks}\label{rmks-upper-bds} ${\,}$
\begin{enumerate}[(a)]
\item The optimal value of $\alpha$ in part \eqref{boundmoment} of Theorem \ref{thm:unconditional} would in fact be $\alpha = 1/2$. For technical reasons that will become apparent in the proof of Lemma \ref{calpha}, we need to choose $\alpha < 1/4$. The numerical value of $\alpha$ plays no role for our application. 
\medskip
\item The requirement that $B<\infty$ in part \eqref{boundmoment} stems from an application of Rankin's trick to convert to an Euler product majorant. This argument is absent in the proof of part \eqref{thm21a}, allowing for $B=\infty$.
\end{enumerate}
\end{remarks}

As in \cite{BB}, the proof of Theorem \ref{thm:unconditional} passes through a dual expression for $P_{\mathscr{D}_1,\mathscr{D}_2}^{\Delta_\nu}(f_1\otimes f_2)$, which we now recall. For $x_i\in H_{\mathscr{D}_i}$ and $\chi\in \ClD^\wedge$, we define the \textit{$\chi$-twisted Heegner period} as
\[
W_{\mathscr{D}_i}(f_i;\chi)= \frac{1}{|\ClD|}\sum_{[\mathfrak{a}]\in \ClD} f_i([\mathfrak{a}].x_i)\chi([\mathfrak{a}]).
\]
By Parseval's identity, we have
\begin{equation}\label{parseval}
P_{\mathscr{D}_1,\mathscr{D}_2}^{\Delta_\nu}(f_1\otimes f_2)=\sum_{\chi \in{\rm Cl}^{\wedge}_D} W_{\mathscr{D}_1}(f_1;\chi^\nu)W_{\mathscr{D}_2}(f_2;\overline{\chi}),
\end{equation}
Note that the exponent $\nu$ is now placed in the first factor.

By contrast with \cite{BB}, we do not convert the twisted Heegner periods into central values of $L$-functions, via Waldspurger's formula (for cusp forms $f_i$) and the Hecke formula (for $f_i=E_\Psi$). Instead we follow a period-theoretic approach, which starts by reducing the stated bounds to asymptotic formulae with power-savings error on twisted moments of each $|W_{\mathscr{D}_i}(f_i;\chi^\nu)|^2$ separately ($i=1,2$). This reduction step, which is executed in Sections \ref{sec:mollification} and \ref{sec-off}, is best understood through the heuristics provided in Sections \ref{sec:heuristic2} and \ref{sec:heuristic1}. We then prove such asymptotic formulae using spectral theory.

In Theorem \ref{thm12} below, we state the asymptotic formulae on twisted moments that underpin the proof of Theorem \ref{thm:unconditional}. As this concerns each $f_i$ separately, we drop the index $i$ from the notation for this setting, writing $f=f_i$, $Y=Y_i$, $\sigma=\sigma_i$, $\pi=\pi_i$, and $\mathscr{D}=\mathscr{D}_i$. For an ideal $\mathfrak{n}$ of $\mathscr{O}_E$ and $x\in H_{\mathscr{D}}$ we write
\[
H_{\mathscr{D}}^\Delta([\mathfrak{n}])=\{([\mathfrak{a}].x,[\mathfrak{n}][\mathfrak{a}].x): [\mathfrak{a}]\in \ClD\}\subset Y\times Y
\]
for the joint Heegner packet of \cite[Section 4.2]{BBK}. For $\psi\in \ClD^\wedge$ we put
\begin{equation}\label{cusp}
P_{\mathscr{D}}^\Delta(f\otimes f;\mathfrak{n},\psi)=\sum_{\chi \in\ClD^{\wedge}} \chi(\mathfrak{n})  W_{\mathscr{D}}(f,\chi \psi)  \overline{W_{\mathscr{D}}(f,\chi)}.
\end{equation}
When $\psi$ is the trivial character, we write simply $P_{\mathscr{D}}^\Delta(f\otimes f;\mathfrak{n})$. Note that, in the case of a trivial shift $\mathfrak{n}=\mathscr{O}_E$, we recover the definitions from Section \ref{secgeneral} in the case of $\mathscr{D}_1=\mathscr{D}_2=\mathscr{D}$, namely, $H_{\mathscr{D}}^\Delta([\mathscr{O}_E])=H_{\mathscr{D}\times\mathscr{D}}^\Delta$ and $P_{\mathscr{D}}^\Delta(f\otimes f;\mathscr{O}_E)=P_{\mathscr{D}\times\mathscr{D}}^\Delta(f\otimes f)$.

\begin{remark}\label{rem:psi-triv-case}
We shall only need the twisted variant with a non-trivial $\psi$ in the $\nu=2$ case, as a way of overcoming the non-surjectivity of the squaring map $\ClD\rightarrow \ClD$, $[\mathfrak{a}]\mapsto [\mathfrak{a}]^2$, for composite discriminants $D$. See Section \ref{sec:off-speed-proof}, where we introduce an average of $P_{\mathscr{D}}^\Delta(f\otimes f;\mathfrak{n},\psi)$ over \textit{quadratic} $\psi$.
\end{remark}

For $N$ coprime to the conductor $d_{\bB}$ of $\pi$, we write
\begin{equation*}
\lambda_{\pi}^{\ast}(N) = \sum_{d^2 \mid N} \frac{\mu(d)}{d} \lambda_{\pi}\Big(\frac{N}{d^2}\Big) 
\end{equation*}
for the normalized ``partial'' Hecke eigenvalue of the double coset operator $\SL_2(\Z)\begin{psmallmatrix} N & \\ & 1\end{psmallmatrix}\SL_2(\Z)$, where $\lambda_{\pi}$ are the usual Hecke eigenvalues appearing as Dirichlet coefficients in the standard $L$-function of $\pi$.  We use the same notation for the eigenvalues of Eisenstein series and write, analogously,
\begin{equation}\label{div}
\tau_{it}(N) =\sum_{ab = N}\Big( \frac{a}{b}\Big)^{it}, \qquad \tau^{\ast}_{it}(N)  = \sum_{abd^2 = N} \frac{\mu(d)}{d} \Big(\frac{a}{b}\Big)^{it}
\end{equation}
for $t \in \mathbb{C}$. Let\footnote{Here and elsewhere we use $\pi$ to denote the cuspidal automorphic representation $\pi=\mathrm{JL}(\sigma)$ on $\mathbf{PGL}_2/\Q$ as well as the numerical constant $\pi=4\int_0^1 \frac{dx}{1+x^2}$. There should be no risk of confusion.} $\phi(s) = \sqrt{\pi} \frac{\Gamma(s - 1/2)\zeta(2s - 1)}{\Gamma(s) \zeta(2s)}$ be the determinant of the usual scattering matrix. We define the arithmetic function $V(N) = [\Gamma_0(N): \SL_2(\Z)] = N \prod_{p \mid N} (1 + p^{-1})$, and put
\begin{equation}\label{Lambda-pi}
\Lambda_{\pi}(\mathfrak{n})=\frac{\sqrt{{\rm N}\mathfrak{n}^{\ast}}}{V( {\rm N}\mathfrak{n}^{\ast})}\lambda^{\ast}_{\pi}({\rm N}\mathfrak{n}^{\ast}),\qquad \Lambda_{it}(\mathfrak{n})= \frac{({\rm N}\mathfrak{n}^{\ast})^{1/2} }{V ({\rm N}\mathfrak{n}^{\ast})} \tau^{\ast}_{it}({\rm N}\mathfrak{n}^{\ast}),
\end{equation}
where $\mathfrak{n}^{\ast}$ is the primitive kernel of an ideal $\mathfrak{n} \subset\mathscr{O}_E$, i.e.,\ the largest integral ideal such that $\mathfrak{n}^{\ast} (n) = \mathfrak{n}$ for some nonzero rational integer $n$.  This arithmetic function occurs naturally in the following asymptotic formula, as explained in \cite[Remark 8.2]{BBK}. 

\begin{theorem}\label{thm12} There are constants $A, \delta > 0$ such that the following holds. Let $\mathfrak{n}\subset\mathscr{O}_E$ be an integral ideal coprime to $D$ and $d_\bB$. Let $\psi$ be a class group character. Then
\[
P_{\mathscr{D}}^\Delta(f\otimes f;\mathfrak{n},\psi)=
\begin{cases}
\delta_{\psi=1}  \Lambda_{\pi}(\mathfrak{n}) +O_{\pi}\big(({\rm N}\mathfrak{n}^{\ast})^A D^{-\delta}\big),& f\textrm{ cuspidal};\\[10pt]
\delta_{\psi=1}\int_{\R}  \Lambda_{it}(\mathfrak{n}) \Phi (t) \frac{{\rm d} t}{2\pi} +O_{\Psi}\big(({\rm N}\mathfrak{n}^{\ast})^A D^{-\delta}\big),& f=E_\Psi,
\end{cases}
\]
where, for $s\in\C$, we have put
\begin{equation}\label{def:Phi}
\Phi(s) = \frac{3}{\pi}  \hat{\bar{\Psi}}(-\tfrac{1}{2} - is)\big(\hat{\Psi}(-\tfrac{1}{2}+is) + \phi(\tfrac{1}{2}+is)  \hat{\Psi}(-\tfrac{1}{2}-is)\big).
\end{equation}
\end{theorem}  

Note that $\Phi(s)$ is entire since $\hat{\Psi}(-1) = 0$, and it is rapidly decaying in $t$ as long as $t$ has fixed imaginary part, since $\Psi$ is smooth. 

\begin{remark} It is at this point where we see the necessity to work with incomplete Eisenstein series. For instance, if $f$ is the \textit{spectral} Eisenstein series $E(s,1/2+it)$ the quantity $P_{\mathscr{D}}^\Delta(f\otimes f;\mathfrak{n}, \textbf{1})$ would grow by $\log D$ in general and even by $(\log D)^3$ for $t=0$ \cite[Theorem 2]{DFI2}. Another important point where incomplete Eisenstein series are crucial is Lemma \ref{calpha}; see the discussion preceding that lemma.
\end{remark}

\begin{remark}\label{reciprocity} If $\psi$ is trivial and $f$ is cuspidal of level coprime to $D$, then, by Waldspurger's theorem (see \cite[(6.5)]{BBK} and \cite[Appendix A]{BB}), we have
\begin{equation}\label{eq:twisted-first-moment}
P_{\mathscr{D}}^\Delta(f\otimes f;\mathfrak{n},\psi) = \frac{c_{\pi, D}}{|\ClD|} \sum_{\chi \in \ClD} \chi(\mathfrak{n}) \frac{L(1/2, \pi \times \chi)}{L(1, \chi_D) L(1, \text{Ad}, \pi)}
\end{equation}
for a complex number $c_{\pi, D} \asymp_\pi 1$. Theorem \ref{thm12} provides an asymptotic formula for this twisted first moment, which we hope is of independent interest. Moreover, the error term can itself be expressed in terms of canonical square-roots of central $L$-values of total degree 12, as in \eqref{cusp-cusp-expasion}. We thus obtain a non-split version of Motohashi's reciprocity formula, and our proof is an incarnation of the strategy envisioned by Michel and Venkatesh \cite[\S 4.5.1]{MV1}. For the Eisenstein spectrum, the resemblence of the Motohashi formula becomes even more apparent in the relation
\[
\sum_{\chi \in \ClD} \chi(\mathfrak{n})  L(1/2, \chi)^2\rightsquigarrow \sum_{\varrho \text { cuspidal of level } N\mathfrak{n}} L(1/2, \varrho)^2 L(1/2, {\rm BC}_{\Q(\sqrt{-D})}(\varrho))^{1/2}.
\]
\end{remark}

\subsection{Ensuring growth}\label{sec:growth}

Theorem \ref{thm:unconditional} reduces the proof of Theorem \ref{thm11} to showing that the expressions $S_D(\pi_1,\pi_2)$ and $T_D(\pi)$ tend to $\infty$ with $D$. Heuristically, one would expect
\[
S_D(\pi_1, \pi_2) \approx \sum_{ \substack{ C_0 \leq  p\leq D^c\\ p \textrm{ {\rm split} in } E}} \frac{2}{p}  \approx \log\log D\qquad\textrm{and}\qquad
T_D(\pi) \approx \sum_{ \substack{ C_0 \leq  p\leq D^c\\ p \textrm{ {\rm split} in } E}} \frac{1}{p} \approx  \frac{1}{2} \log\log D.
\]
There are three potential obstacles which could prevent this from happening: 
\begin{enumerate}
\medskip
\item\label{danger2} the numbers $|\lambda_{\pi_1}(p)-\lambda_{\pi_2}(p)|$ (resp.\ $|\lambda_{\pi}(p)|$) could be often zero or very small; or $\max(|\lambda_{\pi_1}(p)|, |\lambda_{\pi_2}(p)|)$ (resp.\ $|\lambda_{\pi}(p)|$) could often be larger than $B$; 
\medskip
\item\label{danger1} there could be few small split primes $p$;
\medskip
\item\label{danger3} there could be a conspiracy between the small split primes in $E$ and the primes for which \eqref{danger2} holds.
\medskip
\end{enumerate}
The first phenomenon is purely a property of the pair $(\pi_1,\pi_2)$ (resp.\ $\pi$), the second is purely a property of $E$, and the third knows something about both. 

Our strategy is to eliminate \eqref{danger2} unconditionally for a positive density set of primes, then eliminate \eqref{danger1} for negative fundamental discriminants for which $L(s,\chi_{-D})$ satisfies \eqref{zero-freePsi} for another positivity density set of primes $p\geq D^{1/\psi(D)}$. If the combined total density is strictly greater than 1, this will ensure enough overlap on such discriminants to eliminate \eqref{danger3}. We carry out this program in Sections \ref{sec2}-\ref{finalsec}.

We begin by addressing \eqref{danger2}. We shall show, in Lemma \ref{lem22}, that the cuspidal representations $\pi$ arising as Jacquet--Langlands lifts $\pi={\rm JL}(\sigma)$ from $\sigma\subset L^2(\bG(\Q)\backslash\bG(\A))$ with almost maximal level invariance, as defined in Section \ref{secgeneral}, have cuspidal power lifts ${\rm sym}^k\pi$ for $k=2,3,4$, and that, moreover, for any two such representations $\pi_1,\pi_2$ coming from distinct quaternion algebras, those same symmetric powers are distinct. This then allows us to apply the following result, which we prove in Section \ref{sec2}.

\begin{theorem}\label{thm15}
$\,$
\begin{enumerate} 
\item\label{Hecke-part1} Let $\pi_1,\pi_2$ be cuspidal automorphic representations of ${\bf PGL}_2/\Q$. Assume that ${\rm sym}^k\pi_1$ and ${\rm sym}^k\pi_2$ are cuspidal and distinct, for all $k=2,3,4$. There exists $\varepsilon > 0$ such that
\[
\sum_{\substack{p \leq X\\ |\lambda_{\pi_1}(p) - \lambda_{\pi_2}(p)| \geq \varepsilon} }  \log p \geq 0.508 X
\]
for $X \geq X_0$ sufficiently large.

\item\label{Hecke-part2} Let $\pi$ be a cuspidal automorphic representation of ${\bf PGL}_2/\Q$. Assume that ${\rm sym}^k\pi$ is cuspidal for all $k=2,3,4$. Then there exists $\varepsilon > 0$ such that
\[
\sum_{\substack{p \leq X\\ |\lambda_{\pi}(p)| \geq \varepsilon} }  \log p \geq 0.58 X
\]
for $X \geq X_0$ sufficiently large.
\end{enumerate}

\end{theorem}

\begin{remark}\label{rem:S-T}
Theorem \ref{thm15} ensures that the logarithmic density of primes verifying the indicated inequalities on Hecke eigenvalues is strictly larger than the crucial threshold 1/2. If both $\pi_1$ and $\pi_2$ (resp.\ $\pi$) correspond to holomorphic modular forms, this would in fact be an easy consequence of the joint Sato--Tate conjecture (resp.\ Sato--Tate conjecture); see e.g.\ \cite{Th} for an effective version. For Maa{\ss} forms we have to rely on partial progress towards these conjectures provided by the functoriality of the first few symmetric power lifts.  

It is extremely fortunate that the automorphy (and cuspidality criteria) of the first few symmetric powers ${\rm sym}^2$, ${\rm sym}^3$, ${\rm sym}^4$ of ${\bf GL}_2$, proved by Gelbart--Jacquet, Kim--Shahidi, and Kim, just suffices to beat the constant $1/2$. Indeed, if one removes, for example, the hypothesis that ${\rm sym}^4\pi_1\neq {\rm sym}^4\pi_2$ in Part \eqref{Hecke-part1}, but retains the others, then there exist pairs of distinct even icosahedral Galois representations $\rho_1,\rho_2$, which conjecturally correspond to eigenvalue 1/4 Maa\ss\, cusp forms, for which ${\rm tr}(\rho_1 ({\rm Frob}_p))={\rm tr}(\rho_2({\rm Frob}_p))$ for a set of primes $p$ of Dirichlet density $3/5$; see \cite{LMFDB}. Similarly, if one removes in \eqref{Hecke-part2} the hypothesis that ${\rm sym}^2(\pi)$ is cuspidal, then dihedral forms have vanishing Hecke eigenvalues on a density 1/2 of primes.
\end{remark}

Regarding \eqref{danger1}, we prove the following theorem in Section \ref{sec:zeros}. 

\begin{theorem}\label{thm14} Let $1 \leq \psi(D) \leq o(\log D)$ be a function tending to infinity. Then for any fixed $\delta < 1/2$ and any sufficiently large negative fundamental discriminant $-D$ such that $L(s,\chi_{-D})$ satisfies \eqref{zero-freePsi} we have 
\[
\sum_{\substack{X \leq p \leq X^2 \\ \chi_{-D}(p) = 1}} \frac{1}{p} \geq \delta  \sum_{ X \leq p \leq X^2} \frac{1}{p}, 
\]
uniformly for $X \geq D^{1/\psi(D)}$.
\end{theorem}

In Section \ref{finalsec}, we combine the above two results, to deduce the following corollary, which addresses the final problem \eqref{danger3}, essentially by a pigeonhole argument. 

\begin{corollary}\label{lem61} Let $1 \leq \psi(D) \leq o(\log D)$ be a function tending to infinity. Then there are constants $B>2$ and $\kappa>0$ such that for any sufficiently large negative fundamental discriminant $-D$ such that $L(s,\chi_{-D})$ satisfies \eqref{zero-freePsi} we have $S_D(\pi_1,\pi_2)>\kappa \log\psi(D)$ and $T_D(\pi)>\kappa \log\psi(D)$.
\end{corollary}

Combining all results in this section completes the proof of Theorem \ref{thm11}.

\section{Twisted moments of $L$-functions}\label{twistedmoments}

In this section we prove Theorem \ref{thm12}. Related formulae appear in \cite{BBK} for $f$ cuspidal and in \cite{DFI2, Bl} for $f=E_\Psi$. Our proof proceeds by writing \eqref{cusp} as a twisted Heegner period on a Hecke correspondence inside $Y\times Y$, over which we spectrally expand. 

\subsection{Proof of Theorem \ref{thm12}: cuspidal case}\label{sec:Theorem12}

Recall the defining expression \eqref{cusp} for $P_{\mathscr{D}}^\Delta(f\otimes f;\mathfrak{n},\psi)$. By the Parseval relation \cite[(6.4)]{BBK}, we obtain 
\begin{equation}\label{momentM1}
P_{\mathscr{D}}^\Delta(f\otimes f;\mathfrak{n},\psi)=\frac{1}{|\ClD|}\sum_{[\mathfrak{a}]\in \ClD} f([\mathfrak{a}].x)\bar{f}([\mathfrak{n}][\mathfrak{a}].x) \psi([\mathfrak{a}]).
\end{equation}
Note that both sides are invariant under changing $\mathfrak{n}$ into $\mathfrak{n}(n)$ for any nonzero rational integer $n$. In particular, we may replace $\mathfrak{n}$ by its primitive kernel $\mathfrak{n}^{\ast}$, whose norm we agree to denote by $N$. If $\psi$ is trivial, \cite[Prop.\ 9.4]{BBK} gives a spectral decomposition of the right-hand side in terms of automorphic forms of level $N$. The main term comes from the contribution of the constant function and equals $\Lambda_{\pi}(\mathfrak{n})$ (see also \cite[Remark 8.2]{BBK}). For non-trivial $\psi$, the  contribution of the constant function vanishes, since $\sum_{[\mathfrak{a}]\in \ClD}  \psi(\mathfrak{a}) = 0$. 
 
The non-constant part of the spectrum is analyzed in \cite[Lemmata 10.2, 10.3]{BBK} for trivial $\psi$, and \cite[Lemma 10.3]{BBK} holds with small modifications for arbitrary $\psi$. In this way the non-constant contribution can be bounded by
\begin{equation}\label{cusp-cusp-expasion}
\frac{ (ND)^{\varepsilon}}{D^{1/4 } N^{1/2}} \sum_{\varrho \textrm{ cuspidal level } N} |L(1/2, \varrho \times \psi)  L(1/2, \pi\times \pi \times \varrho )|^{1/2} (1 + |t_\varrho|)^{-100}
\end{equation}
where $t_\varrho$ is the spectral parameter of $\varrho$, plus a similar term for the Eisenstein spectrum. 

When $\varrho$ is cuspidal, \cite[Theorem 1]{HM} provides a subconvex bound for $L(1/2, \varrho \times \psi)$ in terms of $D$ with polynomial dependence in the parameters of $\varrho$. In the corresponding Eisenstein term, \cite[Theorem 2.5]{DFI3} provides a subconvex bound on $L(1/2 + it, \psi)$ in the $D$-aspect with polynomial dependence in the $t$ aspect. The convexity bounds for all other $L$-functions gives an error term of the form $O(N^A D^{-\delta})$ for some $A, \delta > 0$.

\begin{remark}
The first term $\Lambda_{\pi}(\mathfrak{n})$ goes to zero as $N\rightarrow\infty$, by known non-trivial bounds towards the Ramanujan conjecture. The utility of Theorem \ref{thm12} lies in the power savings in $D$ with polynomial dependence in $N$, as it shows the existence of some $\alpha>0$ such that the second term also goes to zero as $N\rightarrow\infty$ in the range $N\ll D^\alpha$.
\end{remark}

\begin{remark}\label{rem1/6}
Although not important for our purposes, we can explicate the values of $A$ and $\delta$ using known moment and subconevxity bounds, as in \cite[Remark 10.1]{BBK} for trivial $\psi$. 
\end{remark}

\subsection{Proof of Theorem \ref{thm12}: Eisenstein case}

Applying \eqref{momentM1} with $E_\Psi$ we get
\begin{equation}\label{momentM1eis}
P_{\mathscr{D}}^\Delta(E_\Psi\otimes E_\Psi;\mathfrak{n},\psi)=\frac{1}{|\ClD|}\sum_{[\mathfrak{a}]\in \ClD} E_\Psi(\tau_{[\mathfrak{a}]}) \overline{E}_\Psi(\tau_{[\mathfrak{n}][\mathfrak{a}]}) \psi([\mathfrak{a}]).
\end{equation}
Once again, we may replace $\mathfrak{n}$ by its primitive kernel $\mathfrak{n}^*$ and we write $N={\rm N}\mathfrak{n}^*$. 

We recall some of the structural elements of the spectral expansion of $\varphi\mapsto P_{\mathscr{D}}^\Delta(\varphi)$ on $Y_0(1)\times Y_0(1)$ across the level $N$ Hecke correspondence. These elements are common to both the cuspidal and Eisenstein cases of Theorem \ref{thm12}; we provide more details in the latter case, as this setting was not treated in \cite{BBK}.

Let $Y_0(N)=\Gamma_0(N) \backslash \mathbb{H}$, a degree-$V(N)$ cover of the modular surface $Y_0(1)$. Recall from \cite[Section 7]{BBK} that $\mathfrak{n}$ determines an aligned Heegner packet $H_{\mathscr{D}}(\mathfrak{n})$ of discriminant $-D$ on $Y_0(N)$. For a function $u\in C^\infty(Y_0(N))$ and $x\in H_{\mathscr{D}}(\mathfrak{n})$ we define a $\psi$-twisted Heegner period on $Y_0(N)$ by
\[
W_{\mathscr{D}}(u;\mathfrak{n}, \psi)=\frac{1}{|\ClD|}\sum_{[\mathfrak{a}]\in \ClD} u([\mathfrak{a}].x)\psi([\mathfrak{a}]).
\]
Let $Y_0^\Delta (N)=\{(\SL_2(\Z).z, \SL_2(\Z).gz): g\in\SL_2(\Z)\begin{psmallmatrix} N & \\ & 1\end{psmallmatrix}\SL_2(\Z)\}\subset Y_0(1)\times Y_0(1)$ denote the level-$N$ Hecke correspondence. We have a bijective map
\[
\beta_N: Y_0(N)\rightarrow Y_0^\Delta (N),\qquad \Gamma_0(N).z\mapsto ( \SL_2(\Z).z, \SL_2(\Z).Nz).
\]
It follows from \cite[Sections 8-9]{BBK} that $H_{\mathscr{D}}^\Delta([\mathfrak{n}])\subset Y_0^\Delta(N)$ and that 
\[
P_{\mathscr{D}}^\Delta(E_\Psi\otimes E_\Psi;\mathfrak{n},\psi)=W_{\mathscr{D}}((E_\Psi\otimes E_\Psi)_{|Y_0^\Delta (N)}\circ \beta_N;\mathfrak{n}, \psi).
\]
Since $(E_\Psi\otimes E_\Psi)_{|Y_0^\Delta (N)}\circ \beta_N=E_\Psi \overline{E}_\Psi^{(N)}$, where we have put $E_\Psi^{(N)}(z)=E_\Psi(Nz)$, we deduce that the right-hand side of \eqref{momentM1eis} can be written as $W_{\mathscr{D}}(E_\Psi \overline{E}_\Psi^{(N)};\mathfrak{n}, \psi)$.

We endow $Y_0(N)$ with the uniform probability measure given by $V(N)^{-1}\frac{3}{\pi}\textrm{d}x\textrm{d}y/y^2$. Let $\langle ., .\rangle$ denote the inner product on $L^2(Y_0(N))$ and, in an effort to simplify the notation, we denote the spectral decomposition in $L^2(Y_0(N))$ by 
\[
f(z) =  \langle f,  \textbf{1}\rangle+ \int^{\ast}_{(N)} \langle f, \varpi\rangle \varpi(z)  \, {\rm d} \varpi,
\]
where $\int^{\ast}_{(N)}$ denotes a combined sum/integral over an orthonormal basis of cusp forms and a suitably parametrized set of Eisenstein series. In this notation, we need to evaluate 
\begin{equation}\label{spec}
\delta_{\psi = 1} \langle E_\Psi \overline{E}^{(N)}_\Psi, {\bf 1}\rangle   +   \int^{\ast}_{(N)} \langle E_\Psi \overline{E}^{(N)}_\Psi,  \varpi  \rangle W_{\mathscr{D}_i}(\varpi;\mathfrak{n}, \psi) \, {\rm d} \varpi.
\end{equation}
The main term will come from the contribution of the constant function, which we have singled out; we estimate everything else rather coarsely. 

We begin by estimating the spectral expansion in \eqref{spec}, which is the analogue for $f=E_\Psi$ of \eqref{cusp-cusp-expasion}. Combining \cite[Lemmata 10.3, 10.5]{BBK} with the same subconvexity bounds as in Section \ref{sec:Theorem12}, we have
\[
W_{\mathscr{D}}(\varpi;\mathfrak{n}, \psi)\ll_B D^{- \delta} N^A  (1 + |\mu_{\varpi}|)^{-B}
\]
for some $\delta, A > 0$ and any $B > 0$, where $\mu_{\varpi}$ denotes the archimedean Langlands parameter of $\varpi$. Thus the second term in \eqref{spec} is bounded by
\begin{equation}\label{spectral-error}
\ll_B D^{-\delta} N^A  \int^{\ast}_{(N)} |\langle E_\Psi \overline{E}^{(N)}_\Psi,  \varpi  \rangle |  \frac{{\rm d} \varpi}{(1 + |\mu_{\varpi}|)^{B} },
\end{equation}
and it remains to show that the spectral sum/integral is polynomially bounded in $N$. In the cuspidal setting of Section \ref{sec:Theorem12}, this is achieved by applying triple product and Rankin--Selberg identities which express $\langle f f^{(N)},  \varpi   \rangle$ in terms of $L$-functions and bounding those crudely. For $f=E_\Psi$ an incomplete Eisenstein series we shall estimate $\langle E_\Psi \overline{E}^{(N)}_\Psi ,  \varpi  \rangle$ directly.

Let $\xi(z) = \sum_{\gamma\in  \Gamma_0(N)\backslash \SL_2(\Z)} (\varpi {E_\Psi^{(N)}} )(\gamma z)$. By unfolding we have
\[
 \langle E_\Psi \overline{E}_\Psi^{(N)}, \varpi \rangle= \langle E_\Psi , \varpi {E_\Psi^{(N)}} \rangle  =  \int_0^1\int_0^{\infty} \Psi(y)\xi(z) \frac{3}{\pi}\frac{{\rm d} x\, {\rm d} y}{y^2}\ll \max_{\substack{x \in [0, 1]\\y \in \text{supp}(\Psi)}}  |\xi(x+iy)|.
\]
For a convenient choice of representatives of $\Gamma_0(N)\backslash \SL_2(\Z)$, we may take $\begin{psmallmatrix} * & *\\ u & v\end{psmallmatrix}$ with $v\mid N$, $u$ mod $N/v$. For such representatives $\gamma$ we have the crude bound $N^{-2} \ll_z \Im \gamma z\ll_z 1$. Since there are $V(N)\ll N^{1+\varepsilon}$ representatives, we deduce that
\[
 \langle E_\Psi \overline{E}_\Psi^{(N)}, \varpi \rangle \ll N^{1+\varepsilon}\sup_{\substack{x \in \R\\ N^{-2}\ll  y \ll 1}} |\varpi  {E_\Psi^{(N)}}(x+iy)| \ll N^{1+\varepsilon}\sup_{\substack{x \in \R\\ N^{-2}\ll  y \ll 1}} |\varpi(x+iy)|,
\]
where in the last step we used that  $\|E_\Psi\|_{\infty} \ll 1$, and hence clearly $\|E^{(N)}_\Psi\|_{\infty} \ll 1$. Let us assume that $B$ is a sufficiently large even integer, then we can use Cauchy--Schwarz to bound 
\[
\int^{\ast}_{(N)} | \langle E_\Psi \overline{E}_\Psi^{(N)}, \varpi \rangle |  \frac{{\rm d} \varpi}{(1 + |\mu_{\varpi}|)^{B} }\ll N^{3/2+\varepsilon}\Big(\sup_{\substack{x \in \R\\ N^{-2}\ll  y \ll 1}}\int^{\ast}_{(N)} |\varpi(x+iy)|^2  \frac{{\rm d} \varpi}{1 + \mu_{\varpi}^B}\Big)^{1/2},
\]
since by Weyl's law $\int^{\ast}_{(N)} (1 + |\mu_{\varpi}|)^{-B} \, {\rm d} \varpi \ll N^{1+\varepsilon}$ for $B>2$. Note that $\mu \mapsto 1/(1 + \mu^B)$ is holomorphic for $\Im \mu < 1$ and positive for $\mu\in \R \cup [-i/2, i/2]$. The Harish-Chandra transform of this function is rapidly decaying. We apply the pretrace formula to the remaining spectral term, getting the bound
\[
\int^{\ast}_{(N)} | \langle E_\Psi \overline{E}_\Psi^{(N)}, \varpi \rangle |  \frac{{\rm d} \varpi}{(1 + |\mu_{\varpi}|)^{B} }\ll_C N^{3/2+\varepsilon}\Big(\sup_{\substack{x \in \R\\ N^{-2}\ll  y \ll 1}} \sum_{\gamma \in \Gamma_0(N)} \Big\| \Big(\begin{matrix}y & x \\ & 1\end{matrix}\Big)^{-1} \gamma \Big(\begin{matrix}y & x \\ & 1\end{matrix}\Big) \Big\|^{-C} \Big)^{1/2}.
\]
If we write $\gamma=\begin{psmallmatrix} a & b\\ c & d \end{psmallmatrix}$, then 
\[
\Big(\begin{matrix}y & x \\ & 1\end{matrix}\Big)^{-1} \gamma \Big(\begin{matrix}y & x \\ & 1\end{matrix}\Big) = \Big( \begin{matrix} a - cx & * \\ c y & d + c x\end{matrix}\Big).
\]
For $y \gg N^{-2}$,  $N \mid c$ and $\det \gamma = 1$ (which determines $b$ once $a, c, d$ are chosen), it is not hard to see that the sum over $\gamma$ is (for $C$ sufficiently large) bounded by $O(N)$, so that \eqref{spectral-error} is $\ll D^{-\delta} N^{A+2+\varepsilon}$. 

We now turn to the evaluation of the contribution of the constant function, given by
\[
\frac{1 }{V(N)} \int_{Y_0(N)} E_\Psi(z)  \overline{E}_\Psi(Nz) \frac{3}{\pi}\frac{{\rm d} x \, {\rm d} y}{y^2}.
\]
Again replacing the integral over $Y_0(N)$ with $V(N)$ translates of $Y_0(1)$, the previous display equals
\[
\frac{1 }{ V(N)} \overline{\int_{Y_0(1)} \overline{E}_\Psi(z)    {E}_\Psi(z)|_{\SL_2(\Z) \begin{psmallmatrix} N & \\ & 1\end{psmallmatrix}\SL_2(\Z)}\frac{3}{\pi}\frac{{\rm d} x \, {\rm d} y}{y^2}},
\]
where the slash-operator denotes the action of the double coset operator which we denote by $\sqrt{N} \cdot T_{N, 1}^{\ast}$. The eigenvalue of $T_{N, 1}^{\ast}$ for the Eisenstein series $E(z, s)$ is $\tau^{\ast}_{s - 1/2}(N)$. Then for $\sigma > 1$ we have
\[
\sqrt{N}\cdot  T_{N, 1}^{\ast} E_\Psi(z) = \sqrt{N}  \int_{(\sigma)} T_{N, 1}^{\ast} E(z, s) \hat{\Psi}(-s) \frac{{\rm d}s}{2\pi i}  =  \sqrt{N}  \int_{(\sigma)} \tau_{s - 1/2}^{\ast}(N) E(z, s)\hat{\Psi}(-s) \frac{{\rm d}s}{2\pi i}.
\]
By unfolding we obtain\begin{displaymath}
\begin{split}
\frac{\sqrt{N} }{ V(N)} \overline{\int_0^{\infty} \overline{\Psi}(y) \int_0^1   \int_{(\sigma)} \tau_{s - 1/2}^{\ast}(N) E(x+iy, s)\hat{\Psi}(-s) \frac{{\rm d}s}{2\pi i}  \, \frac{3}{\pi}\frac{{\rm d} x\, {\rm d} y}{y^2}}.
\end{split}
\end{displaymath}
The constant term computation yields
\[
\int_0^{\infty} \overline{\Psi}(y) \int_0^1 E(x+iy, s) \frac{3}{\pi}\frac{{\rm d} x\, {\rm d} y}{y^2} =  \int_0^{\infty}  \overline{\Psi}(y)  (y^s + \phi(s) y^{1-s}) \frac{3}{\pi}\frac{{\rm d} y}{y^2}  = \frac{3}{\pi}\left(\hat{\overline{\Psi}}(s-1) + \phi(s)  \hat{\overline{\Psi}}(-s)\right),
\]
which completes the proof of Theorem \ref{thm12} in the Eisenstein case upon noting that $\tau^{\ast}_{it}(N) = \tau^{\ast}_{-it}(N) = \overline{\tau^{\ast}_{it}(N)}$, $\overline{\hat{\bar{\Psi}}(-1/2 -it)} = \hat{\Psi}(-1/2 + it)$ and $\overline{\phi(1/2 - it)} = \phi(1/2 + it)$. \qed

\section{Proof of Theorem \ref{thm:unconditional} : the equivariant case}\label{sec:mollification}

The goal of this section is to prove part \eqref{boundmoment} of Theorem \ref{thm:unconditional}, where $\nu=1$. Recall that in this case, the hypothesis of Theorem \ref{thm11} is that $\bB_1\not\simeq\bB_2$. Without loss of generality, we may assume that $\bB_1\not\simeq {\rm Mat}_2$, so that $f_1$ is cuspidal.

\subsection{Heuristic}\label{sec:heuristic2}
The initial observation is that one can first introduce complex weights, or \textit{mollifiers}, to make the two factors in \eqref{parseval} almost parallel before applying the Cauchy--Schwarz inequality. The weights will then remember the correlation between the two factors in the subsequent analysis. 

In our situation, ``almost parallel" will be too much to ask for, since we do not have access to the phase of the $\chi$-twisted Heegner periods $W_{\mathscr{D}_i}(f_i,\chi)$, which, in particular, depends on the choice of base point $x_i$. Instead, inspired by \cite{RS}, we insert \textit{positive real-valued} mollifiers $\mathcal{M}(\chi)^{\pm 1/2}\in\R_{>0}$ to make the two factors roughly of the same size in absolute value. An application of Cauchy--Schwarz then yields
\[
P_{\mathscr{D}_1,\mathscr{D}_2}^{\Delta}(f_1\otimes f_2)\leq \sum_{\chi\in \ClD^\wedge}\left(|W_{\mathscr{D}_1}(f_1,\chi)|^2 \mathcal{M}(\chi)\right)^{1/2}\left(|W_{\mathscr{D}_2}(f_2,\chi)|^2 \mathcal{M}(\chi)^{-1}\right)^{1/2}.
\]
The better $\mathcal{M}(\chi)$ approximates $|W_{\mathscr{D}_2}(f_2,\chi)|/|W_{\mathscr{D}_1}(f_1,\chi)|$, the less information we have lost; we are sacrificing any control on the correlation of the phases. See \cite{CCLR} for a recent variation.

The construction of such mollifiers is guided by the Waldspurger formula, which relates $|W_{\mathscr{D}_i}(f_i,\chi)|^2$ (for $f_i$ cuspidal) to $L$-functions, as in Remark \ref{reciprocity}. The essential information that we extract from this formula is that $|W_{\mathscr{D}_i}(f_i,\chi)|^2$ should behave like a short Euler product. We emphasize, however, that it is only in this auxiliary choice of weight that $L$-functions, or approximations to them, play a role: we leave untouched the $\chi$-twisted Heegner periods $|W_{\mathscr{D}_i}(f_i,\chi)|^2$, in anticipation of an application of the Parseval relation \eqref{momentM1}, and the subsequent use of spectral theory.

We shall take two different mollifiers, according to whether both $f_1$ and $f_2$ are cuspidal, or only $f_1$ is cuspidal and $f_2=E_\Psi$. This corresponds with the subdivision into symmetric or asymmetric situations of Section \ref{secoff}. In the symmetric case, we take for the mollifier $\mathcal{M}(\chi)$ an expression of the form
\[
 \exp\Big( -\frac{1}{2} \sum_{C_0\leq N\mathfrak{p} \leq D^c} \frac{\lambda_{\pi_1}(N\mathfrak{p})\chi(\mathfrak{p})}{N\mathfrak{p}^{1/2}} +  \frac{1}{2} \sum_{C_0\leq N\mathfrak{p} \leq D^c} \frac{\lambda_{\pi_2}(N\mathfrak{p})\chi(\mathfrak{p})}{N\mathfrak{p}^{1/2}}\Big)
\]
for suitable constants $C_0,c>0$. From Waldspurger's formula, applied to both $\chi$-twisted Heegner periods $W_{\mathscr{D}_1}(f_1,\chi)$ and $W_{\mathscr{D}_2}(f_2,\chi)$, as well as an approximation argument \cite[Lemma 6]{BB}, it follows that the above Euler product should act as an effective approximation to $|W_{\mathscr{D}_2}(f_2,\chi)|/|W_{\mathscr{D}_1}(f_1,\chi)|$. In fact, to make use of this substitute, care must be taken to work entirely with Dirichlet polynomials, which leads to the use of truncated exponentials as well as a multi-piece mollifier to deal with the truncation error at different ideal-norm scales. 

In the asymmetric situation, when $f_2=E_\Psi$, there is a price to pay for working with incomplete Eisenstein series:  the Fourier coefficients are not multiplicative and therefore not suitable for mollification. We therefore change our mollification scheme and take for $\mathcal{M}(\chi)$ an expression of the form
\[
\exp\Big( -\alpha \sum_{C_0\leq N\mathfrak{p} \leq D^c} \frac{\lambda_{\pi_1}(N\mathfrak{p})\chi(\mathfrak{p})}{N\mathfrak{p}^{1/2}} \Big),
\]
for some parameter $0 < \alpha < 1/2$. By the Waldspurger formula, this serves as an approximation to $|W_{\mathscr{D}_1}(f_1,\chi)|^{-2\alpha}$. Arguing as above, we will ultimately produce a bound expressed in terms of the Dirichlet polynomial $T_D(\pi_1)$, defined in Section \ref{secoff}. As was mentioned in Remark \ref{rmks-upper-bds}, we would like to take $\alpha = 1/2$ and obtain half the savings as in the cuspidal case, reflecting the fact that our mollifier carries only half the information. The proof of Lemma \ref{calpha}, however, requires the flexibility of choosing $\alpha$ a bit smaller.

We now outline the contents of this section. In Sections \ref{mollifier-intro}-\ref{sec:mollifier-parameters}, we carefully choose parameters in our mollifier. We then apply in Section \ref{sec:BBK-input} the twisted first moments of Theorem \ref{thm12}. Finally, since the majorant in Theorem \ref{thm:unconditional} is a short Euler product, it will remain to convert our estimates from short Dirichlet polynomials to short Euler products. We do this in Sections \ref{sec-SEP}-\ref{sec:end-of-proof}, most notably by employing Rankin's trick.

\subsection{Introduction of mollifiers}\label{mollifier-intro}
For a non-negative integer $\ell$ and a real number $x$ we let
\[
E_\ell(x)=\sum_{j=0}^\ell \frac{x^j}{j!}
\]
denote the truncated exponential. While the functions $E_\ell$ are no longer multiplicative, the following lower bound provides a useful substitute.

\begin{lemma}\label{lemma:e-E}
Let $\ell$ be a non-negative even natural number. Then
\[
E_{\ell}(x)E_{\ell}(-x)\geq 1
\]
for all $x\in\R$.
\end{lemma}
\begin{proof}
The product $E_{\ell}(x)E_{\ell}(-x)$ is clearly an even polynomial function of $x$, of degree $2\ell$, with constant term $1$, regardless of the parity of the integer $\ell$. To investigate the signs of the coefficients we write
\[
E_{\ell}(x)E_{\ell}(-x)=\sum_{n=0}^\ell a_{2n} x^{2n},\quad\textrm{where}\quad a_{2n}=\sum_{\substack{j+k=2n\\ j,k\leq \ell}} \frac{(-1)^k}{j!k!}.
\]
Observe that $a_{2n}=(1-1)^{2n}=0$ when $2n\leq \ell$. The remaining coefficients are of the form
\[
a_{\ell+t}=\sum_{\nu=0}^{\ell-t} \frac{(-1)^{t+\nu}}{(\ell-\nu)!(t+\nu)!}\qquad (1\leq t\leq \ell,\;\; t\equiv \ell\!\!\!\!\mod 2),
\]
where we have put $\nu=\ell-j=t-k$. By telescoping the binomial coefficients, the above expression equals
\begin{align*}
\frac{1}{(t+\ell)!}\sum_{\nu=0}^{\ell-t}& \frac{(-1)^{t+\nu}(t+\ell)!}{(\ell-\nu)!(t+\nu)!}=\frac{1}{(t+\ell)!}\sum_{\nu=0}^{\ell-t} (-1)^{t+\nu}\binom{t+\ell}{\ell-\nu}\\
&=\frac{1}{(t+\ell)!}\sum_{\nu=0}^{\ell-t} (-1)^{t+\nu}\left\{\binom{t+\ell-1}{\ell-\nu}+\binom{t+\ell-1}{\ell-\nu-1}\right\}\\
&=\frac{1}{(t+\ell)!}\left\{(-1)^t\binom{t+\ell-1}{\ell}+(-1)^\ell\binom{t+\ell-1}{t-1}\right\}=\frac{(-1)^t+(-1)^\ell}{(t+\ell)(t-1)!\ell!}.
\end{align*}
Now $\ell$ and $t$ have the same parity, so this coefficient is positive or negative according to whether $\ell$ is even or odd, respectively.
\end{proof}

For a collection of Dirichlet polynomials $\{\mathcal{P}(\chi)\}_\chi$ we introduce the weighted sum
\begin{equation}\label{mollified-first-moment}
M_D(f_i,\mathcal{P})=\sum_{\chi\in \ClD^\wedge}|W_{\mathscr{D}_i}(f_i,\chi)|^2\mathcal{P}(\chi).
\end{equation}
When expressed in terms of $L$-functions, similarly to \eqref{eq:twisted-first-moment}, $M_D(f_i,\mathcal{P})$ is seen to be the mollified first moment.

\begin{corollary}\label{cor:CS}
For $R\in\N$ let $\{\mathcal{Q}_1(\chi),\ldots ,\mathcal{Q}_R(\chi)\}_{\chi}$ be a collection of $R$-tuples of real-valued Dirichlet polynomials, indexed by $\chi\in \ClD^\wedge$. For $i=1,2$ and $\ell_1,\ldots ,\ell_R\in\N$ let 
\[
\mathcal{P}^{(i)}(\chi)=\prod_{j=1}^R E_{\ell_j}((-1)^i \mathcal{Q}_j(\chi)).
\]
Then
\[
P_{\mathscr{D}_1,\mathscr{D}_2}^{\Delta}(f_1\otimes f_2) \leq M_D(f_1,\mathcal{P}^{(1)})^{1/2} M_D(f_2,\mathcal{P}^{(2)})^{1/2}.
\]
\end{corollary}
\begin{proof}
From Lemma \ref{lemma:e-E} it follows that $\mathcal{P}^{(1)}(\chi)\mathcal{P}^{(2)}(\chi)\geq 1$. Hence
\[
P_{\mathscr{D}_1,\mathscr{D}_2}^{\Delta}(f_1\otimes f_2)\leq \sum_{\chi\in \ClD^\wedge}\left(|W_{\mathscr{D}_1}(f_1,\chi)|^2\mathcal{P}^{(1)}(\chi)\right)^{1/2}\left(|W_{\mathscr{D}_2}(f_2,\chi)|^2\mathcal{P}^{(2)}(\chi)\right)^{1/2}.
\]
We conclude by the Cauchy--Schwarz inequality. 
\end{proof}

\subsection{Construction of mollifiers}\label{sec:mollifier-const}
Let $R$ be a positive integer. Let $P_1,\ldots ,P_R$ be disjoint finite sets of primes, which we assume are all split in $E$. From now on we assume in addition when both $f_1, f_2$ are cuspidal (the ``symmetric situation" of Section \ref{secoff}) that all sets $P_j$ consist only of primes $p$ satisfying $\max(|\lambda_{\pi_1}(p)|, |\lambda_{\pi_2}(p)|) \leq  B$ for a suitable (fixed) constant $B$ to be chosen later, while when $f_2=E_\Psi$ (the ``asymmetric situation") all sets $P_j$ consist only of primes $p$ satisfying $ |\lambda_{\pi_1 }(p)| \leq B$. Let $P_j^E$ denote the sets of prime ideals of $\mathscr{O}_E$ for which ${\rm N}\mathfrak{p}\in P_j$. Let $I_j^E$ denote the integral ideals of $\mathscr{O}_E$ supported on prime ideals in $P_j^E$.  We write $I_j=\{{\rm N}\mathfrak{n}: \mathfrak{n}\in I_j^E\}$ for the positive integers supported on primes $p\in P_j$. We will adopt the notational convention of writing $\mathfrak{n}_j$ (resp.\ $n_j$) for an arbitrary element in $I_j^E$ (resp.\ $I_j$).

For an integral ideal $\mathfrak{n}\subset\mathscr{O}_E$, we let $\omega(\mathfrak{n})$ denote the (not completely) multiplicative function defined on powers of prime ideals by $\omega(\mathfrak{p}^\alpha)=\alpha!$. Let $\Omega(\mathfrak{n})$ denote the number of prime ideals appearing in the decomposition of $\mathfrak{n}$, counted with multiplicity. Let
\begin{equation}\label{defn:ab}
a_{\pi_1,\pi_2}(p)=\lambda_{\pi_1}(p)-\lambda_{\pi_2}(p).
\end{equation}
Write $\tilde{a}_{\pi_1,\pi_2}(n)$ for the totally multiplicative function defined at primes by $a_{\pi_1,\pi_2}(p)$, and  $\tilde{a}_{\pi}(n)$ for the totally multiplicative function defined at primes by $\lambda_{\pi}(p)$.

For a function $F: I_j^E\rightarrow \C$ (such as the character $\chi(\mathfrak{n})$ or the normalized Hecke eigenvalues $\Lambda_\pi(\mathfrak{n})$ or $\Lambda_{it}(\mathfrak{n})$, defined in \eqref{Lambda-pi}) we put
\begin{equation}\label{def:D1}
\mathcal{P}_j^{(i)}(F)=\sum_{\substack{\mathfrak{n}\in I_j^E\\ \Omega(\mathfrak{n})\leq \ell_j}}\frac{\tilde{a}_{\pi_1,\pi_2}({\rm N}\mathfrak{n})}{\sqrt{{\rm N}\mathfrak{n}}}\left(\frac{(-1)^i}{2}\right)^{\Omega(\mathfrak{n})} \frac{1}{\omega(\mathfrak{n})}F(\mathfrak{n})
\end{equation}
and, for a parameter $\alpha\in (0,1)$, we put
\begin{equation}\label{def:D2}
\mathcal{P}_{j,\alpha}^{(i)}(F)=\sum_{\substack{\mathfrak{n}\in I_j^E\\ \Omega(\mathfrak{n})\leq \ell_j}}\frac{\tilde{a}_{\pi}({\rm N}\mathfrak{n})}{\sqrt{{\rm N}\mathfrak{n}}} ((-1)^i \alpha )^{\Omega(\mathfrak{n})} \frac{1}{\omega(\mathfrak{n})}F(\mathfrak{n}).
\end{equation}
Finally, put $\mathcal{P}^{(i)}(F)=\prod_{j=1}^R\mathcal{P}_j^{(i)}(F)$ and $\mathcal{P}_{\alpha}^{(i)}(F)=\prod_{j=1}^R\mathcal{P}_{j,\alpha}^{(i)}(F)$. 

The following lemma should be compared to the discussion in Section \ref{sec:heuristic2}.

\begin{lemma}\label{lem:short-Dirichlet}
For a class group character $\chi$ of $\ClD$ and $j=1,\ldots ,R$, we have
\[
\mathcal{P}_j^{(i)}(\chi)=E_{\ell_j}\big((-1)^i\mathcal{Q}_j(\chi)\big)\quad\textrm{and}\quad \mathcal{P}_{j,\alpha}^{(i)}(\chi)=E_{\ell_j} \big((-1)^i \mathcal{Q}_{j,\alpha}(\chi)\big),
\]
where
\[
\mathcal{Q}_j(\chi)=\frac{1}{2}\sum_{\mathfrak{p}\in P_j^E}\frac{a_{\pi_1,\pi_2}({\rm N}\mathfrak{p})\chi(\mathfrak{p})}{{\rm N}\mathfrak{p}^{1/2}}\quad\textrm{and}\quad \mathcal{Q}_{j,\alpha}(\chi)=\alpha\sum_{\mathfrak{p}\in P_j^E}\frac{\lambda_{\pi}({\rm N}\mathfrak{p})\chi(\mathfrak{p})}{{\rm N}\mathfrak{p}^{1/2}}.
\]
\end{lemma}

\begin{proof} Note that the Dirichlet polynomials $\mathcal{Q}_j$, $\mathcal{Q}_{j,\alpha}$ are real-valued since the coefficients are real and depend only on ${\rm N}\mathfrak{p}$. 
With the above notation we compute
\begin{align*}
E_{\ell_j}\Big((-1)^i \mathcal{Q}_j(\chi)\Big)&=E_{\ell_j}\bigg((-1)^i\frac{1}{2} \sum_{\mathfrak{p}\in P_j^E}\frac{a_{\pi_1,\pi_2}({\rm N}\mathfrak{p})\chi(\mathfrak{p})}{\sqrt{{\rm N}\mathfrak{p}}}\bigg)\\
&=\sum_{f=0}^{\ell_j} \frac{1}{f!}\bigg(\frac{(-1)^i}{2}\sum_{\mathfrak{p}\in P_j^E}\frac{a_{\pi_1,\pi_2}({\rm N}\mathfrak{p}) \chi(\mathfrak{p})}{\sqrt{{\rm N}\mathfrak{p}}}\bigg)^f\\
&=\sum_{f=0}^{\ell_j} \frac{1}{f!}\left(\frac{(-1)^i}{2}\right)^f\sum_{\substack{\mathfrak{n}\in I_j^E\\\Omega(\mathfrak{n})= f}}\frac{f!}{\omega(\mathfrak{n})}\frac{\tilde{a}_{\pi_1,\pi_2}({\rm N}\mathfrak{n})\chi(\mathfrak{n})}{\sqrt{{\rm N}\mathfrak{n}}},
\end{align*}
where we have used the multinomial theorem. Switching the order of summation then proves the first identity of the lemma, and the second is proved similarly.  
\end{proof}

\subsection{Choosing parameters}\label{sec:mollifier-parameters}
Let $c>0$ be a small absolute constant. We shall suppose that
\begin{equation}\label{c-size}
 0<c\leq\min(1/16,  (\delta/4A)^2),
\end{equation}
where $A\geq 1/2$ majorizes the values of $A$ and $\delta\in (0,1/4]$ minorizes the values of $\delta$ in Theorem \ref{thm12}. Moreover, we assume
\begin{equation}\label{c-size2}
(70 B^2 \log (c^{-1/2}))^2 \leq 1/c^{1/2}.
\end{equation}

We shall make the following choice of parameters, which will depend on $c>0$ and another fixed constant $C_0 > 0$ that excludes the first few primes.

\begin{definition}\label{def-parameters}
For $-D$ a large fundamental discriminant, we let
\medskip
\begin{enumerate}
\item\label{ell-defn} $\{\ell_j\}_{j\geq 1}$ be defined by $\ell_1=2\lfloor 30 B^2 \log\log D\rfloor$, and $\ell_{j+1}=2\lfloor 30 B^2 \log \ell_j\rfloor$ for $j\geq 1$;
\vspace{0.1cm}
\item\label{R-size} $R$ be the largest natural number with $\ell_R>1/\sqrt{c}$;
\vspace{0.1cm}
\item\label{Pj-interval} $P_1$ be the set of primes, split in $E$, lying in the interval $[C_0, D^{1/\ell_1^2}]$, and for $2\leq j\leq R$ let $P_j$ be the set of primes, split in $E$, lying in the interval $(D^{1/\ell_{j-1}^2},D^{1/\ell_j^2}]$, where we always restrict to primes satisfying $\max(|\lambda_{\pi_1}(p) |, |\lambda_{\pi_2}(p)|) \leq B$ in the definition of $\mathcal{Q}_j$, resp.\  $|\lambda_{\pi}(p)|\leq B$ in the definition of $\mathcal{Q}_{j,\alpha}$. 
\end{enumerate}
\end{definition}
We will choose $C_0 > B$ sufficiently large, but fixed, so that all ramified primes of $\pi_1, \pi_2$ are below $C_0$. 

\begin{lemma}\label{lem:ell-bds} Let $\{\ell_j\}_{j\geq 1}$ be the sequence of positive even integers given by Definition \ref{def-parameters}.\eqref{ell-defn}.
\begin{enumerate}
\item We have $\ell_j>\ell_{j+1}^2$. In particular, the sequence $\ell_j$ is monotone decreasing for such $j$ and, for $D$ sufficiently large, the integer $R$ in Definition \ref{def-parameters}.\eqref{R-size} is well-defined. 
\item\label{ell-sum} We have $1/\ell_1+\cdots +1/\ell_R<2\sqrt{c}$.
\end{enumerate}
\end{lemma}

\begin{proof}
The first claim follows from \eqref{c-size2}. 

It follows from this and Definition \ref{def-parameters}.\eqref{R-size}, that
\[
1/\ell_R+\cdots +1/\ell_1\leq \ell_R^{-1}+\sum_{j=1}^{R-1} \ell_R^{-2j}=\ell_R^{-1}\left(1+\frac{1-\ell_R^{-2(R-1)}}{\ell_R-\ell_R^{-1}}\right)<\sqrt{c}\left(1+\frac{1}{c^{-1/2}-1}\right).
\]
Since $c\in (0,1/4]$ (which follows from assumption \eqref{c-size}), this is $<2\sqrt{c}$.\end{proof}


\begin{lemma}\label{lem:Dir-length}
The Dirichlet polynomials $\mathcal{P}^{(i)}(\chi), \mathcal{P}_{\alpha}^{(i)}(\chi)$ are of length at most $D^{2\sqrt{c}}$.
\end{lemma}
\begin{proof}
From Lemma \ref{lem:short-Dirichlet} and Definition \ref{def-parameters}.\eqref{Pj-interval} it follows that $\mathcal{P}_j^{(i)}(\chi)=E_{\ell_j}((-1)^i \mathcal{Q}_j(\chi))$ is of length at most $(D^{1/\ell_j^2})^{\ell_j}=D^{1/\ell_j}$. Thus $\mathcal{P}^{(i)}(\chi)$ is of length at most $D^{1/\ell_1+\cdots +1/\ell_R}$, which, by Lemma \ref{lem:ell-bds}.\eqref{ell-sum}, is bounded by $D^{2\sqrt{c}}$. The result for $\mathcal{P}_{\alpha}^{(i)}(\chi)$ is similar. 
\end{proof}

\subsection{Asymptotic evaluation of mollified moments}\label{sec:BBK-input}

We now use Theorem \ref{thm12} to asymptotically evaluate the mollified first moments 
\[
M_D(f_i,\mathcal{P}^{(i)}) \quad (i=1,2),\quad M_D(f_1,\mathcal{P}_{\alpha}^{(1)}), \quad M_D(E_\Psi,\mathcal{P}_{\alpha}^{(2)}),
\]
defined in \eqref{mollified-first-moment}. Doing so will effectively replace the function $\chi(\mathfrak{n})$ with the normalized Hecke eigenvalues $\Lambda_{\pi_i}(\mathfrak{n})$ or $\Lambda_{it}(\mathfrak{n})$ in the Dirichlet polynomials \eqref{def:D1}-\eqref{def:D2}.

\begin{proposition}\label{cor:asymp-M1}
There is $\beta>0$ such that
\begin{displaymath}
\begin{split}
M_D(f_i,\mathcal{P}^{(i)})&= \mathcal{P}^{(i)}(\Lambda_{\pi_i}) +O(D^{-\beta})\qquad (i=1,2),\\
M_D(f_1,\mathcal{P}_{\alpha}^{(1)})&=  \mathcal{P}_{\alpha}^{(1)}(\Lambda_\pi) +O(D^{-\beta}),\\
M_D(E_\Psi,\mathcal{P}_{\alpha}^{(2)})&= \int_{\R}  \mathcal{P}_{\alpha}^{(2)}(\Lambda_{it}) \Phi(t) \frac{{\rm d} t}{2\pi} +O(D^{-\beta}),
\end{split}
\end{displaymath}
where $\Phi$ is defined in \eqref{def:Phi}.
\end{proposition}

\begin{proof}
Recall that 
\[
M_D(f_i,\mathcal{P}^{(i)})=\sum_{\chi\in \ClD^\wedge}|W_{\mathscr{D}_i}(f_i,\chi)|^2\prod_{j=1}^R E_{\ell_j}\big((-1)^i \mathcal{Q}_j(\chi)\big).
\]
Lemma \ref{lem:short-Dirichlet} then shows that
\[
M_D(f_i,\mathcal{P}^{(i)})=\sum_{\substack{\mathfrak{n}_1\in I_1^E\\ \Omega(\mathfrak{n}_1)\leq \ell_1}} \cdots \sum_{\substack{\mathfrak{n}_R\in I_R^E\\ \Omega(\mathfrak{n}_R)\leq \ell_R}}\frac{\tilde{a}_{\pi_1,\pi_2}({\rm N}\mathfrak{n}_1 \cdots \mathfrak{n}_R)}{\sqrt{{\rm N}\mathfrak{n}_1 \cdots \mathfrak{n}_R}}\left(\frac{(-1)^i}{2}\right)^{\Omega(\mathfrak{n}_1\cdots\mathfrak{n}_R)}\frac{P_{\mathscr{D}_i}^\Delta(f_i\otimes f_i;\mathfrak{n}_1\cdots \mathfrak{n}_R)}{\omega(\mathfrak{n}_1)\cdots \omega(\mathfrak{n}_R)}. 
\]
An application of Theorem \ref{thm12} then yields the main term 
$$\sum_{\substack{\mathfrak{n}_1\in I_1^E\\ \Omega(\mathfrak{n}_1)\leq \ell_1}} \cdots \sum_{\substack{\mathfrak{n}_R\in I_R^E\\ \Omega(\mathfrak{n}_R)\leq \ell_R}}\frac{\tilde{a}_{\pi_1,\pi_2}({\rm N}\mathfrak{n}_1 \cdots \mathfrak{n}_R)}{\omega(\mathfrak{n}_1)\cdots \omega(\mathfrak{n}_R)\sqrt{{\rm N}\mathfrak{n}_1 \cdots \mathfrak{n}_R}}\left(\frac{(-1)^i}{2}\right)^{\Omega(\mathfrak{n}_1\cdots\mathfrak{n}_R)}\Lambda_{\pi_i}(\mathfrak{n}_1\cdots \mathfrak{n}_R)$$
 plus an error of the form
\[
O\bigg(D^{-\delta}\sum_{N\mathfrak{n} \leq D^{2\sqrt{c}} } N\mathfrak{n}^A \bigg)=O(D^{-\delta+2A\sqrt{c}}),
\]
where we have used the trivial bounds for the coefficients of the Dirichlet polynomial $\mathcal{P}_i(\chi)$  as well as the bound on the length of the Dirichlet polynomial in Corollary \ref{lem:Dir-length}. Note that the exponent is $-\delta+2A\sqrt{c}\leq -\delta/2$, by \eqref{c-size}. The main term factorizes since $\mathfrak{n}_{j_1} \overline{\mathfrak{n}}_{j_1}$ and $\mathfrak{n}_{j_1} \overline{\mathfrak{n}}_{j_2}$ are coprime for different indices $j_1, j_2$, and we obtain the desired main term. 

The other two formulas follow in the same way. 
\end{proof}

\subsection{Short Euler products}\label{sec-SEP}

The restriction according to the number of prime factors $\Omega(n)$ in the Dirichlet polynomials \eqref{def:D1}-\eqref{def:D2} makes it difficult to compare the various formulae in Proposition \ref{cor:asymp-M1} with the majorant in Theorem \ref{thm:unconditional}, which, by contrast, is a short Euler product. 

With this in mind, for  $i\in\{1,2\}$, $1\leq j\leq R$, we let
\[
\mathcal{E}_j^{(i)}= \sum_{ \mathfrak{n}_j\in I_j^E }  \frac{\tilde{a}_{\pi_1,\pi_2}({\rm N}\mathfrak{n}_j)}{\omega(\mathfrak{n}_j)\sqrt{{\rm N}\mathfrak{n}_j}}\left(\frac{(-1)^i}{2}\right)^{\Omega(\mathfrak{n}_j)}\Lambda_{\pi_i}(\mathfrak{n}_j),
\]
and
\[
\mathcal{E}_{j,\alpha}= \sum_{ \mathfrak{n}_j\in I_j^E }  \frac{\tilde{a}_{\pi}({\rm N}\mathfrak{n}_j)}{\omega(\mathfrak{n}_j)\sqrt{{\rm N}\mathfrak{n}_j}}(-\alpha)^{\Omega(\mathfrak{n}_j)}\Lambda_\pi(\mathfrak{n}_j).\]
Moreover, we put $\mathcal{E}^{(i)}=\prod_{j=1}^R\mathcal{E}_j^{(i)}$ and $\mathcal{E}_{\alpha}=\prod_{j=1}^R\mathcal{E}_{j,\alpha}$. These extended sums, in which the condition on $\Omega(n)$ has been removed, can be more easily compared to the majorant in Theorem \ref{thm:unconditional}, as the following lemma attests.

\begin{lemma}\label{lemma-EP} For $1 \leq j \leq R$ we have 
\begin{displaymath}
\begin{split}
\mathcal{E}_j^{(1)}\mathcal{E}_j^{(2)}& = \prod_{p \in P_j} \Big(1 -  \frac{ ( \lambda_{\pi_2}(p) - \lambda_{\pi_1}(p))^2}{2p} + O\Big(\frac{1}{p^2}\Big)\Big), \\
 \mathcal{E}_{j,\alpha} &= \prod_{p \in P_j} \Big(1 + \frac{(\alpha^2 - 2 \alpha) \lambda_{\pi}(p)^2}{p} + O\Big(\frac{1}{p^2}\Big)\Big). 
 \end{split}
 \end{displaymath}
with an absolute implied constant. 

\end{lemma}

\begin{proof} Expressing the sums in terms of Euler products we have that $\mathcal{E}_j^{(1)}\mathcal{E}_j^{(2)}$ is
\begin{displaymath}
\begin{split}
&\sum_{ \mathfrak{n}_j\in I_j^E }  \sum_{ \mathfrak{m}_j\in I_j^E }  \frac{\tilde{a}_{\pi_1,\pi_2}({\rm N}\mathfrak{n}_j)}{\omega(\mathfrak{n}_j)\sqrt{{\rm N}\mathfrak{n}_j}} \frac{(-1)^{\Omega(\mathfrak{n}_j)}}{2^{\Omega(\mathfrak{n}_j\mathfrak{m}_j)}}  \Lambda_{\pi_1}(\mathfrak{n}_j) \frac{\tilde{a}_{\pi_1,\pi_2}({\rm N}\mathfrak{m}_j)}{\omega(\mathfrak{m}_j)\sqrt{{\rm N}\mathfrak{m}_j}}\Lambda_{\pi_2}(\mathfrak{m}_j)\\
& = \prod_{p \in P_j} \sum_{r_1, s_1, r_2, s_2 = 0}^{\infty}  \frac{\tilde{a}_{\pi_1, \pi_2}(p^{r_1+s_1+r_2+s_2}) \lambda^{\ast}_{\pi_1}(p^{|r_1-s_1|})\lambda^{\ast}_{\pi_2}(p^{|r_2-s_2|}) (-1)^{r_1+s_1}}{p^{ \max(r_1, s_1) + \max(r_2, s_2)}2^{r_1+s_1+r_2+s_2}  r_1!s_1!r_2!s_2! (1 + \delta_{r_1 \neq s_1} p^{-1})(1 + \delta_{r_2 \neq s_2} p^{-1})}.
\end{split}
\end{displaymath}
Here we decomposed ideals $\mathfrak{n}_j$ into prime powers $\mathfrak{p}^{r_1} \overline{\mathfrak{p}}^{s_1}$ and similarly $\mathfrak{m}_j$ into prime powers $\mathfrak{p}^{r_2} \overline{\mathfrak{p}}^{s_2}$
The contribution of $ \max(r_1, s_1) + \max(r_2, s_2) = 1$ is
\begin{displaymath}
\begin{split}
&-\frac{\tilde{a}_{\pi_1, \pi_2}(p) \lambda^{\ast}_{\pi_1}(p)}{2p}-\frac{\tilde{a}_{\pi_1, \pi_2}(p) \lambda^{\ast}_{\pi_1}(p)}{2p}+\frac{\tilde{a}_{\pi_1, \pi_2}(p^2)}{4p}\\
&\quad\quad\quad  +\frac{\tilde{a}_{\pi_1, \pi_2}(p) \lambda^{\ast}_{\pi_2}(p)}{2p}+\frac{\tilde{a}_{\pi_1, \pi_2}(p) \lambda^{\ast}_{\pi_2}(p)}{2p}+\frac{\tilde{a}_{\pi_1, \pi_2}(p^2)}{4p}  + O\Big(\frac{1}{p^2}\Big)\\
& =  \frac{\tilde{a}_{\pi_1, \pi_2}(p)( \lambda_{\pi_2}(p) - \lambda_{\pi_1}(p))}{p}+\frac{\tilde{a}_{\pi_1, \pi_2}(p^2)}{2p} + O\Big(\frac{1}{p^2}\Big) = - \frac{ ( \lambda_{\pi_2}(p) - \lambda_{\pi_1}(p))^2}{p}+\frac{\tilde{a}_{\pi_1, \pi_2}(p)^2}{2p} + O\Big(\frac{1}{p^2}\Big)\\
&= - \frac{ ( \lambda_{\pi_2}(p) - \lambda_{\pi_1}(p))^2}{2p}  + O\Big(\frac{1}{p^2}\Big)
\end{split}
\end{displaymath}
where we recall \eqref{defn:ab}. 
The contribution of $ \max(r_1, s_1) + \max(r_2, s_2) > 1$ can be bounded as above by
\[
\leq  \sum_{k + \ell \geq 2}  \frac{(2B^2)^k}{p^k k!}  \frac{(2B^2)^l}{p^l l!}  = \sum_{m = 2}^{\infty} \frac{(4B^2)^m}{p^m m!} \ll \frac{1}{p^2}.
\]
This concludes the proof of the first identity. 

By the same argument we have 
\begin{equation*}
\begin{split}
\mathcal{E}_{j,\alpha} &= \sum_{ \mathfrak{n}_j\in I_j^E }   \frac{\tilde{a}_{\pi}({\rm N}\mathfrak{n}_j) (-\alpha)^{\Omega(\mathfrak{n}_j)}}{\omega(\mathfrak{n}_j)\sqrt{{\rm N}\mathfrak{n}_j}} \Lambda_{\pi}(\mathfrak{n}_j) = \prod_{p \in P_j} \sum_{r, s}^{\infty}  \frac{\tilde{a}_{\pi}(p^{r+s})  (-\alpha)^{r+s}\lambda^{\ast}_{\pi}(p^{|r-s|}) }{p^{ \max(r, s) }  r !s ! (1 + \delta_{r  \neq s } p^{-1})}\\
&= \prod_{p \in P_j} \Big(1 + \frac{(\alpha^2 - 2 \alpha) \lambda_{\pi}(p)^2}{p} + O\Big(\frac{1}{p^2}\Big)\Big).\qedhere
\end{split}
\end{equation*} \end{proof}

\subsection{Bounding the tails}
Our goal is to replace the short Dirichlet polynomials appearing in the main terms of the asymptotic expressions in Proposition \ref{cor:asymp-M1} by the short Euler products of Lemma \ref{lemma-EP}. To do so, we must bound the error incurred in removing the condition $\Omega(\mathfrak{n})\leq\ell_j$.

For this reason, for $i\in\{1,2\}$, $1\leq j\leq R$, we define the tail sums
\[\mathcal{T}_j^{(i)}=\mathcal{E}_j^{(i)}-\mathcal{P}^{(i)}(\Lambda_{\pi_i}),\qquad
\mathcal{T}_{j,\alpha}= \mathcal{E}_{j, \alpha}  - \mathcal{P}_{j, \alpha}^{(1)}(\Lambda_\pi).
\]
With this notation, the products of the main terms which result from the application of Cauchy--Schwarz in Lemma \ref{cor:CS} become
 \[
\mathcal{P}^{(1)}(\Lambda_{\pi_1})\mathcal{P}^{(2)}(\Lambda_{\pi_2})
=\prod_{j=1}^R\big(\mathcal{E}_j^{(1)} \mathcal{E}_j^{(2)}  -\mathcal{T}_j^{(1)} \mathcal{E}_j^{(2)} - \mathcal{E}_j^{(1)} \mathcal{T}_j^{(2)} + \mathcal{T}_j^{(1)} \mathcal{T}_j^{(2)}\big),
\]
 while 
$$ \mathcal{P}_{\alpha}^{(1)}(\Lambda_\pi) \int_{\R} \mathcal{P}_{\alpha}^{(2)}(\Lambda_{it}) \Phi(t) \frac{{\rm d} t}{2\pi}=  \int_{\R}\prod_{j=1}^R\big(\mathcal{E}_{j, \alpha} - \mathcal{T}_{j, \alpha}\big) \mathcal{P}_{j, \alpha}^{(2)}(\Lambda_{it})\Phi(t) \frac{{\rm d} t}{2\pi}.$$

To control the various cross terms, we will need the following upper bounds.

\begin{lemma}\label{lem41} For $i =1, 2$, $1 \leq j \leq R$, $0 < \alpha < 1$, $t\in \R$ we have 
\begin{displaymath}
\begin{split}
&|\mathcal{E}_j^{(i)}|, |\mathcal{E}_{j, \alpha}|, |\mathcal{P}_{j, \alpha}^{(2)}(\Lambda_{it}) |  \leq   \exp\Big(\sum_{p \in P_j}\frac{2B^2}{p}\Big), \\
&|\mathcal{T}_j^{(i)}|,  |\mathcal{T}_{j, \alpha}|  \leq 2^{-\ell_j}  \exp\Big(\sum_{p \in P_j}\frac{8B^2}{p}\Big).
\end{split}
\end{displaymath}
\end{lemma}

\begin{proof} 
By our assumption we have $|\lambda^{\ast}_{\pi}(n)| \leq B^{\Omega(n)}$ and $|\tilde{a}_{\pi_1,\pi_2}(n)| \leq (2B)^{\Omega(n)}$, so that 
\begin{displaymath}
\begin{split}
|\mathcal{E}_j^{(i)}| &\leq  \sum_{ \mathfrak{n}_j\in I_j^E }  \frac{(2B)^{\Omega({\rm N}\mathfrak{n}_j)}}{\omega(\mathfrak{n}_j)\sqrt{{\rm N}\mathfrak{n}_j}} \frac{1}{2^{\Omega(\mathfrak{n}_j)}} \frac{ \sqrt{{\rm N}\mathfrak{n}^{\ast}_j}B^{\Omega( \mathfrak{n}^{\ast}_j)}}{V( {\rm N}\mathfrak{n}^{\ast}_j)} \leq \prod_{p \in P_j} \sum_{r, s = 0}^{\infty} \frac{B^{r+s}B^{|r-s|}}{p^{(r+s)/2 } p^{|r-s|/2} r!s!}\\ &= \prod_{p \in P_j} \sum_{r, s = 0}^{\infty} \frac{B^{2\max(r, s)}  }{p^{\max(r, s) }    r!s!}   = \prod_{p \in P_j} \sum_{k = 0}^{\infty} \frac{(2B^2)^{k}  }{p^k k! } =   \exp\Big(\sum_{p \in P_j}\frac{2B^2}{p}\Big).  
\end{split} \end{displaymath}
By the same argument and using in addition Rankin's trick, we have
\begin{displaymath}
\begin{split}
|\mathcal{T}_j^{(i)}| &\leq  \sum_{ \mathfrak{n}_j\in I_j^E }  \frac{(2B)^{\Omega({\rm N}\mathfrak{n}_j)}}{\omega(\mathfrak{n}_j)\sqrt{{\rm N}\mathfrak{n}_j}} \frac{1}{2^{\Omega(\mathfrak{n}_j)}} \frac{\sqrt{{\rm N}\mathfrak{n}^{\ast}_j} B^{\Omega( \mathfrak{n}^{\ast}_j)}}{V({\rm N}\mathfrak{n}^{\ast}_j)} \frac{2^{\Omega(\mathfrak{n}_j)}}{2^{\ell_j}}  \leq 2^{-\ell_j}  \exp\Big(\sum_{p \in P_j}\frac{8B^2}{p}\Big).  
\end{split} \end{displaymath}
Similarly we see
\begin{displaymath}
\begin{split}
|\mathcal{E}_{j, \alpha}| &\leq   \prod_{p \in P_j} \sum_{r, s = 0}^{\infty} \frac{(B\alpha)^{r+s} B^{|r-s|}}{p^{(r+s)/2 } p^{|r-s|/2} r!s!}\leq  \prod_{p \in P_j} \sum_{r, s = 0}^{\infty} \frac{B^{2\max(r, s)}  }{p^{\max(r, s) }    r!s!}     =   \exp\Big(\sum_{p \in P_j}\frac{2B^2}{p}\Big)
\end{split} \end{displaymath}
and the bound for $\mathcal{T}_{j, \alpha}$ follows in the same way by Rankin's trick. For the bounds for $\mathcal{P}_{j, \alpha}^{(2)}(\Lambda_{it})$ we simply note that $|\tau_{it}(N) | \leq  2 \leq B$, so the proof carries over verbatim. 
\end{proof}

The analysis of the integral of $\mathcal{P}_{\alpha}^{(2)}(\Lambda_{it})$ is somewhat different. We note that the following lemma would incur a fatal power of $\log D$ without the presence of the integral over $\R$. It is for this reason that we work with incomplete Eisenstein series $E_\Psi$, which feature this additional integral.

\begin{lemma}\label{calpha} We have
$$\int_{\R} \mathcal{P}_{\alpha}^{(2)}(\Lambda_{it}) \Phi(t) \frac{{\rm d} t}{2\pi}\ll \prod_{j=1}^R \prod_{p \in P_j} \Big(1 + \frac{\alpha^2 \lambda_{\pi}(p)^2}{p} + O\Big(\frac{1}{p^2}\Big)\Big)$$
provided that $\alpha < 1/4$.  
\end{lemma}

\begin{proof} Let $I^E$ denote the ideals supported on the union of the $P^E_j$, $1 \leq j \leq R$. Every ideal $\mathfrak{n} \in I^E$ can be decomposed uniquely as $\mathfrak{n} = \prod_{j=1}^R \mathfrak{n}_j$ with $\mathfrak{n}_j \in I_j^E$ and we denote by $\xi$ the characteristic function on the set of ideals $\mathfrak{n}\in I^E$ satisfying $\Omega(\mathfrak{n}_j) \leq \ell_j$ for all $1 \leq j \leq R$. 
Further decomposing $\mathfrak{n}$ into a primitive part and a principal ideal generated by a rational integer, we then have
 \begin{displaymath}
\begin{split}
\int_{\R}  \mathcal{P}_{\alpha}^{(2)}(\Lambda_{it})  \Phi(t) \, {\rm d} t &= \sum_{ \substack{\mathfrak{n} \in I^E\\ \mathfrak{n} \text{ primitive}} }   \frac{\tilde{a}_{\pi}({\rm N}\mathfrak{n}) \alpha^{\Omega({\rm N}\mathfrak{n})}}{V({\rm N}\mathfrak{n})\omega(\mathfrak{n})}  
\int_{\R} \tau^{\ast}_{it}({\rm N}\mathfrak{n}) 
\Phi(t) \frac{{\rm d} t}{2\pi}
 \sum_{ n \in \N, (n)\in I^E }   \frac{\tilde{a}_{\pi}(n)^2 \alpha^{\Omega(2n)}}{V(n)\omega((n))}  \xi(\mathfrak{n}(n))\\
 & \ll \sum_{ \substack{\mathfrak{n} \in I^E\\ \mathfrak{n} \text{ primitive}} }   \frac{|\tilde{a}_{\pi}({\rm N}\mathfrak{n}) |\alpha^{\Omega({\rm N}\mathfrak{n})}}{({\rm N}\mathfrak{n})\omega(\mathfrak{n})}  
\Big|\int_{\R} \tau^{\ast}_{it}({\rm N}\mathfrak{n})  \Phi(t) \frac{{\rm d} t}{2\pi} \Big|
 \sum_{ n \in \N, (n)\in I^E }   \frac{\tilde{a}_{\pi}(n)^2 \alpha^{\Omega(2n)}}{n\omega((n))}.
\end{split}
\end{displaymath}
We estimate this rather coarsely. The $n$-sum is bounded by 
\[
\prod_{j=1}^R \prod_{p \in P_j} \Big(1 + \frac{\alpha^2 \lambda_{\pi}(p)^2}{p} + O\Big(\frac{1}{p^2}\Big)\Big).
\]
To estimate the sum over ideals $\mathfrak{n}$, we write $N = {\rm N}\mathfrak{n}$ and observe that there  are at most $2^{\omega(N)} \leq 2^{\Omega(N)}$ primitive ideals of norm $N$. We recall \eqref{div}. 
Finally we express the integral over $t$ in terms of the Fourier transform $\widecheck{\Phi}$ of $\Phi$ which satisfies the bound $\widecheck{\Phi}(x) \ll_A (1 + |x|)^{-10}$ from partial integration. (In fact, we have $\widecheck{\Phi}(x) \ll e^{-2\pi |x|}$ since $\Phi$ is entire and hence we may shift the contour. We do not need this stronger bound.) We obtain the upper bound
\[
\underset{a, b, d}{\left.\sum\right.^{\ast}} \frac{|\tilde{a}_{\pi}(abd^2)|  (2\alpha)^{\Omega(abd^2)}}{abd^3} \Big|\widecheck{\Phi}\Big( \frac{1}{2\pi} \log\Big(\frac{a}{b}\Big)\Big)\Big| \cdot \prod_{j=1}^R \prod_{p \in P_j} \Big(1 + \frac{\alpha^2 \lambda_{\pi}(p)^2}{p} + O\Big(\frac{1}{p^2}\Big)\Big),
\]
where the asterisk indicates that we are summing over number $a, b$ composed of primes $p > C_0 > B$ such that $|\tilde{a}_{\pi}(p)| \leq B$ (we drop the fact that the primes have to be split). Recall that $\tilde{a}_{\pi}$ and $(2\alpha)^{\Omega}$ are completely multiplicative. In particular, we can immediately drop the sum over $d$. 
We now focus on the sum over $a, b$. In particular, we will see in a moment that this sum converges, at least for sufficiently small $\alpha$. This will complete the proof. 

 Let us fix $a$ for a moment and choose $z \geq 2/a$, then
 \begin{displaymath}
\begin{split}
\underset{\frac{1}{2}az \leq b  \leq az}{\left.\sum\right.^{\ast}} \frac{|\tilde{a}_{\pi}(b)| (2\alpha)^{\Omega(b)}}{b}\Big|\widecheck{\Phi}\Big( \frac{1}{2\pi} \log\Big(\frac{a}{b}\Big)\Big)\Big|&\ll \frac{1}{az (1 + |\log z|)^{10} }\underset{b \ll az}{\left.\sum\right.^{\ast}}|\tilde{a}_{\pi}(b)| (2\alpha)^{\Omega(b)} \\
& \ll \frac{1}{az (1 + |\log z|)^{10} } \frac{az}{\log az}\prod_{\substack{C_0 \leq p \ll az \\ |\lambda_{\pi}(p)| \leq B}}\Big( 1 + \frac{|\lambda_{\pi}(p)|(2\alpha)}{p}\Big)
\end{split}
\end{displaymath}
by a result of Wirsing \cite[Satz 1]{Wi}. (This result needs $|\lambda_{\pi}(p)|(2 \alpha) < 2$ as an assumption to make sure that $\sum_{k }( |\lambda_{\pi}(p)|(2 \alpha) )^k p^{-k}$ converges for every $p$. Since we exclude small primes, it suffices to assume $|\lambda_{\pi}(p)|(2 \alpha) < C_0$ which is guaranteed by the condition $|\lambda_{\pi}(p)| \leq B < C_0$ and $0 < \alpha < 1/2$.)

By Cauchy--Schwarz and Rankin--Selberg we have 
$$\sum_{p \leq x} \frac{|\lambda_{\pi}(p)|}{p} \leq \log\log x + O(1),$$
so that
$$\underset{\frac{1}{2}az \leq b  \leq az}{\left.\sum\right.^{\ast}} \frac{|\tilde{a}_{\pi}(b)| (2\alpha)^{\Omega(b)}}{b}\Big|\widecheck{\Phi}\Big( \frac{1}{2\pi} \log\Big(\frac{a}{b}\Big)\Big)\Big| \ll \frac{(\log az)^{2\alpha - 1}}{(1 + |\log z|)^{10} }. $$
Summing this over dyadic values of $z$ and inserting it into the remaining $a$-sum, we obtain
$$\left.\sum_a\right.^{\ast} \frac{|\tilde{a}_{\pi}(a)| (2\alpha)^{\Omega(a)}}{a} (\log a)^{2\alpha - 1}.$$
We split the sum into dyadic ranges and apply Wirsing's theorem again to conclude that the above is bounded by
$$\sum_{A = 2^{k}} \Big((\log A)^{2\alpha - 1}\Big)^2 \ll 1$$
provided that $\alpha < 1/4$. This completes the proof. 
\end{proof}

\subsection{End of proof of Theorem \ref{thm:unconditional}}\label{sec:end-of-proof}
We claim that
\begin{equation}\label{c-c-bound}
\prod_{i={1,2}}\big(\mathcal{P}^{(i)}(\Lambda_{\pi_i})+ O(D^{-\beta})\big)\ll \exp\Big(-\frac12 S_D(\pi_1,\pi_2)\Big), 
\end{equation}
and
\begin{equation}\label{c-e-bound}
\big(\mathcal{P}_{\alpha}^{(1)}(\Lambda_\pi) + O(D^{-\beta})\big)\Big(\int_{\R}  \mathcal{P}_{\alpha}^{(2)}(\Lambda_{it}) \Phi(t) \frac{{\rm d} t}{2\pi}+ O(D^{-\beta})\Big) \ll \exp\Big(-2(\alpha - \alpha^2) T_D(\pi)\Big)
\end{equation}
provided that $\alpha < 1/4$. In light of Corollary \ref{cor:CS} and Proposition \ref{cor:asymp-M1}, this will suffice to conclude the proof of Theorem \ref{thm:unconditional}.

We observe that 
\begin{displaymath}
\begin{split}
  \exp\Big( \sum_{p \in P_j} \frac{1}{p} \Big)& \leq \exp\big(   \log\log D^{1/\ell_j^2} - \log\log D^{1/\ell_{j-1}^2} + O(1)\big)\\
  & \ll \exp\big( 2 (\log \ell_{j-1} - \log \ell_j) \big)  = \exp\big( 2 (\tfrac{1}{60B^2 } \ell_j + O(1) - \log \ell_j) \big) \ll \exp\big( \tfrac{1}{30B^2 } \ell_j\big)
\end{split}
\end{displaymath}
for $j > 1$, and the final bound remains true for $j =1$. 

From 
Lemma \ref{lem41} we conclude that $| \mathcal{T}_j^{(1)} \mathcal{E}_j^{(2)}| +| \mathcal{E}_j^{(1)} \mathcal{T}_j^{(2)} | + |\mathcal{T}_j^{(1)} \mathcal{T}_j^{(2)}|$ is bounded by
\[
3 \cdot 2^{-\ell_j} \exp\Big( \sum_{p \in P_j} \frac{16B^2 }{p}\Big)\\ \ll 2^{-\ell_j} \exp \Big(\frac{16}{30} \ell_j\Big)  =  \exp \Big(\Big(\frac{16}{30} - \log 2\Big)\ell_j\Big), 
\]
so that 
$$ \mathcal{E}_j^{(1)} \mathcal{E}_j^{(2)} \gg  \exp\Big(-\frac{1}{4} \ell_j\Big)\big(|\mathcal{T}_j^{(1)} \mathcal{E}_j^{(2)}| +| \mathcal{E}_j^{(1)} \mathcal{T}_j^{(2)} | + |\mathcal{T}_j^{(1)} \mathcal{T}_j^{(2)}|\big).$$
Taking the product over all $j=1,\ldots ,R$, we obtain
\begin{displaymath}
\begin{split}
\mathcal{P}^{(1)}(\Lambda_{\pi_1})\mathcal{P}^{(2)}(\Lambda_{\pi_2}) & \geq  \mathcal{E}^{(1)} \mathcal{E}^{(2)}\Big(1 - O\big( \exp(-\tfrac{1}{8} \ell_j)\big)\Big) \gg  \mathcal{E}^{(1)} \mathcal{E}^{(2)} \\
&\gg \exp\Big( -\sum_{j=1}^R \sum_{p \in P_j}  \frac{ ( \lambda_{\pi_2}(p) - \lambda_{\pi_1}(p))^2}{2p} \Big).
\end{split}
\end{displaymath}
By a similar reasoning, we have for an individual value $i = 1, 2$ that
$$\mathcal{P}^{(i)}(\Lambda_{\pi_i})= \prod_{j=1}^R (\mathcal{E}_j^{(i)} -\mathcal{T}_j^{(i)})  \ll \prod_{j=1}^R \exp\Big(\sum_{p \in P_j} \frac{8B^2}{p}\Big) (1 + 2^{-\ell_j}) = (\log D)^{O(1)}$$
so that the error term $O(D^{-\beta})$ is negligible. This concludes the proof of \eqref{c-c-bound}. 

In the same way we obtain
\begin{displaymath}
\begin{split}
\mathcal{P}_{\alpha}^{(1)}(\Lambda_\pi)\ll \exp\big((\alpha^2 - 2\alpha)T_D(\pi)\big),
\end{split}
\end{displaymath}
which together with Lemma \ref{calpha} establishes \eqref{c-e-bound} as before.

\section{Proof of Theorem \ref{thm:unconditional} : the off-speed case}\label{sec-off}

The goal of this section is to prove part \eqref{thm21a} of Theorem \ref{thm:unconditional}, which states the bound 
\begin{equation}\label{recall-nu=2}
P_{\mathscr{D}_1,\mathscr{D}_2}^{\Delta_2}(f_1\otimes f_2) \ll (1+T_D(\pi_1))^{-1/2}
\end{equation}
under the hypothesis that $\bG_1$ is anisotropic. In the definition of $T_D(\pi_1)$ we have taken $B=\infty$, so that the restriction to Hecke eigenvalues satisfying $|\lambda_{\pi_1}(p)|\leq B$ is automatically satisfied.

The off-speed (or $\nu=2$) case treated in this section requires a different approach from the one presented in Section \ref{sec:mollification}. Indeed, a direct implementation of that method would lead to sums of the form
\[
\sum_{\chi \in\ClD^{\wedge}} |W_{\mathscr{D}_1}(f_1, \chi^2)|^2 \chi(\mathfrak{n}).
\]
But unless $\mathfrak{n}^{\ast}$ is a square, the work \cite{BBK} requires GRH to treat such a mean value. To see this, assume for simplicity that $|\ClD[2]|=1$. The obstacle is that, even if $[\mathfrak{n}^\ast]$ is minimally represented by an ideal of sufficiently small norm so that Theorem \ref{thm12} applies, the same is not necessarily true for the square-root class of $[\mathfrak{n}^\ast]$ in $\ClD$. To circumvent this problem, we proceed with a less symmetric argument, which is agnostic to whether $\bB_1\simeq \bB_2$ and more suited to detecting the differing character frequencies $\chi$ and $\chi^2$. It relies only upon the unconditional asymptotic formulae of Theorem \ref{thm12}.

A further complication arises in this setting from the fact that the image of the second component $\{[\mathfrak{a}]^2.[x_2]:[\mathfrak{a}]\in \ClD\}$ lies in a single coset of squares. Indeed, by Gau{\ss}' genus theory, the image captures only a proportion of about $2^{-\omega(D)}$ of the Heegner packet, which can be as small as $\exp(-\log D/\log\log D)$. We resolve this problem, not by assuming the triviality of the $2$-torsion of the class group, as in \cite[Theorem 2]{BB}, but by introducing an extra averaging over quadratic class group characters $\psi$, giving rise to the twisted periods \eqref{cusp}. (The $\nu=1$ case treated in Section \ref{sec:mollification} only used $\psi=1$.) This argument is carried out in Section \ref{sec:off-speed-proof}, which in fact needs no more information than $2^{\omega(D)} = D^{o(1)}.$ 

\subsection{Heuristics}\label{sec:heuristic1}
In this section we present our argument in a heuristic form; it will be carried out rigorously in Section \ref{sec:off-speed-proof}. For simplicity, we assume that $|\ClD[2]|=1$.

Recall from \eqref{parseval} the expression 
\[
P_{\mathscr{D}_1,\mathscr{D}_2}^{\Delta_2}(f_1\otimes f_2)=\sum_{\chi \in{\rm Cl}^{\wedge}_D} W_{\mathscr{D}_1}(f_1;\chi^2)W_{\mathscr{D}_2}(f_2;\overline{\chi}).
\]
A direct application of Cauchy--Schwarz, combined with the asymptotics
\begin{equation}\label{eq:mean-value}
\begin{aligned}
\sum_{\chi\in \ClD^\wedge} |W_{\mathscr{D}_1}(f_1,\chi^2)|^2\underset{|\ClD[2]|=1}{=}\sum_{\chi\in \ClD^\wedge} |W_{\mathscr{D}_1}(f_1,\chi)|^2&=1+O(D^{-\delta}),\\
\sum_{\chi\in \ClD^\wedge} |W_{\mathscr{D}_2}(f_2,\chi)|^2&=1+O(D^{-\delta}),
\end{aligned}
\end{equation}
from Theorem \ref{thm12} (applied with $\psi=1$ and $\mathfrak{n}=\mathscr{O}_E$), would give an upper bound of $O(1)$.

To improve this bound we need to exploit the distinct character frequencies --- $\chi^2$ and $\chi$ --- involved in the two factors. Rather than inserting mollifiers, as in Section \ref{sec:heuristic2}, our strategy to show \eqref{recall-nu=2} will be to first provide a convenient description of the two sets of characters that are responsible for the mean values in \eqref{eq:mean-value} and then to show that such character sets have little overlap. Both steps will use the full strength of Theorem \ref{thm12}. We can then reduce waste by conditioning our use of Cauchy--Schwarz on such information.

We adopt a probabilistic viewpoint and define the following weighted measure on $\ClD^\wedge$:
\[
\textrm{d}m_1: A\mapsto \sum_{\chi\in A} |W_{\mathscr{D}_1}(f_1,\chi^2)|^2.
\]
Since, by \eqref{eq:mean-value}, the total mass is $\textrm{d}m_1(\ClD^\wedge)=1+O(D^{-\delta})$, we may ignore the normalization in the heuristic and treat $\textrm{d}m_1$ as a probability measure. We then consider the sum
\begin{equation}\label{def:Q-chi}
\mathcal{Q}(\chi)=\sum_{\substack{C_0\le p\le D^c\\ p=\mathfrak p\bar{\mathfrak p}\ \mathrm{split}}}
\frac{\lambda_{\pi_1}(p)\big(\chi(\mathfrak p)+\overline{\chi}(\mathfrak p)\big)}{p^{1/2}},
\end{equation}
for appropriate constants $c,C_0>0$. Waldspurger's formula and the approximation argument of \cite[Lemma 6]{BB} together imply that $D^{-1/2}e^{\mathcal{Q}(\chi^2)}$ should approximate $|W_{\mathscr{D}_1}(f_1,\chi^2)|^2$, but it will be enough for our purposes to consider $\mathcal{Q}(\chi^2)$ itself. Viewing $\mathcal{Q}(\chi^2)$ as a random variable on $\ClD^\wedge$, we define its mean and variance relative to $\textrm{d}m_1$ as
\[
\mu= \sum_{\chi\in\ClD^\wedge}|W_{\mathscr{D}_1}(f_1,\chi^2)|^2\mathcal{Q}(\chi^2)\qquad\textrm{and}\qquad \sigma^2= \sum_{\chi\in\ClD^\wedge}|W_{\mathscr{D}_1}(f_1,\chi^2)|^2(\mathcal{Q}(\chi^2)-\mu)^2.
\]
In Lemma \ref{lem71} we show that $\mu=2T_D(\pi_1)+O(1)$ and $\sigma^2=O(T_D(\pi_1))$. This leads us to define the following subset, which should ``carry" the weighted measure $\textrm{d}m_1$:
\begin{equation}\label{def:mathcalX}
\mathcal{X}=\Big\{\chi\in\ClD^\wedge:\ |\mathcal{Q}(\chi^2)-2T_D(\pi_1)|\le \tfrac{1}{10}T_D(\pi_1)\Big\}.
\end{equation}
As the size of the window is $T_D(\pi_1)$, rather than $\sigma=O(T_D(\pi_1)^{1/2})$, the measure $\textrm{d}m_1$ should have negligible mass outside of $\mathcal{X}$. Indeed, an application of Chebyshev's inequality gives
\[
\sum_{\chi\notin\mathcal{X}} |W_{\mathscr{D}_1}(f_1,\chi^2)|^2\ll \frac{\sigma^2}{T_D(\pi_1)^2}\asymp T_D(\pi_1)^{-1}.
\]
We call $\mathcal{X}$ the set of \textit{$\textrm{d}m_1$-typical characters.}

On the other hand, we may also define a measure $\textrm{d}m_2$ on $\ClD^\wedge$ with respect to the weights of the second factor $|W_{\mathscr{D}_2}(f_2,\chi)|^2$. We would like to say that $\textrm{d}m_1$ is statistically disjoint with $\textrm{d}m_2$, in the sense that $\textrm{d}m_2$ should put negligible mass on $\mathcal{X}$. Now $\mathcal{Q}(\chi^2)^2\gg T_D(\pi_1)^2$ for all $\chi\in\mathcal{X}$. Hence, by Markov's inequality, 
\[
\textrm{d}m_2(\mathcal{X})
=\sum_{\chi\in\mathcal{X}} |W_{\mathscr{D}_2}(f_2,\chi)|^2
\ll \frac{1}{T_D(\pi_1)^2}\sum_{\chi\in \ClD^\wedge} |W_{\mathscr{D}_2}(f_2,\chi)|^2\mathcal{Q}(\chi^2)^2.
\]
Lemma \ref{lem71bis} bounds the last quantity by $O(T_D(\pi_1)^{-1})$.

Applying Cauchy--Schwarz in two different ways, according to whether $\chi\notin\mathcal{X}$ or $\chi\in\mathcal{X}$, leads to \eqref{recall-nu=2}. The argument crucially exploits the mismatch between $\chi$ and $\chi^2$, and therefore does not detect whether $f_1=f_2$. While this makes it unsuitable to the $\nu=1$ case, it allows for self-products in the $\nu=2$ case, as mentioned in Remark \ref{rem:variants}.

\subsection{Some moments}\label{sec:AES}

We keep the notation from Section \ref{sec:heuristic1}. In particular, we shall retain the definition of $\mathcal{W}(\chi)$ in \eqref{def:Q-chi}. The estimates in this section are expressed in terms of the quantity $T_D(\pi_1)$ from Section \ref{secoff}. These estimates are meaningful, of course, only if $T_D(\pi_1)$ is large, which is a separate question addressed by Corollary \ref{lem61}. 

The following moments provide a means for detecting which characters $\chi$ are responsible for the large values of $|W_{\mathscr{D}_1}(f_1,\chi)|^2$. 

\begin{lemma}\label{lem71} We have
\begin{align*}
&\sum_{\chi \in \ClD^{\wedge}} |W_{\mathscr{D}_1}(f_1,\chi)|^2\mathcal{Q}(\chi) = 2 T_D(\pi_1) + O(1),\\
& \sum_{\chi \in \ClD^{\wedge}} |W_{\mathscr{D}_1}(f_1,\chi)|^2\mathcal{Q}(\chi)^2 = 4 T_D(\pi_1)^2 + O(1+T_D(\pi_1)).
\end{align*}
\end{lemma}

\begin{proof} 
Both statements are simple consequences of Theorem \ref{thm12}. Indeed, we may switch the order of summation and apply Theorem \ref{thm12} with $\mathfrak{n}=\mathfrak{p}, \overline{\mathfrak{p}}$, to obtain
\[
\sum_{\chi \in \ClD^{\wedge}} |W_{\mathscr{D}_1}(f_1,\chi)|^2\mathcal{Q}(\chi) =  \sum_{\substack{ C_0 \leq p \leq D^c\\ p \text{ split}}}  \frac{2\lambda_{\pi_1}(p)^2}{p+1} + O(1),
\]
which yields the first claim. Similarly, we may apply Theorem \ref{thm12} with $\mathfrak{n}=\mathfrak{pq}, \mathfrak{p} \overline{\mathfrak{q}},  \overline{\mathfrak{p}} \mathfrak{q}, \overline{\mathfrak{p}} \overline{\mathfrak{q}}$, to get
\begin{align*}
& \sum_{\chi \in \ClD^{\wedge}} |W_{\mathscr{D}_1}(f_1,\chi)|^2\mathcal{Q}(\chi)^2 \\
&=  \sum_{ \substack{C_0\leq p \neq q \leq D^c\\ p, q \text{ split}}}  \frac{4\lambda_{\pi_1}(pq)^2}{(p+1)(q+1)} +  \sum_{ \substack{C_0 \leq p  \leq D^c\\ p \text{ split}}}   2\lambda_{\pi_i}(p)^2 \Big( \frac{\lambda_{\pi_1}(p^2)}{(p+1)\sqrt{p}} + \frac{1}{p}\Big) + O(1)\\
& =    \sum_{ \substack{C_0 \leq p ,  q \leq D^c\\ p, q, \text{ split}}}  \frac{4\lambda_{\pi_1}(p)^2 \lambda_{\pi_1}(q)^2}{pq} +    O(1+T_D(\pi_1)),
\end{align*}
as desired.\end{proof}

The next moment encodes the statistical independence of $|W_{\mathscr{D}_1}(f_1,\chi^2)|^2$ and $|W_{\mathscr{D}_2}(f_2,\chi)|^2$. We note that the upper bound is comparable to the error term of the second asymptotic in Lemma \ref{lem71}.
\begin{lemma}\label{lem71bis} 
For $f_2$ cuspidal or $f_2=E_\Psi$ we have 
\[
\sum_{\chi \in \ClD^{\wedge}}|W_{\mathscr{D}_2}(f_2,\chi)|^2 \mathcal{Q}_1(\chi^2)^2  \ll  1+T_D(\pi_1).
\]
\end{lemma}

\begin{proof}
If $f_2$ is cuspidal, Theorem \ref{thm12} gives
\begin{displaymath}
\begin{split}
 &\sum_{\chi \in \ClD^{\wedge}} |W_{\mathscr{D}_2}(f_2,\chi)|^2 \mathcal{Q}_1(\chi^2)^2  \\
 & =  \sum_{\substack{ C_0 \leq p \neq q \leq D^c\\ p, q \text{ split}}}  \frac{4\lambda_{\pi_1}(pq)\lambda_{\pi_2}(p^2q^2)}{\sqrt{pq}(p+1)(q+1)} +  \sum_{\substack{ C_0 \leq p  \leq D^c\\ p \text{ split}}} 2\lambda_{\pi_1}(p)^2 \Big( \frac{\lambda_{\pi_2}(p^4)}{p^2(p+1)} + \frac{1}{p}\Big) + O(1)\\
 & \ll  1+T_D(\pi_1), 
\end{split}\end{displaymath}
and the same argument works for $|W_{\mathscr{D}_2}(E_\Psi,\chi)|^2$ in place of $|W_{\mathscr{D}_2}(f_2,\chi)|^2$. 
\end{proof}

\subsection{Two different instances of Cauchy--Schwarz}\label{sec:off-speed-proof}

We now have the ingredients necessary to prove \eqref{recall-nu=2}. Note that we may, and will, assume that $T_D(\pi_1) > 1$, since if $T_D(\pi_1) \leq 1$, the bound follows trivially from the Cauchy--Schwarz inequality and Theorem \ref{thm12} with $\mathfrak{n} =\mathscr{O}_E$. 

As a way of overcoming the non-bijectivity of $\chi\mapsto\chi^2$, we average over quadratic characters, getting 
\begin{equation}\label{psi-av}
\begin{aligned}
P_{\mathscr{D}_1,\mathscr{D}_2}^{\Delta_2}(f_1\otimes f_2)&= \frac{1}{|\ClD[2]|}\sum_{\substack{\psi   \in \ClD^\wedge\\  \psi^2 = 1}} \sum_{\chi\in \ClD^\wedge} W_{\mathscr{D}_1}(f_1,(\chi\psi)^2)   \overline{W_{\mathscr{D}_2}(f_2,\chi\psi)}\\
&=\sum_{\chi\in \ClD^\wedge} W_{\mathscr{D}_1}(f_1,\chi^2)  \frac{1}{|\ClD[2]|}  \sum_{\substack{\psi   \in \ClD^\wedge\\  \psi^2 = 1}} \overline{W_{\mathscr{D}_2}(f_2,\chi\psi)}.
\end{aligned}
\end{equation}
We define $\mathcal{X}$ as in \eqref{def:mathcalX}. Write $P_{\mathcal{X}}$ for the contribution to \eqref{psi-av} of those characters in $\mathcal{X}$ and $P_{\mathcal{X}^c}$ for the contribution to \eqref{psi-av} of those characters not in $\mathcal{X}$. By an argument modelled on the heuristics of Section \ref{sec:heuristic1}, we will show that
\[
P_{\mathcal{X}^c}\ll T_D(\pi_1)^{-1/2},\quad P_{\mathcal{X}}\ll T_D(\pi_1)^{-1/2} \cdot (1 + |\ClD[2]|^{1/2} \cdot D^{c - \delta / 2}).
\]
with $c$ as in the definition of $T_D(\pi_1)$ and $\delta$ as in Theorem \ref{thm12}. Using $|\ClD[2]|\ll_\varepsilon D^\varepsilon$ and taking $c$ small enough, this will finish the proof.

By Cauchy--Schwarz we have
\[
|P_{\mathcal{X}^c}|^2\leq \Big(\sum_{\chi\not\in \mathcal{X}} |W_{\mathscr{D}_1}(f_1,\chi^2)|^2\Big) \Big(\sum_{\chi\notin \mathcal{X}} \Big|\frac{1}{|\ClD[2]|}   \sum_{\substack{\psi  \in \ClD^\wedge\\  \psi^2 = 1}}W_{\mathscr{D}_2}(f_2,\chi\psi) \Big|^2\Big). 
\]
For the first factor, we use $|\mathcal{Q}(\chi^2) - 2T_D(\pi_1)|^2 \gg T_D(\pi_1)^2$ for $\chi\not\in \mathcal{X}$, to get
\[
\sum_{\chi\not\in \mathcal{X}} |W_{\mathscr{D}_1}(f_1,\chi^2)|^2\ll \sum_{\chi\in  \ClD^{\wedge}} |W_{\mathscr{D}_1}(f_1,\chi^2)|^2\frac{|\mathcal{Q}(\chi^2) - 2T_D(\pi_1)|^2}{T_D(\pi_1)^2}.
\]
Changing variables and using positivity,  the last quantity is equal to
\begin{align*}
|\ClD[2]|\sum_{\chi\in  (\ClD^{\wedge})^2}& |W_{\mathscr{D}_1}(f_1,\chi)|^2\frac{|\mathcal{Q}(\chi) - 2T_D(\pi_1)|^2}{T_D(\pi_1)^2}\\
&\leq |\ClD[2]|\sum_{\chi\in  \ClD^{\wedge}} |W_{\mathscr{D}_1}(f_1,\chi)|^2\frac{|\mathcal{Q}(\chi) - 2T_D(\pi_1)|^2}{T_D(\pi_1)^2}\ll \frac{|\ClD[2]|}{T_D(\pi_1)}
\end{align*}
where the last bound follows from   Lemma \ref{lem71}. The presence of the $|\ClD[2]|$ is potentially problematic, but this will be compensated in the second term, thanks to the additional $\psi$-average. Indeed, for the second factor, we extend the sum to all $\chi\in \ClD^\wedge$ by positivity and open up the square to obtain
\begin{displaymath}
\begin{split}
\frac{1}{|\ClD[2]|^2} \sum_{\substack{\psi_1, \psi_2  \in \ClD^\wedge\\  \psi_1^2= \psi_2^2 = 1}} P_{\mathscr{D}_2}(f_2\otimes f_2; \mathscr{O}_E; \psi_1\psi_2^{-1}).
 \end{split}
\end{displaymath}
By Theorem \ref{thm12} we bound this by $|\ClD[2]|^{-1} + O(D^{-\delta}) \ll |\ClD[2]|^{-1}$, as desired. 

On the other hand, by Cauchy--Schwarz
\[
|P_{\mathcal{X}}|^2\leq \Big(\sum_{\chi \in \mathcal{X}} |W_{\mathscr{D}_1}(f_1,\chi^2)|^2 \Big) \Big(\sum_{\chi\in \mathcal{X}} \Big|\frac{1}{|\ClD[2]|}   \sum_{\substack{\psi  \in \ClD^\wedge\\  \psi^2 = 1}}W_{\mathscr{D}_2}(f_2,\chi \psi)\Big|^2\Big).
\]
For the first factor, we extend to all $\chi\in \ClD^\wedge$ by positivity, change variables, and use positivity again to obtain
\[
\sum_{\chi \in \mathcal{X}} |W_{\mathscr{D}_1}(f_1,\chi^2)|^2\leq \sum_{\chi \in \ClD^\wedge} |W_{\mathscr{D}_1}(f_1,\chi^2)|^2\leq |\ClD[2]|\sum_{\chi \in \ClD^\wedge} |W_{\mathscr{D}_1}(f_1,\chi)|^2.
\]
The cuspidal case of Theorem \ref{thm12} then applies, with $\mathfrak{n}=\mathscr{O}_E$, and yields a bound of  $\ll  |\ClD[2]|$. Once again, the factor of $ |\ClD[2]|$ is potentially problematic, but will be compensated by the second factor, thanks to the additional $\psi$-average. Indeed, for the second factor, we use $\mathcal{Q}(\chi^2) \gg T_D(\pi_1)$ for $
\chi \in \mathcal{X}$, so that 
\begin{align*}
 &  \sum_{\chi\in \mathcal{X}} \Big|\frac{1}{|\ClD[2]|}   \sum_{\substack{\psi  \in \ClD^\wedge\\  \psi^2 = 1}} W_{\mathscr{D}_2}(f_2,\chi \psi)\Big|^2  \ll  \sum_{\chi\in \mathcal{X}}  \frac{\mathcal{Q}(\chi^2)^2}{T_D(\pi_1)^2}  \Big|\frac{1}{|\ClD[2]|}   \sum_{\substack{\psi  \in \ClD^\wedge\\  \psi^2 = 1}} W_{\mathscr{D}_2}(f_2,\chi \psi)\Big|^2\\
   &=     \frac{1}{|\ClD[2]|^2 }   \sum_{\substack{\psi_1, \psi_2  \in \ClD^\wedge\\ \psi_1^2 =  \psi_2^2 = 1}}\sum_{\chi\in \ClD^\wedge}W_{\mathscr{D}_2}(f_2,\chi \psi_1)\overline{W_{\mathscr{D}_2}(f_2,\chi \psi_2)} \frac{\mathcal{Q}(\chi^2)^2}{T_D(\pi_1)^2} .  
\end{align*}
We can use Lemma \ref{lem71bis}  to bound the contribution of $\psi_1 = \psi_2$ by  $\ll ( T_D(\pi_1) |\ClD[2]| )^{-1}$. For the contribution $\psi_1 \neq \psi_2$ we open $\mathcal{Q}(\chi^2)$ to get 
\[
\mathcal{Q}(\chi^2)^2 = \sum_{\substack{C_0 \leq p_1, p_2 \leq D^c\\ p_1 = \mathfrak{p}_1 \bar{\mathfrak{p}}_1 \text{ split}\\ p_2 = \mathfrak{p}_2 \bar{\mathfrak{p}}_2 \text{ split}  }} \frac{\lambda_{\pi_1}(p_1)\lambda_{\pi_1}(p_2) (\chi(\mathfrak{p}_1^2) + \bar{\chi}(\mathfrak{p}_1^2))(\chi(\mathfrak{p}_2^2) + \bar{\chi}(\mathfrak{p}_2^2))}{p_1^{1/2}p_2^{1/2}}
\]
and apply Theorem \ref{thm12} termwise with $\mathfrak{n} = \mathfrak{p}_1^2 \mathfrak{p}_2^2$ (so that $[\mathfrak{n}^*]=[\mathfrak{p}_1\mathfrak{p}_2]^2$ is a square class; see the opening remarks of this section) or certain Galois conjugates to show that this contribution is bounded by $T_D(\pi_1)^{-2} D^{2c-\delta}$, as desired.

\section{Distribution of Hecke eigenvalues}\label{sec2}

In this section we prove Theorem \ref{thm15}, using critically the automorphy and cuspidality conditions of symmetric $k$-power lifts from ${\bf GL}_2$, for $k=2,3,4$, by \cite{GJ, KS1, KS2}. We review this information in Section \ref{sec:cuspidality}. The proof also uses a delicate construction of polynomials (in two variables and one variable, according to the two cases in Theorem \ref{thm15}) satisfying a degree restriction and certain positivity requirements. We execute this in Section \ref{sec:polynomials}.

\subsection{Warming up}
Let us briefly illustrate the use of these two ingredients (small degree symmetric power lifts and polynomial expressions in Hecke eigenvalues), without striving for optimality.

Let $\pi_1,\pi_2$ be cuspidal representations of ${\bf PGL}_2/\Q$, satisfying the conditions of part \eqref{Hecke-part1} of Theorem \ref{thm15}. For a small enough $\varepsilon>0$ let $S=\{p\leq X: |\lambda_{\pi_1}(p)-\lambda_{\pi_2}(p)|\geq \varepsilon\}$ and consider the sum
\begin{equation}\label{eq:first-trial}
\sum_{p\in S}\frac{(\lambda_{\pi_1}(p)-\lambda_{\pi_2}(p))^2}{p}.
\end{equation}
Using the Hecke relations and the known analytic properties of
\begin{equation}\label{eq:L-funs1}
\zeta(s),\quad L(s,\pi_1\times \pi_2), \quad\textrm{and}\quad L(s,{\rm sym}^2\pi_i),
\end{equation}
we may show that \eqref{eq:first-trial} is asymptotic to $(2+O(\varepsilon^2))\sum_{p\leq X}p^{-1}$. On the other hand, by Cauchy--Schwarz, one may bound \eqref{eq:first-trial} by
\begin{equation}\label{eq:C-S}
\left(\sum_{p\leq X}\frac{(\lambda_{\pi_1}(p)-\lambda_{\pi_2}(p))^4}{p}\right)^{1/2}\left(\sum_{p\in S}p^{-1}\right)^{1/2}.
\end{equation}
As before, the Hecke relations and the analytic properties of 
\begin{equation}\label{eq:L-funs2}
\zeta(s), \;\; L(s,{\rm sym}^2\pi_i), \;\; L(s,\pi_i\times \pi_j), \;\; L(s,{\rm sym}^2\pi_i\times {\rm sym}^2\pi_j), \;\;\textrm{and}\;\; L(s,\pi_i\times {\rm sym}^3\pi_j)
\end{equation}
imply that \eqref{eq:C-S} is asymptotic to $\left(10\sum_{p\leq X}p^{-1}\right)^{1/2}\left(\sum_{p\in S}p^{-1}\right)^{1/2}$. Putting these bounds together proves a lower bound of $2/5 + O(\varepsilon^2)$ on the Dirichlet density of $S$.

Note that the above argument uses only the information contained in \eqref{eq:L-funs1} and \eqref{eq:L-funs2}, and does not access all of the automorphic information currently available, due to the foundational work of Kim and Shahidi \cite{KS1, KS2}. Indeed, their work establishes the good analytic properties of $L(s,{\rm sym}^k\pi_i)$, for all $k\leq 8$, and $L(s,{\rm sym}^{k_1}\pi_i\times {\rm sym}^{k_2}\pi_j)$, for all $k_1,k_2\leq 4$. Because of this, the above argument is not strong enough to detect the example in Remark \ref{rem:S-T}. (That example saturates the 2/5 bound but does not satisfy all conditions in the statement of Theorem \ref{thm15}.) On the other hand, applying the same argument to a sum of the shape 
\[
\sum_{p\in S}\frac{(\lambda_{\pi_1}(p)-\lambda_{\pi_2}(p))^2(3+\lambda_{\pi_1}(p)\lambda_{\pi_2}(p))}{p},
\]
say, instead of \eqref{eq:first-trial}, yields a lower bound on the Dirichlet density of size $4/7>1/2$, \textit{provided one assumes the conjectural properties of $L(s,{\rm sym}^6 \pi_i\times {\rm sym}^2\pi_j)$ and $L(s,{\rm sym}^5 \pi_i\times {\rm sym}^3\pi_j)$.}

The question is then whether an efficient use of the entirety of the automorphic information provided by the work of Kim and Shahidi is sufficient to go beyond the 1/2-barrier. Our Theorem \ref{thm15} affirms that the answer is yes, although just barely!

\subsection{Cuspidality of symmetric powers}\label{sec:cuspidality}

Theorem \ref{thm15} presupposes certain conditions on the cuspidality and distinctness of the first few symmetric power lifts. That these conditions apply to our situation follows from the following lemma.

\begin{lemma}\label{lem22} $\,$
\begin{enumerate}
\item\label{cuspidality} Let $\sigma\subset L^2([\bG(\A)])$ have almost maximal level invariants $\sigma^K\neq 0$, as defined in Section \ref{secgeneral}. Let $\pi={\rm JL}(\sigma)$ be its Jacquet--Langlands lift to ${\bf PGL}_2/\Q$. Then ${\rm sym}^k \pi$ is cuspidal for $2 \leq k \leq 4$.

\item\label{distinct-sym} For $i=1,2$ let $\pi_i={\rm JL}(\sigma_i)$ as in part \eqref{cuspidality}, where $\bG_1\not\simeq\bG_2$. Then ${\rm sym}^k \pi_1\not\simeq {\rm sym}^k \pi_2$ for $1 \leq k \leq 4$.
\end{enumerate}
\end{lemma}

\begin{proof} 
Let $\pi\simeq\otimes'_v\pi_v$ be as in part \eqref{cuspidality}. Note that for $p\in {\rm Ram}(\bB)$ the local component $\pi_p={\rm JL}(\sigma_p)$ is a twisted Steinberg representation of $\mathbf{PGL}_2(\Q_p)$. Indeed, at such places $K_p$ is an index 2 subgroup of the compact group $\bG(\Q_p)$, so that the almost maximal level invariance at $p\in {\rm Ram}(\bB)$ implies that $\sigma_p$ is a character (trivial or quadratic) of $\bG(\Q_p)$. Under the Jacquet--Langlands correspondence, characters go to twisted Steinberg representation.

On the other hand, by the work of \cite{KS2}, the cuspidal automorphic representations $\pi$ of ${\bf GL}_2/\Q$ for which all ${\rm sym}^k\pi$, for all $k=2,3,4$, are cuspidal are precisely those which do not correspond to an Artin representation of solvable polyhedral type. It therefore suffices to observe that an automorphic Artin representation can never have a twisted Steinberg representation as a local component. 

To see the latter claim, we argue by the local Langlands correspondence \cite{BuHe}, according to which smooth irreducible complex representations of ${\bf GL}_2(\Q_p)$ correspond with semisimple 2-dimensional
Weil--Deligne representations of the Weil group $W_p$. The latter are given by pairs $(\rho, N)$, with $\rho$ a representation of $W_p$ and $N$ a nilpotent endomorphism verifying a compatibility condition. The important point is that in this correspondence the Steinberg representation and its twists are precisely those representations of ${\bf GL}_2(\Q_p)$ whose associated Weil--Deligne representation has a non-trivial $N$.

Now given a (global) 2-dimensional Artin representation $r: {\rm Gal}(\bar{\Q}/\Q)\rightarrow {\bf GL}_2(\C)$ we obtain a 2-dimensional (local) ``Weil group representation
of Galois-type" simply by restricting to $W_p$ and setting $N=0$. As already noted, $(r|_{W_p},0)$ cannot be mapped via the local Langlands correspondence to a twisted Steinberg representation.

For part \eqref{distinct-sym}, since $\sigma_i$ has almost maximal level, the conductor of $\pi_i={\rm JL}(\sigma_i)$ is $N_i=\prod_{p\in {\rm Ram}(\bB_i)}p$. Since $\bG_1\not\simeq \G_2$, there is $p$ with $p\mid N_1$ but $p\nmid N_2$, so that $({\rm sym}^k\pi_1)_p$ is ramified while $({\rm sym}^k\pi_2)_p$ is not.
\end{proof}

\subsection{A linear programming problem}\label{sec:polynomials}

Let $C(i) = \binom{2i}{i} -\binom{2i}{i+1} $ be the $i$-th Catalan number.

\begin{lemma}\label{lem:poly}
There exists $f  = \sum_{i, j} a_{i, j} x^i y^j \in \R[x, y]$ with the following properties:
\begin{enumerate}
\item\label{small-degree} $a_{i, j} = 0$ unless $0 \leq i, j \leq 4$ or $ij = 0, 0 \leq i, j \leq 8$;
\item\label{half} we have $\sum_{i, j} a_{2i, 2j} C(i) C(j)<1/2$;
\item\label{pos} $f(x, y) \geq 0$ for all $x, y \in \R$;
\item\label{diag} there is $\varepsilon>0$ such that $f(x, y) \geq 1$ whenever $|x- y|\leq \varepsilon$.
\end{enumerate}
In fact, the upper bound in \eqref{half} can be taken to be $0.492$.
\end{lemma}
\begin{proof}
Let 
\begin{equation}\label{deff}
f(x,y) = \frac{1}{20^{10}}(u(x, y) + u(-x, -y))
\end{equation}
where
\begin{flalign*}
  u(x,y) & =  41813434533 x^8 + 359512561893 x^4 y^4 
          + 41790603610 y^8 + 1371693928 x^7 \\ &- 407914952 x^4 y^3 
          + 409594776 x^3 y^4 - 1371845320 y^7 - 152799075816 x^6 \\
         & - 1238844857440 x^4 y^2 + 914071084096 x^3 y^3 - 1238781629424 x^2 y^4 - 152613570520 y^6 \\
         & - 10586722304 x^5 + 155563456 x^4 y - 2039508160 x^3 y^2 + 2035437600 x^2 y^3\\
         &  - 167534816 x y^4  
           + 10587831584 y^5 + 655694115936 x^4 - 2074041622592 x^3 y \\
         & + 5223381207392 x^2 y^2 - 2074067626368 x y^3 + 655195946720 y^4 + 25484108416 x^3\\
         &  - 2399492480 x^2 y + 2432177280 x y^2 - 25477793408 y^3 - 3336702739328 x^2\\
         &  + 5796896462336 x y 
          - 3336250252672 y^2 - 16640770048 x + 16623868928 y \\
          & + 5217874549248,
\end{flalign*}
which clearly verifies \eqref{small-degree}. A numerical computation shows that
\[
\sum_{i, j} a_{2i, 2j} C(i) C(j)= \frac{1258136733059}{2560000000000 } = 0.4914\ldots
\] 
{\tt Mathematica} tells us\\

\noindent {\tt  In[1] := NMinimize[f[x, y], \{x, y\}]\\
Out[1] := \{0.000147597, \{x -> 1.96743, y -> 0.756825\}\} }\\

\noindent which establishes \eqref{pos}. Furthermore,\\

\noindent {\tt  In[2] := NMinimize[f[x, x], x]\\
Out[2] := \{1.0001, \{x -> 0.502488\}\} }.\\

\noindent By continuity, there exists some $\varepsilon > 0$, such that \eqref{diag} is satisfied.  
\end{proof}

In Sections \ref{sec:pos}-\ref{sec:diag} we provide direct, non computer-assisted, proofs of both \eqref{pos} and \eqref{diag}. For the reader's convenience, Lemma \ref{lem:poly} has been formalized in Lean4. The code is available at
\href{https://github.com/maksym-radziwill/BBR}{https://github.com/maksym-radziwill/BBR}. 

\medskip

The following one-variable version is much simpler. 
\begin{lemma}\label{polyh}
There exists $h  = \sum_i a_ix^i  \in \R[x]$ with the following properties:
\begin{enumerate}
\item ${\rm deg}(h)\leq 8$;
\item\label{half2} we have $\sum_i a_{2i} C(i) <1/2$;
\item $h(x) \geq 0$ for all $x \in \R$;
\item there is $\varepsilon>0$ such that $h(x) \geq 1$ whenever $|x|\leq \varepsilon$.
\end{enumerate}
In fact, the upper bound in \eqref{half2} can be taken to be $.42$.
\end{lemma}

\begin{proof}
Define the degree 8 polynomial
\[
h(x) = \frac{100}{81} - \frac{80}{27} x^2 + \frac{68}{27} x^4 - \frac{8}{9} x^6  + \frac{1}{9}x^8  =  \Big(\frac{10}{9} - \frac{4}{3} x^2 + \frac{1}{3}x^4\Big)^2.
\]
Then $h$ is clearly non-negative and satisfies
\[
\frac{100}{81} - \frac{80}{27} \cdot 1+ \frac{68}{27} \cdot 2 - \frac{8}{9} \cdot 5  + \frac{1}{9} \cdot 14 =  \frac{34}{81} = 0.419\ldots
\]
Finally, $h(0) = 10/9 > 1$, and the existence of $\varepsilon>0$ as in the last point follows from continuity.
\end{proof}

\subsection{Non-negativity}\label{sec:pos}
Here we verify directly property \eqref{pos} of Lemma \ref{lem:poly}. Throughout this section we denote by $f$ the polynomial defined in \eqref{deff} and the large equation following that.

\begin{lemma}
  The polynomial $f$ is non-negative.
\end{lemma}
\begin{proof}
  This follows because we can write $u$ as a sum of thirty-one squares of polynomials.
  Precisely,
  \[
  u(x,y) = 2\sum_{i = 1}^{15} p_{i}(x,y)^{2} + 6 \cdot p_{5}(x,y)^{2} + c(x,y)
  \]
  where
  \begin{flalign*}
    c(x,y) & = 144792 (x^4 - x^3 y)^2 + 173336 (x^4 + x^2 y^2)^2 + 106511 (x^4 - x^2 y^2)^2 \\
    & + 219778 (x^4 - x y^3)^2 + 177162 (x^4 - 2 x^2 y)^2 + 293174 (x^4 + 2 x y^2)^2 \\
    & + 230298 (x^4 + 4 x y)^2 + 68230 (x^3 y - y^4)^2 + 157876 (x^2 y^2 + y^4)^2 \\
    & + 147938 (x^2 y^2 - y^4)^2 + 774352 (x^2 y^2 + 2 y^3)^2 + 580764 (x^2 y^2 - 2 y^3)^2 \\
    & + 76098 (x y^3 - y^4)^2 + 20320 (y^4 + 2 x y^2)^2 + 40640 (y^4 - 2 x y^2)^2 
     + 56386 (y^4 + 4 x y)^2 
  \end{flalign*}
  and the polynomials $p_{i}(x,y)$ are given as follows
  \begin{flalign*}
    p_{1}(x,y) & = 110664 x^4 - 62436 x^3 y - 125829 x^2 y^2 - 62424 x y^3 + 110640 y^4 + 872 x^3 + 364 x^2 y& \\
               &  - 346208 x^2 + 111564 x y                 
                 - 346144 y^2 - 2208 x + 2200 y + 369968 &
  \end{flalign*}
    \begin{flalign*}
      p_{2}(x,y) & = 72082 x^4 + 115475 x^3 y + 85682 x^2 y^2 + 115539 x y^3 + 72018 y^4 + 262 x^3 - 222 x^2 y & \\
                 & + 224 x y^2 - 262 y^3 - 301192 x^2  
                 - 566348 x y - 301000 y^2 - 1264 x + 1256 y + 123456 &
    \end{flalign*}
    \begin{flalign*}
      p_{3}(x,y) & = 55620 x^4 - 53290 x^3 y - 10 x^2 y^2 + 53273 x y^3 - 55648 y^4  - 3232 x^3 + 776 x^2 y& \\
                 & + 776 x y^2 - 3232 y^3 - 167796 x^2  
                  + 60 x y + 167896 y^2 + 8488 x + 8488 y - 16 &
    \end{flalign*}
\begin{flalign*}
  p_{4}(x,y) & = 19116 x^4 + 81137 x^3 y - 291925 x^2 y^2 + 81119 x y^3 + 19126 y^4  - 432 x^3 + 370 x^2 y& \\
             & - 370 x y^2 + 434 y^3 + 547060 x^2  
              - 702352 x y + 547024 y^2 + 2032 x - 2032 y - 1567392 &
\end{flalign*}
\begin{flalign*}
  p_{5}(x,y) & = 1021 x^4 - 503 x^3 y + 503 x y^3 - 1022 y^4 + 101755 x^3 + 4272 x^2 y + 4262 x y^2& \\
             &  + 101760 y^3 - 3670 x^2 + 3672 y^2  
               - 306992 x - 307004 y &
\end{flalign*}
\begin{flalign*}
  p_{6}(x,y) & = 462 x^4 + 1121 x^3 y - 1122 x y^3 - 463 y^4 - 406 x^3 + 2114 x^2 y + 2114 x y^2& \\
             &  - 404 y^3 - 3192 x^2 + 3192 y^2 
              - 1544 x - 1544 y &
\end{flalign*}
\begin{flalign*}
  p_{7}(x,y) & = 368 x^4 + 16 x^3 y + 283 x^2 y^2 + 16 x y^3 + 368 y^4+ 54 x^3 + 52 x^2 y  & \\
             & - 52 x y^2 - 56 y^3 + 884 x^2  
               + 1220 x y + 884 y^2 - 56 x + 56 y - 9664
\end{flalign*}
\begin{flalign*}
  p_{8}(x,y) & = 362 x^4 + 863 x^3 y - 863 x y^3 - 362 y^4 + 528 x^3  - 2896 x^2 y - 2896 x y^2 + 530 y^3& \\
             & - 2524 x^2 + 2524 y^2               
               + 2272 x + 2272 y
\end{flalign*}
\begin{flalign*}
  p_{9}(x,y) & = 222 x^4 + 63 x^3 y - 62 x y^3 - 222 y^4 + 22 x^3 + 8 x^2 y + 8 x y^2 + 22 y^3 & \\
             & + 860 x^2 + 4 x y  
               - 860 y^2 + 120 x + 120 y + 16
\end{flalign*}
\begin{flalign*}
  p_{10}(x,y) & = 61 x^4 - 17 x^3 y + 37 x^2 y^2 - 17 x y^3 + 61 y^4 + 4 x^3 - 4 x^2 y + 4 x y^2& \\
              &  - 2 y^3 + 312 x^2  
               - 76 x y + 312 y^2 + 24 x - 24 y + 2208
\end{flalign*}
\begin{flalign*}
  p_{11}(x,y) & = 15 x^4 - x^3 y + 11 x^2 y^2 - x y^3 + 14 y^4  - 536 x^3 - 988 x^2 y & \\
              &+ 988 x y^2 + 536 y^3 + 28 x^2 
                + 48 x y + 24 y^2 - 136 x + 128 y - 512
\end{flalign*}
\begin{flalign*}
  p_{12}(x,y) & = 15 x^4 - 4 x^3 y + 13 x^2 y^2 - 4 x y^3 + 15 y^4  - 1582 x^3 + 808 x^2 y - 808 x y^2& \\
              & + 1582 y^3 + 20 x^2  
              + 24 x y + 20 y^2 + 6616 x - 6616 y - 432
\end{flalign*}
\begin{flalign*}
  p_{13}(x,y) & = 10 x^4 + 3 x^3 y - 4 x y^3 - 10 y^4 - 266 x^3  - 118 x^2 y - 118 x y^2 - 266 y^3& \\
              & + 36 x^2 - 4 x y 
                - 36 y^2 - 1440 x - 1432 y
\end{flalign*}
\begin{flalign*}
  p_{14}(x,y) & = 5 x^4 + 199 x^3 y - 99 x^2 y^2 + 199 x y^3 + 7 y^4  + 12 x^3 - 12 x^2 y + 10 x y^2 & \\
              &- 14 y^3 - 20 x^2  
               + 1128 x y - 12 y^2 + 88 x - 96 y + 1776
\end{flalign*}
\begin{flalign*}
  p_{15}(x,y) & = 3 x^4 + 5 x^3 y - x^2 y^2 + 5 x y^3 + 3 y^4  - 224 x^3 + 130 x^2 y - 130 x y^2& \\
              & + 224 y^3 + 12 x^2 
              + 32 x y + 12 y^2 - 1112 x + 1112 y + 128
\end{flalign*}
This concludes the proof of the lemma.
  \end{proof}

\subsection{Diagonal lower bound}\label{sec:diag}

Now we verify directly property \eqref{diag} of Lemma \ref{lem:poly}. We again let $f$ denote the polynomial defined in \eqref{deff} and the large equation following that.

\begin{lemma}
We have
  \[
  |f(x,y)| < 7 \cdot (2 + x^{8} + y^{8})
  \]
  and
  \[
  \Big | \frac{\partial f}{\partial x}(x, y) \Big | \leq 56 \cdot (2 + x^{8} + y^{8}) \ , \ \Big | \frac{\partial f}{\partial y} (x, y)\Big | \leq 56 \cdot (2 + x^{8} + y^{8}).
  \]
\end{lemma}
\begin{proof}  Let 
\[
S := \{(i,j) : i, j \leq 4\} \cup \{(i, 0): i \leq 8\} \cup \{(0, i): i \leq 8\}.
\]
Then for any pair $(i, j) \in S$ we have
\[ 
|x^{i} y^{j}| \leq x^{2i} + y^{2j} \leq (2 + x^{8} + y^{8}),
\]
so that for any polynomial of the form
$w(x,y) = \sum_{(i,j) \in S} a_{i j} x^i y^j$ we have
  \[
  |w(x,y)| \leq C(w) (2 + x^8 + y^8) \ , \ C(w) := \sum_{(i,j) \in S} |a_{i j}|
  \]
  By the same argument,
  \[
  \Big | \frac{\partial w}{\partial x} (x, y)\Big | \leq 8 C(w) (2 + x^8 + y^8) \ , \ \Big | \frac{\partial w}{\partial y}(x, y) \Big | \leq 8 C(w) (2 + x^8 + y^8).
  \]
  We compute
  \[
  C(u) = \frac{1631426159869}{256000000000} < 7
  \]
  and since  $f$ is a convex combination of $u$, this completes the proof. 
\end{proof}

  We now show that close to the diagonal $f(x,y)$ is strictly greater than one.
  \begin{lemma}
    For $|x - y| \leq 10^{-15}$ we have 
    $
    f(x,y) > 1. 
    $
  \end{lemma}
  \begin{proof} Let $\Delta(x) = f(x, x)$. We find
  $$\Delta(x, x) = 1 + \frac{1}{10000} + \sum_{i=1}^3 q_i(x)^2$$
  where
  \begin{displaymath}
  \begin{split}
 &  q_1(x) = \frac{815955355589474040 - 3670934166676790610 x^2 + 
 1740619888194368521 x^4}{429086076000000 \sqrt{190161229}},\\
 & q_2(x) = \frac{1}{200000} \sqrt{\frac{168706828659586711271257}{203988838897368510}} x^2\\
 &\quad\quad\quad- \frac{
 883362315013506570851197  }{
 1600000 \sqrt{34414310092326386253543842538654099917070}} x^4,\\
 & q_3(x) = \frac{1}{858172152000000} \sqrt{\frac{324146632648955034540879425665616037031045777981}{
337413657319173422542514}} x^4,
  \end{split}
  \end{displaymath}
  so that 
  \[
  \Delta(x, x) \geq \max(1 + 10^{-5}, q_3(x)^2) \geq \max(1 + 10^{-5}, 10^{-6}x^8). 
  \]
   By the bound for the partial derivatives of $f$ we have for $|x - y| \leq 10^{-15}$ that
    \[
    f(x,y) \geq \Delta(x) - 10^{-15} \cdot 56 \cdot (2 + x^8 + y^8).
    \]
    If $|x| \leq 1$ and $|y| \leq 1$, then $x^8 + y^8 \leq 2$.
    If either $|x| \geq 1$ or $|y| \geq 1$ then since $|x - y| \leq 10^{-8}$ we have $|\max(x, y)| < (1 + 2 \cdot 10^{-8}) |x|$,
    hence $x^8 + y^8 \leq 2 x^8$. In both cases $x^8 + y^8 \leq 2 + 2 x^8$ for $|x - y| \leq 10^{-8}$.
    Therefore
    \[
    f(x,y) \geq \Delta(x) - 10^{-15} \cdot 56 \cdot (4 + 2 x^8).
    \]
whenever $|x - y| \leq 10^{-15}$. If $|x| \leq 6$, then 
\[
\Delta(x) - 10^{-15} \cdot 56 \cdot (4 + 2 x^8) \geq  1 + 10^{-5} - 10^{-15} \cdot 56 \cdot (4 + 2 \cdot 6^8) = \frac{31250306621337}{31250000000000} > 1,
\]
and if $|x| > 6$, we have 
\[
\Delta(x) - 10^{-15} \cdot 56 \cdot (4 + 2 x^8) \geq  10^{-6} \cdot 6^8 - 10^{-15} \cdot 56 \cdot (4 + 2 \cdot 6^8) = \frac{52487994121337}{31250000000000} > 1.
\]
This completes the proof. 
   \end{proof}

\subsection{Proof of Theorem \ref{thm15}} 

Part \eqref{Hecke-part1} uses Lemma \ref{lem:poly}, whereas Part \eqref{Hecke-part2} uses Lemma \ref{polyh}. We prove only the former, as the latter follows from a similar argument.

 Let $X_0 > 0$ be such that $\pi_1, \pi_2$ are unramified at all primes $p \geq X_0$. For $p \geq X_0$ let $\alpha_1(p), \beta_1(p)$ denote the Satake parameters of $\pi_1$ and $\alpha_2(p), \beta_2(p)$ the Satake parameters of $\pi_2$. Let $f\in\R[x,y]$ be as in Lemma \ref{lem:poly}. Then by Properties \eqref{pos} and \eqref{diag} there is $\varepsilon>0$ such that 
\[
\sum_{\substack{X_0 \leq p \leq X \\ |\lambda_{\pi_1}(p) - \lambda_{\pi_2}(p)| \leq \varepsilon} }  \log p  \leq \sum_{k = 1}^{K-1} \sum_{X_0 \leq p^k \leq X }(\log p) f\left(\alpha_1(p)^k  + \beta_1(p)^k, \alpha_2(p)^k  + \beta_2(p)^k\right)
\]
for any $K \geq 2$. 

For a representation $\Pi$ let $\Lambda_{\Pi}$ denote the Dirichlet coefficients of $(L'/L)(s, \Pi)$. Using the Kim--Sarnak bound $|\alpha_j(p)|, |\beta_j(p)| \leq  p^{7/64}$, we have
\[
\frac{L'(s, \Pi)}{L(s, \Pi)}=\sum_{i, j} a_{i, j}\sum_{X_0 \leq n \leq X }  \Lambda_{\pi_1^{\otimes i} \otimes \pi_2^{\otimes j}}(n) + O\left(X^{1/K} (X^{1/K})^{ \frac{7}{64} \cdot 8 K}\right).
\]
For $\pi \in \{\pi_1, \pi_2\}$ we can decompose
\[
\pi^{\otimes i}  = \bigboxplus_{0 \leq j \leq i/2} \Big(\binom{i}{j} - \binom{i}{j+1}\Big)   \text{sym}^{i - 2j}\pi,
\]
with the understanding $\text{sym}^0\pi = \textbf{1}$. By our assumption that the symmetric powers are cuspidal and distinct, we know by Rankin--Selberg theory  that $L(s, \text{sym}^i\pi_1 \otimes \text{sym}^j\pi_2)$ is holomorphic and non-vanishing on $\Re s \geq 1$ for $1 \leq i, j \leq 4$. By \cite{KS2} we also know that $L(s, \text{sym}^i\pi)$ is holomorphic and nonvanishing for $\Re s\geq 1$ for $1 \leq i \leq 8$. By the prime number theorem for the relevant $L$-functions, together with Property \eqref{small-degree} of Lemma \ref{lem:poly}, we have
$$\frac{L'(s, \Pi)}{L(s, \Pi)}=\left(\sum_{i, j} a_{2i, 2j} C(i) C(j)\right)(1 + o(1)) X + O(X^{9/10})$$
for $K > 40$. We conclude by Property \eqref{half} of Lemma \ref{lem:poly}.

\section{Siegel zeros and small split primes}\label{sec:zeros}

In this section we prove Theorem \ref{thm14}. The bulk of the work is contained in Theorem \ref{zero-theorem} below --- an important technical ingredient of independent interest --- which gives minimal conditions on the location of zeros under which short character sums over primes show some cancellation.

\subsection{Reduction to smooth character sums over primes} To motivate the statement of Theorem \ref{zero-theorem}, and to prepare for its proof, we first observe that
\[
\sum_{\substack{X \leq p \leq X^2 \\\chi_{-D}(p) = 1}} \frac{1}{p} = \frac{1}{2}  \sum_{ X \leq p \leq X^2} \frac{1}{p} + \frac{1}{2} \sum_{ X \leq p \leq X^2 } \frac{\chi_{-D}(p)}{p}.
\]
It therefore suffices to establish bounds on short character sums over primes, under the stated conditions of Theorem \ref{thm14}. To allow for complex analytic techniques we convert to a smooth summation. Fix some small $\varepsilon > 0$ and consider a smooth test function $W$ that is $1$ on $[1+\varepsilon, 2 - \varepsilon]$ and 0 outside $[1, 2]$, so that
\[
\sum_{ X \leq p \leq X^2 } \frac{\chi_{-D}(p)}{p} = \sum_{ X \leq p \leq X^2} \frac{\chi_{-D}(p)}{p} W\Big(\frac{\log p}{\log X}\Big) + O(\varepsilon).
\]
Let $\psi$ be given by the hypotheses of Theorem \ref{thm14}. If we can show that
\begin{equation}\label{eq:show}
 \sum_{ X \leq p \leq X^2} \frac{\chi_{-D}(p)}{p} W\Big(\frac{\log p}{\log X}\Big)=o_{\varepsilon}(1) \quad  \textrm{ for } \quad X \geq D^{1/\psi(D)}, 
 \end{equation}
 then
for every $0<\delta<1/2$ we can make $o_{\varepsilon}(1)+O(\varepsilon)\leq 1/2-\delta$ to obtain
\[
\sum_{\substack{X \leq p \leq X^2 \\ \chi_{-D}(p) = 1}} \frac{1}{p} \geq \delta\sum_{ X \leq p \leq X^2} \frac{1}{p}\]
for $X$ sufficiently large in terms of $\delta$, as required. 

Assume that $L(s, \chi_{-D})$ has no zeros for $\sigma > 1 - \Delta/\log (D(1+|t|))$, say. To establish cancellation in smooth character sums over primes of size $X$ under this hypothesis, there is a basic tension between the two parameters $X$ and $\Delta$: the smaller we take $X$ (relative to $D$) the larger we must take $\Delta$. Recall that the ratio $\log X/\log D$ can be as small as $1/\psi(D)$, which we wish to decay arbitrary slowly. The quantity $\Delta$ ultimately appears in the hypothesis \eqref{zero-freePsi} as a function of $\psi(D)$. (One can also be more precise with the range of the $t$ variable, as in \eqref{zero-freePsi} and Theorem \ref{zero-theorem} below, to make the analysis completely finitary.) 

Standard analytic techniques can set up a weak relationship between $\log X/\log D$ and $\Delta$. Indeed, by a simple application of Perron's formula, we obtain an expression of the form
\[
\int_{(2)} \frac{L'}{L}(s, \chi_{-D}) X^s \hat{W}(s)\frac{ds}{2\pi i},
\]
where the Mellin transform $\hat{W}$ decays rapidly on vertical strips. A contour shift produces a saving of about $X^{-\Delta/\log D} = \exp( - \Delta \log X/\log D)$, suggesting that $\log X$ can shrink relative to $\log D$ as soon as $\Delta \rightarrow \infty$. However, we also need to take into account the size of $\frac{L'}{L}(s, \chi_{-D})$ for which only bounds of size $O(\log D)$ are available, without making more assumptions. To compensate this, $\Delta$ must grow at least as quickly as $\log\log D$. Alternatively, we can shift the contour beyond the zero-free region, but then we need a bound for the number of zeros that we pass, and this again incurs a factor of size $O(\log D)$. In either case, the required zero-free region is of the form $\sigma\geq 1-C\log\log D/\log D$, which is of the same quality as that described in \cite[Section 1]{MV} to remove the congruence condition in Linnik's theorem.

By a much finer analysis of the interplay of zero-free regions and the size of $L$-functions near the 1-line, we bypass these difficulties. The following  theorem allows $\log X = o(\log q)$ whenever $\Delta \rightarrow \infty$ without any speed restrictions, and this is the best we can hope for.  See \cite[Chapter 9]{Mo} for a precursor. 

\begin{theorem}\label{zero-theorem} Let $q \in \N$ and $\chi$ a primitive Dirichlet character modulo $q$. 
  Let $W$ be a smooth function, compactly supported in $[1,2]$ with $|W| \leq 1$. 
    Let $0 < \eta \leq \tfrac 12$, and set $X := q^{\eta}$. Let $1 \leq \Delta = o(\log X)$, $1 \leq T \leq \log X$.  
  Suppose that $L(s, \chi)$ has no zeros in the region
  \[
  |t| \leq \frac{2 T}{\log X} \ , \ \sigma > 1 - \frac{2 \Delta}{\log X}. 
  \]
    Then for every $K$ there exists a constants $c_K > 0$ depending only on $K$ such that 
  \[
\sum_{p} \frac{\chi(p)}{p} W \Big ( \frac{\log p}{\log X} \Big ) \ll \Big ( \frac{\Delta}{\eta} \Big )^3 \cdot \big (  e^{-\Delta} + c_K \Delta T^{-K} \big ).
  \]
\end{theorem}

Let us explain the meaning of this bound in more detail. The left-hand side is supported on $X \leq p \leq X^2$, so the trivial bound for the left-hand side is $O(1)$. Hence for example the choice $\Delta = 10 \log(1 / \eta)$ and $T = 1 / \eta$ is enough to obtain non-trivial results. We think of $\eta$ as going to zero arbitrarily slowly in terms of $q$, and then $T, \Delta$ can tend to infinity arbitrarily slowly. In other words, this result gives a precise dictionary between quantitative improvements of ``no Siegel zeros'' and cancellation in character sums over primes below polynomial length in the conductor. In particular, Theorem \ref{zero-theorem} yields \eqref{eq:show} and hence Theorem \ref{thm14} provided that \eqref{zero-freePsi} holds. 

The remainder of this section is therefore dedicated to the proof of Theorem \ref{zero-theorem}. 

\subsection{Preparatory estimates}
Let $\chi$ be a primitive character modulo $q$. We generally write $\tau = |t|+2$ and $s = \sigma + it$. We also write
\[
\mathcal{L} := \log q\tau.
\]

 We start with two well-known estimates. Uniformly in $-1 \leq \sigma \leq 2$ and $r \geq 1$ we have 
\begin{equation}\label{dav}
  \frac{L'}{L}(s, \chi) = \sum_{\substack{\rho := \beta + i \gamma \\ |t - \gamma| \leq r}} \frac{1}{s - \rho} + O(\mathcal{L})
  \end{equation}
For $r \geq 0$, $t\in \R$ let $n_\chi(r; t)$ be the number of zeros of $L(s, \chi)$ in the disk $\{s \in \mathbb{C}: |1 + it - s| \leq r\}$. Then
\begin{equation}\label{fogels}
n_\chi(r; t)  \ll 1 + r \mathcal{L}.  
\end{equation}
The latter can be found, for instance, in \cite[Lemma 7]{Fogels}, but for convenience we give the quick proof. We 
notice that
 \[
  \Re \sum_{\substack{\rho := \beta + i \gamma \\ |t - \gamma| \leq 1}} \frac{1}{1 + r + it - \rho} = \sum_{\substack{\rho := \beta + i \gamma \\ |t - \gamma| \leq 1}} \frac{1 + r - \beta}{|1 + r + it - \rho|^2}.
 \]
 For any zero $\rho = \beta + i \gamma$ with $|1 + it - \rho| \leq r$ we have $|1 + r + it - \rho| \leq 2r$, and also (trivially, since $\beta \leq 1$) $1 + r - \beta \geq r$. Therefore, for any zero $\rho$ with $|1 + it - \rho| \leq r$ we obtain
  \[
  \frac{1}{4 r} = \frac{r}{(2 r)^2} \leq \frac{1 + r - \beta}{|1 + r + it - \rho|^2}.
  \]
  It follows that
\begin{displaymath}
\begin{split}
  \frac{n_\chi(r; t)}{4 r}& \leq \Re \sum_{\substack{\rho = \beta + i \gamma \\ |t - \gamma| \leq \max(1, r)}} \frac{1}{1 + r + it - \rho}
   \ll \Big | \Re \frac{L'}{L}(1 + r + it) \Big | + \mathcal{L}\ll \frac{\zeta'}{\zeta}(1+r) +  \mathcal{L} \ll \frac{1}{r} + \mathcal{L},
   \end{split}
   \end{displaymath}
  as desired. 
  
 Our first technical lemma is the following.  
  \begin{lemma} \label{le:bound}
 For $\sigma_0 > \tfrac 12$ and $t \in \mathbb{R}$  and define 
  \[
  \mathcal{C}_q(\sigma_0, t) := \Big \{ s \in \mathbb{C} : | \Im s - t| \leq \mathcal{L}^{-1}\ , \ \Re s > \sigma_0 \Big \}
  \]
If $L(s, \chi)$ has no zeros in $\mathcal{C}_q(\sigma_0, t)$, then for any $\sigma \geq \sigma_0 + \mathcal{L}^{-1}$ we have 
  \[
  \Big ( \frac{L'}{L} \Big )'(\sigma + it) \ll (1 + |\sigma - 1|\mathcal{L}) \cdot \mathcal{L}^2.
  \]
\end{lemma}
\begin{proof}
 By Cauchy's integral formula we have
  \[
  \Big ( \frac{L'}{L} \Big )' (\sigma + it) = \frac{1}{2\pi i}\oint_{|z - (\sigma + it)| = \mathcal{L}^{-1}} \frac{L'}{L}(z, \chi) \cdot \frac{dz}{(z - (\sigma + it))^2}.
  \]
  Inserting \eqref{dav}, we obtain
  \begin{equation} \label{eq:bound}
  \Big ( \frac{L'}{L} \Big )'(\sigma + it) = - \sum_{\substack{\rho = \beta + i \gamma \\ |t - \gamma| \leq 1}} \frac{1}{(\sigma+ i t - \rho)^2} + O(\mathcal{L}^2).
  \end{equation}
 Splitting into dyadic ranges, this is 
  \begin{align*} 
  \ll   \sum_{\substack{\rho = \beta + i \gamma \\ |t - \gamma| \leq 1 \\ |\sigma + it - \rho| \leq \mathcal{L}^{-1}}} \frac{1}{|\sigma + i t - \rho|^2} 
   + \sum_{k \geq 0} \sum_{\substack{\rho = \beta + i \gamma \\ |t - \gamma| \leq 1 \\ |\sigma + it - \rho| \mathcal{L} \in [2^k, 2^{k + 1}]}} \frac{1}{|\sigma + i t - \rho|^2} +\mathcal{L}^2.
  \end{align*}
  By assumption we have for $\sigma \geq \sigma_0 + \mathcal{L}^{-1}$ that 
  $
  |\sigma + it - \rho| \geq  \mathcal{L}^{-1}
  $
  for all zeros $\rho$.
  The number of zeros in the region $|\sigma  + it - \rho| \leq 2^k\mathcal{L}^{-1}$ is majorized by the number of zeros in $|1 + it - \rho| \leq 2^k\mathcal{L}^{-1}+ |\sigma - 1|$. By \eqref{fogels}, the latter is $\ll 2^k + |\sigma  - 1| \mathcal{L}$.
  Hence \eqref{eq:bound} is bounded by
 \[
  \mathcal{L}^2  ( 1 + |\sigma  - 1| \mathcal{L})   +\mathcal{L}^2  \sum_{k \geq 0} (2^k + |\sigma - 1| \mathcal{L})  2^{-2k} + \mathcal{L}^2  \ll |\sigma - 1|\mathcal{L}^3 + \mathcal{L}^2.\qedhere 
\]
\end{proof}

\begin{lemma}\label{le:main}
  Let $t\in \R$,   $10 < A = o(\log \mathcal{L})$. Suppose that $L(s, \chi)$ has no zeros in the region $\mathcal{C}_q(1 - (A + 1) \mathcal{L}^{-1}, t)$.
  Then, uniformly in $1 - A \mathcal{L}^{-1} \leq \sigma$ we have
  \[
  \log L  (\sigma + it, \chi  ) - \log L (\sigma + A \mathcal{L}^{-1} + it, \chi ) \ll A^3.
  \]
\end{lemma}

\begin{remark} By a standard argument based on the Borel-Carath\'eodory inequality (see e.g.\ \cite[p.\ 158]{Te}) one can show the slightly stronger bound 
\[ 
\log L  (\sigma + it, \chi  ) - \log L (\sigma + A \mathcal{L}^{-1} + it, \chi ) \ll A
\]
subject to the slightly stronger requirement that  $L(s, \chi)$ has no zeros in the larger region 
\[
\Re s > 1 - (A + 1) \mathcal{L}^{-1}, \quad |\Im  s - t|\leq 1.
\]
Here we prefer to keep the assumptions on zero-free regions as minimal as possible. 
\end{remark}

\begin{proof}
 For any $\sigma \geq 1 - A \mathcal{L}^{-1}$ we have by Lemma \ref{le:bound} that 
  \begin{equation*} 
  \frac{L'}{L}(\sigma + it, \chi) - \frac{L'}{L}   (1 + \mathcal{L}^{-1} +it, \chi   ) = \int_{\sigma}^{1 +  \mathcal{L}^{-1}} \Big ( \frac{L'}{L} \Big )'(x + it) \, {\rm d} x
    \ll |\sigma - 1 -  \mathcal{L}^{-1}| \cdot A  \mathcal{L}^2 \ll A^2 \mathcal{L} 
\end{equation*}
  since $|\sigma - 1 -  \mathcal{L}^{-1}| \ll A \mathcal{L}^{-1}$.
  Moreover,
  \[
  \frac{L'}{L}   (1 +  \mathcal{L}^{-1} + it, \chi   ) \ll \frac{\zeta'}{\zeta}  (1 +  \mathcal{L}^{-1}   ) \ll  \mathcal{L}, 
  \]
  so that 
  \begin{equation*} 
  \frac{L'}{L}(\sigma + it, \chi) \ll A^2 \mathcal{L}. 
  \end{equation*}
  for all $1 -A  \mathcal{L}^{-1} \leq \sigma$. Hence 
    \[
  \log L(\sigma + it, \chi) - \log L   ( \sigma +A  \mathcal{L}^{-1}+ it, \chi  ) = \int_{\sigma}^{\sigma + A  \mathcal{L}^{-1}} \frac{L'}{L}(x + it, \chi) \, {\rm d} x
  \ll \frac{A}{ \mathcal{L}} \cdot A^2  \mathcal{L} \ll A^3
  \]
  as claimed.
  \end{proof}
\subsection{Bounding short character sums over primes}

We are now ready to prove Theorem \ref{zero-theorem}. Replacing $W(x)$ with $ W(x)e^x / e^2$  (which is still compactly supported in $[1,2]$ and bounded by $1$), it suffices to bound
  \[
\mathcal{S}_0(X) =   \sum_{p} \frac{\chi(p)}{p^{1 + 1/\log X}} W \Big ( \frac{\log p}{\log X} \Big ). 
  \]
 Since trivially
  \[
 \Big| \sum_{p} \frac{\chi(p)}{p^{1 + (1 + \Delta) / \log X}} W \Big ( \frac{\log p}{\log X} \Big ) \Big| \leq e^{-\Delta}  \sum_{X \leq p \leq X^2} \frac{1}{p^{1 + 1/ \log X}} \ll e^{-\Delta}, 
  \]
  we have 
  \begin{equation*} 
\mathcal{S}_0(X) =   \sum_{p} \Big ( \frac{\chi(p)}{p^{1 + 1 / \log X}} - \frac{\chi(p)}{p^{1 + (1 + \Delta) / \log X}} \Big ) W \Big ( \frac{\log p}{\log X} \Big ) + O(e^{-\Delta}). 
  \end{equation*}
  Opening $W$ into a Fourier transform denoted by $\widecheck{W}$, the main term becomes
  \begin{align*} \int_{\mathbb{R}} \widecheck{W}(\xi) \Big ( \log L\Big (1 + \frac{1}{\log X} + \frac{2\pi i\xi}{\log X}, \chi \Big ) - \log L \Big ( 1 + \frac{\Delta + 1}{\log X} + \frac{2\pi i \xi}{\log X}, \chi \Big ) \Big ) d \xi + O \Big (\frac{1}{X} \Big ).
  \end{align*}
  Let
  \[
  H(s) = \log L \Big ( 1 + \frac{1}{\log X} + s, \chi \Big ) - \log L \Big ( 1 + \frac{\Delta + 1}{\log X} + s, \chi \Big ),
  \]
  so that after a change of variables we have 
  \[
  \mathcal{S}_0(X) =  \log X \int_{(0)} \widecheck{W}\Big(\frac{s \log X }{2\pi i}\Big) H(s) \frac{{\rm d}s}{2\pi i} + O(e^{-\Delta}).
  \]
  For each $t$ we define $A_t \geq 1$ by 
  \begin{equation*}
  \frac{A_t}{\log (q \tau)} = \frac{\Delta}{\log X}.
  \end{equation*}
  Our assumptions imply that $L(s, \chi)$ has no zeros in $\mathcal{C}_q (1 - (A_t + 1)\mathcal{L}^{-1}, t)$, uniformly in $|t| \leq T/\log X$, since $T \geq 1$ and $\Delta/\log X \geq \mathcal{L}^{-1}$. Hence Lemma \ref{le:main} gives
  \begin{equation} \label{eq:fbound}
  H(\sigma + it) \ll A_t^3 \ll \Big ( \frac{\Delta}{\eta} \Big )^3 \quad \text{ uniformly in }  \quad - \frac{\Delta}{\log X} \leq \sigma \text{ and } |t| \leq \frac{T}{\log X}.
  \end{equation}
 Since $W$ is compactly supported, the function $\widecheck{W}(\xi)$ is entire in $\xi$ and on fixed horizontal lines satisfies the bound
 $$\widecheck{W}(\xi) \ll_K \frac{e^{-2 \pi \Im \xi}}{(1 + |\xi|)^{K}}$$
 for any $K > 0$. 
 
  We now deform the contour into the following three regions as follows:
  \begin{align}
   \label{eq:first} |t| \geq \frac{T}{\log X} \ & , \ \sigma = 0, \\
   \label{eq:second} t = \pm \frac{T}{\log X} \ & , \ - \frac{\Delta}{\log X}  \leq \sigma \leq 0, \\
   \label{eq:third} |t| \leq \frac{T}{\log X} \ & , \ \sigma = - \frac{\Delta}{\log X}.
  \end{align}
For $\sigma \geq 0$ we use the trivial bound
\[
  |H(\sigma + it)| \leq  \sum_{p} \frac{1}{p} \cdot \Big ( \frac{1}{p^{1 / \log X}} - \frac{1}{p^{(\Delta + 1) / \log X}} \Big ) + O(1) \ll \log (2+\Delta).
\]
Therefore the contribution of the region \eqref{eq:first} to the main term of $\mathcal{S}_0(X)$ is
\[
\ll_K  \log X \cdot \log \Delta \int_{T/\log X}^{\infty} (1 + t \log X)^{-K} \ll_{K} \log (2+\Delta) \cdot T^{-(K-1)}
\]
for any choice of $K > 10$. We use \eqref{eq:fbound} in the remaining regions \eqref{eq:second} and \eqref{eq:third}. The region \eqref{eq:second} contributes
\[
  \ll_K  \Big(\frac{\Delta }{ \eta}\Big)^3 ( \log X)  T^{-K} \int_{-\Delta/\log X}^0 e^{ \sigma} d\sigma \ll \Big(\frac{\Delta }{ \eta}\Big)^3 \Delta  T^{-K}. 
\]  
Finally, the region \eqref{eq:third} contributes
\[
\log X \cdot \Big ( \frac{\Delta}{\eta} \Big )^3 \cdot \int_{|t| \leq T / \log X} \frac{e^{-\Delta}}{1 + |t \log X|^{10}} \, {\rm d} t   \ll  \Big  ( \frac{\Delta}{\eta} \Big )^3 e^{-\Delta}.
\]
Combining all estimates completes the proof of Theorem \ref{zero-theorem}.

\section{Pigeonhole}\label{finalsec}

We now give the proof of Corollary \ref{lem61}. We show the bound $S_D(\pi_1,\pi_2)>\kappa \log\psi(D)$ if $B$ is chosen sufficiently large in the definition of $S_D(\pi_1,\pi_2)$. The other claim is proved similarly. 

Let $\varepsilon>0$ be as in Theorem \ref{thm15} and $X > D^{1/\psi(D)}$ as in Theorem \ref{thm14}. We have 
\[
\sum_{\substack{X \leq p \leq X^2 \\ \chi_{-D}(p) = 1\\  |\lambda_{\pi_1}(p)|, |\lambda_{\pi_2}(p)|\leq B}} \frac{(\lambda_{\pi_1}(p) - \lambda_{\pi_2}(p))^2}{p} >  \varepsilon^2 \sum_{\substack{X \leq p \leq X^2 \\ \chi_{-D}(p) = 1\\  |\lambda_{\pi_1}(p) - \lambda_{\pi_2}(p)| \geq \varepsilon\\ |\lambda_{\pi_1}(p)|, |\lambda_{\pi_2}(p)|\leq B}} \frac{1}{p},
\]
Of the three conditions on $p$, the density for which the first occurs is given by Theorem \ref{thm14}, with, say, $\delta=0.4999$. Theorem \ref{thm15} implies that $$\sum_{\substack{X \leq p \leq X^2 \\ |\lambda_{\pi_1}(p) - \lambda_{\pi_2}(p)| \geq \varepsilon}} \frac{1}{p} \geq 0.501\sum_{X \leq p \leq X^2} \frac{1}{p}.$$
Indeed, this follows by splitting the range $X \leq p \leq X^2$ into pieces of the form $[Y, \gamma Y]$ with $\gamma > 1$ sufficiently small and estimating $1/p \geq 1/(\gamma Y)$ for $p \in [Y, \gamma Y]$. Finally, Rankin--Selberg theory governs the last condition, since
\[
\sum_{\substack{X \leq p \leq X^2\\ \max(|\lambda_{\pi_1}(p)|, |\lambda_{\pi_2}(p)| > B}}\frac{1}{p} \leq  \frac{1}{B^2} \sum_{X \leq p \leq X^2} \frac{\lambda_{\pi_1}(p)^2 + \lambda_{\pi_2}(p)^2 }{p}  =   \frac{2 + o(1)}{B^2} \sum_{X \leq p \leq X^2} \frac{1}{p}.
\]
Using inclusion-exclusion we obtain
\[
\sum_{\substack{X \leq p \leq X^2 \\ p \textrm{  split }\\  |\lambda_{\pi_1}(p)|, |\lambda_{\pi_2}(p)|\leq B}} \frac{(\lambda_{\pi_1}(p) - \lambda_{\pi_2}(p))^2}{p} \geq \varepsilon^2 \Big(0.4999 + 0.501 + \frac{B^2 - 2 + o(1)}{B^2} - 2\Big)  \sum_{X \leq p \leq X^2}\frac{1}{p}.
\]
In particular, for $B$ sufficiently large there is $\kappa>0$ such that 
\[
\sum_{\substack{X \leq p \leq X^2\\ p \textrm{ is split }\\  |\lambda_{\pi_1}(p)|, |\lambda_{\pi_2}(p)|\leq B}} \frac{(\lambda_{\pi_1}(p) - \lambda_{\pi_2}(p))^2}{p} \geq  \kappa \sum_{X \leq p \leq X^2} \frac{1}{p},
\]
valid uniformly in $X$ in the range $[D^{1/\psi(D)},D^c]$. Summing over all intervals of the shape $[Z, Z^2]$ for $Z = X, X^2, X^4, \ldots, D^{c/2}$ we obtain the desired claim.

\subsection*{Acknowledgements} We would like to thank Steve Lester, Philippe Michel, Asbj\o rn Nordentoft, and Radu Toma for helpful conversations.

\end{document}